\pdfoutput=1
\documentclass[11pt,a4paper,reqno]{amsart}
\usepackage[british]{babel}

\usepackage{kpfonts}

\usepackage[T1]{fontenc}

\usepackage{setspace}
\usepackage{relsize}
\usepackage{amssymb}
\usepackage{amsmath}
\usepackage{amsthm}
\usepackage{enumitem}

\usepackage[colorlinks,citecolor=blue,urlcolor=black,linkcolor=blue,linktocpage]{hyperref}
\usepackage{mathtools}

\makeatletter
\@namedef{subjclassname@2020}{\textup{2020} Mathematics Subject Classification}
\makeatother

\newtheorem{thm}{Theorem}[section]
\newtheorem{cor}[thm]{Corollary}
\newtheorem{lem}[thm]{Lemma}

\newtheorem{prop}[thm]{Proposition}

\newtheorem{problem}[thm]{Problem}

\theoremstyle{definition}
\newtheorem{exmpl}[thm]{Example}

\newtheorem{definition}[thm]{Definition}

\newtheorem{remark}[thm]{Remark}

\renewcommand{\epsilon}{\varepsilon}
\renewcommand{\phi}{\varphi}
\newcommand{\defeq}{\mathrel{\mathop:}=}                         
\newcommand{\eqdef}{\mathrel{\mathopen={\mathclose:}}}          

\DeclareMathOperator{\I}{I}
\DeclareMathOperator{\Diff}{D}
\DeclareMathOperator{\Ed}{E}
\DeclareMathOperator{\Qalg}{Q}
\DeclareMathOperator{\B}{B}
\DeclareMathOperator{\GL}{GL}
\DeclareMathOperator{\T}{T}
\DeclareMathOperator{\lat}{L}
\DeclareMathOperator{\Mean}{M}
\DeclareMathOperator{\Samuel}{S}
\DeclareMathOperator{\RUCB}{RUCB}
\DeclareMathOperator{\UEB}{UEB}
\DeclareMathOperator{\RUEB}{RUEB}
\DeclareMathOperator{\RUC}{RUC}
\DeclareMathOperator{\Cont}{C}
\DeclareMathOperator{\UCB}{UCB}
\DeclareMathOperator{\Lip}{Lip}
\DeclareMathOperator{\E}{\mathbb{E}}
\DeclareMathOperator{\N}{\mathbb{N}}
\DeclareMathOperator{\Z}{\mathbb{Z}}

\DeclareMathOperator{\R}{\mathbb{R}}

\DeclareMathOperator{\id}{id}

\DeclareMathOperator{\diam}{diam}
\DeclareMathOperator{\Ker}{Ker}
\DeclareMathOperator{\Sub}{Sub}
\DeclareMathOperator{\Pfin}{\mathscr{P}_{fin}}
\DeclareMathOperator{\cent}{Z}
\DeclareMathOperator{\Lev}{Lev}
\DeclareMathOperator{\Irr}{Irr}
\DeclareMathOperator{\Max}{Max}
\DeclareMathOperator{\Nest}{\mathscr{N}\!}
\DeclareMathOperator{\Nestmax}{\mathscr{N}_{max}}
\DeclareMathOperator{\Flag}{\mathscr{F}\!}
\DeclareMathOperator{\Flagmax}{\mathscr{F}_{max}}
\DeclareMathOperator{\Prob}{Prob}

\begin{document}

\setlist{noitemsep}

\author{Friedrich Martin Schneider}
\address{F.M.S., Institute of Discrete Mathematics and Algebra, TU Bergakademie Freiberg, 09596 Freiberg, Germany}
\email{martin.schneider@math.tu-freiberg.de}

\title[Concentration of invariant means]{Concentration of invariant means and dynamics of chain stabilizers in continuous geometries}
\date{\today}

\begin{abstract} 
	We prove a concentration inequality for invariant means on topological groups, namely for such adapted to a chain of amenable topological subgroups. The result is based on an application of Azuma's martingale inequality and provides a method for establishing extreme amenability. Building on this technique, we exhibit new examples of extremely amenable groups arising from von Neumann's continuous geometries. Along the way, we also answer a question by Pestov on dynamical concentration in direct products of amenable topological groups.
\end{abstract}



\maketitle


\tableofcontents

\section{Introduction}

Concentration of measure, a phenomenon originally extracted by Vitali Milman from earlier work of Paul L\'evy~\cite{levy} and brought to the asymptotic geometry of Banach spaces~\cite{Milman71,MilmanSchechtman}, is known to have striking applications in topological dynamics. In the realm of topological groups, such dynamical consequences are frequently linked with the concept of extreme amenability: a topological group $G$ is said to be \emph{extremely amenable} if every continuous action of $G$ on a non-void compact Hausdorff space admits a fixed point. In Milman's seminal joint work with Misha Gromov~\cite{GromovMilman}, concentration of measure was identified as a source of extreme amenability: if a topological group $G$ contains a directed family of compact subgroups whose union is dense in $G$ and whose normalized Haar measures concentrate in~$G$, then~$G$ is extremely amenable. Following the examples of extremely amenable groups discovered in~\cite{GromovMilman}, this method has since found numerous further applications~\cite{glasner,GiordanoPestov,Pestov07,DowerkThom,CarderiThom} and extensions~\cite{Pestov02,FarahSolecki,PestovWhirly,PestovSchneider,Schneider,SchneiderSolecki}.

The purpose of this work is twofold. First, expanding the scope of dynamical ramifications of concentration techniques, we aim to initiate the study of \emph{concentration of invariant means} on groups. This includes abstract sources of concentration (Theorem~\ref{theorem:main.concentration} and Proposition~\ref{proposition:amenable.folding}), dynamical consequences (Theorem~\ref{theorem:extreme.amenability} and Corollary~\ref{corollary:extreme.amenability.1}), as well as obstructions and counterexamples (Theorem~\ref{theorem:pestov} and Corollary~\ref{corollary:pestov.2}). Second, building on the resulting techniques, we exhibit new examples of extremely amenable topological groups in the context of John von Neumann's continuous geometries (Theorem~\ref{theorem:stable.groups} and Corollaries~\ref{corollary:prefinal}, \ref{corollary:final}, \ref{corollary:levitzki.extremely.amenable}).

At the core of the present work is a new concentration inequality concerning certain measures on the \emph{Samuel compactification} $\Samuel (G)$ of an arbitrary topological group $G$, i.e., the Gelfand space of the algebra of all right-uniformly continuous bounded functions on $G$. Due to the Riesz--Markov--Kakutani representation theorem, the space $\Mean (G)$ of all means (that is, positive linear forms) on the same algebra admits a natural affine homeomorphism $\Mean (G) \to \Prob(\Samuel (G)), \, \mu \mapsto \hat{\mu}$ onto the space $\Prob(\Samuel (G))$ of all regular Borel probability measures on $\Samuel (G)$ (see Section~\ref{section:concentration}). Furthermore, this map constitutes a monoid isomorphism with regard to the respective convolution products on $\Mean (G)$ and $\Prob(\Samuel (G))$ inherited from $G$. Now, with any chain of subgroups $\{ e \} = G_{0} \leq \ldots \leq G_{n} = G$ and any bounded, continuous, right-invariant pseudo-metric on $G$, one may associate the quantity \begin{displaymath}
	\ell(G_{0},\ldots,G_{n};d) \, \defeq \, \left(\sum\nolimits_{i=0}^{n-1} \left(\sup\nolimits_{g \in G} \diam \left(G_{i+1}/G_{i},d^{g}_{i}\right) \right)^{2} \right)^{1/2},
\end{displaymath} where \begin{displaymath}
	d_{i}^{g} \colon \, G_{i+1}/G_{i} \times G_{i+1}/G_{i} \, \longrightarrow \, \R_{\geq 0}, \quad (xG_{i},yG_{i}) \, \longmapsto \, \inf\nolimits_{h \in G_{i}} d(gx,gyh) .
\end{displaymath} The following result, our first main contribution, generalizes a concentration inequality for normalized Haar measures on compact groups due to Milman--Schechtman~\cite[I, Theorem~7.12(i)]{MilmanSchechtman} (see Corollary~\ref{corollary:precompact}) to invariant means on amenable topological groups.

\begin{thm}[Theorem~\ref{theorem:main.concentration}]\label{theorem:main} Consider a chain of amenable topological groups\footnote{each equipped with the relative topology inherited from $G$} \begin{displaymath}
	\{ e \} \, = \, G_{0} \, \leq \, \ldots \, \leq \, G_{n} \, = \, G .
\end{displaymath} For each $i \in \{ 1,\ldots,n\}$ pick a $G_{i}$-left-invariant mean $\nu_{i} \in \Mean (G_{i})$, and consider \begin{displaymath}
	\mu \, \defeq \, \nu_{n}\cdots \nu_{1} \, \in \, \Mean (G) .
\end{displaymath} Moreover, let $d$ be a bounded, continuous, right-invariant pseudo-metric on $G$. Then, for every $f \in \Lip_{1}(G,d)$ and every $\epsilon \in \R_{>0}$, \begin{displaymath}
	\hat{\mu}\left(\left\{ \xi \in \Samuel (G) \left\vert \, \left\lvert \xi(f) - \mu(f) \right\rvert \geq \epsilon \right\}\right) \right. \! \, \leq \, 2 \exp \left( -\tfrac{\epsilon^{2}}{2 \ell(G_{0},\ldots,G_{n};d)^{2}} \right) .
\end{displaymath} \end{thm}

Among the most important ingredients in the proof of Theorem~\ref{theorem:main}, next to Kazuoki Azuma's martingale inequality~\cite{azuma}, is a description of certain conditional expectations on the Samuel compactification of a topological group in terms of convolution operators on the underlying algebra of right-uniformly continuous functions (Lemma~\ref{lemma:conditional.expectation}).

Motivated by Theorem~\ref{theorem:main}, we introduce the \emph{amenable length functional} (Definition~\ref{definition:amenable.length}), which maps every bounded, continuous, right-invariant pseudo-metric on an amenable topological group to some non-negative real number. Using Theorem~\ref{theorem:main}, we prove that vanishing amenable length implies extreme amenability (Theorem~\ref{theorem:extreme.amenability}  and Corollaries~\ref{corollary:extreme.amenability.1}--\ref{corollary:extreme.amenability.2}). This result has some non-trivial ramifications, such as the amplification technique established by the following corollary.

\begin{cor}[Proposition~\ref{proposition:amenable.folding} $+$ Corollary~\ref{corollary:extreme.amenability.2}]\label{corollary:intro} Let $G$ be a topological group and $d$ be a right-invariant metric generating the topology of $G$. Suppose that there exist a dense subset $D \subseteq [0,1]$ with $\{0,1\} \subseteq D$ as well as a family of continuous endomorphisms $\phi_{t} \colon G \to G$ $(t \in D)$ such that \begin{itemize}
	\item[$(1)$] $\phi_{0}(G) = \{ e \}$ and $\phi_{1}(G) = G$,
	\item[$(2)$] if $s,t \in D$ and $s\leq t$, then $\phi_{s}(G) \subseteq \phi_{t}(G)$, and
	\item[$(3)$] there exists $C \in \R_{\geq 0}$ such that \begin{displaymath}
			\qquad \forall s,t \in D \ \forall g,h \in G \colon \quad d(g\phi_{s}(h),g\phi_{t}(h)) \, \leq \, C\cdot \vert s-t \vert .
		\end{displaymath}
\end{itemize} If $G$ is amenable, then $G$ is extremely amenable. \end{cor}

For instance, this technique may be used to recover the extreme amenability of the group of measurable maps with values in an amenable topological group (Example~\ref{example:pestov.schneider}). New examples of extremely amenable groups resulting from Corollary~\ref{corollary:intro} are to be found in the context of von Neumann's continuous geometries, as discussed below.

Before exploring further applications, one may wonder about potential improvements and extensions of Theorem~\ref{theorem:main}. Indeed, the present work has been crucially inspired by the following problem by Vladimir Pestov, which was communicated to the author in 2016 and published in~\cite[Question~4.5]{PestovSchneider}.

\begin{problem}[Pestov]\label{problem} For each $n \in \N_{>0}$, let $d_{n}$ be a compatible right-invariant metric on an amenable second-countable group $G_{n}$ such that $\diam (G_{n},d_{n}) \leq 1$. For every $n \in \N_{>0}$, consider the topological group $H_{n} \defeq \prod_{i=1}^{n} G_{i}$ equipped with the compatible right-invariant metric \begin{displaymath}
	d'_{n} \colon \, H_{n} \times H_{n} \, \longrightarrow \, [0,1], \quad (x,y) \, \longmapsto \, \tfrac{1}{n} \sum\nolimits_{i=1}^{n} d_{i}(x_{i},y_{i}) .
\end{displaymath} For each $n \in \N_{>0}$, fix an $H_{n}$-left-invariant mean $\mu_{n} \in \Mean (H_{n})$. Is it true that, for every sequence $f_{n} \in \Lip_{1}(H_{n},d'_{n})$ $(n \in \N_{>0})$, \begin{displaymath}
	\mu_{n} \left( (f_{n} - \mu_{n}(f_{n}))^{2} \right) \, \xrightarrow{n\to\infty} \, 0 \ ?
\end{displaymath} \end{problem}

Problem~\ref{problem} suggests a substantial improvement of our Theorem~\ref{theorem:main} in the particular case of direct products of topological groups equipped with normalized Hamming ($\ell^{1}$-type sum) distances: an affirmative answer to Pestov's question would assert concentration for any arbitrary sequence of invariant means on such topological groups rather than just convolution products. However, as part of the present work, Problem~\ref{problem} is resolved in the negative (Corollary~\ref{corollary:pestov.2}).

The source of concentration of invariant means explored in this work is von Neumann's continuous geometry~\cite{VonNeumannBook}, a continuous version of projective geometry. A \emph{continuous geometry} is a complete, complemented, modular lattice whose operations furthermore satisfy a certain continuity property (see Section~\ref{section:continuous.geometries}). Examples of such objects include all orthocomplemented complete modular lattices~\cite{kaplansky}, such as the projection lattice of any finite von Neumann algebra (see, e.g.,~\cite[Propositions~6.3, 6.14]{redei}), as well as the completion of a natural inductive limit of finite-dimensional projective geometries over a fixed division algebra with respect to a normalized dimension distance~\cite{NeumannExamples} (see also~\cite[Section~VIII.10]{BirkhoffBook}). Generalizing the classical Veblen--Young theorem, von Neumann~\cite{VonNeumannBook} proved the following deep and technically outstanding \emph{coordinatization theorem}: for every complemented modular lattice $L$ with an order greater than or equal to four there exists an (up to isomorphism unique) regular ring $R$ such that $L$ is isomorphic to the lattice $\lat(R)$ of principal right ideals of $R$. In case of an irreducible continuous geometry $L$, the corresponding, necessarily irreducible, regular ring $R$ admits a unique rank function $\rho_{R} \colon R \to [0,1]$, and the induced metric $d_{R} \colon R \times R \to [0,1], \, (a,b) \mapsto \rho_{R}(a-b)$ is complete. For such an irreducible \emph{continuous ring} $R$, we apply the concentration methods outlined above to study the dynamics of topological subgroups of the unit group $\GL(R)$ carrying the topology generated by the metric $d_{R}$.

A cornerstone in our treatment of groups of invertible elements in a non-discrete irreducible, continuous ring $R$ is the \emph{continuous triangularization theorem} (Theorem~\ref{theorem:invariant.flags}), which asserts that every subring $S \leq R$ that contains and is finite-dimensional over the center $\cent(R)$ must stabilize some maximal chain in $\lat(R)$ under multiplication from the left. In particular, this applies to every single element of $R$ algebraic over $\cent(R)$. Since the set of $\cent(R)$-algebraic elements is dense in $R$ by another fundamental result of von Neumann~\cite{VonNeumann37,Halperin62} (see also Section~\ref{section:algebraic.elements}), this suggests to examine maximal chains in $\lat(R)$ and their stabilizer subgroups in $\GL(R)$ in greater detail. To this end, we develop a representation of chains in $\lat(R)$ by \emph{nests} in $R$, i.e., chains of idempotent elements of $R$ (Proposition~\ref{proposition:flags.vs.nests} and Theorem~\ref{theorem:maximal.nests}). Each such nest $E$ in $R$ gives rise to a \emph{stabilizer ring} $R_{E} \leq R$ (Definition~\ref{definition:stabilizer}) and, by a mechanism inspired by the classical Jordan--Chevalley decomposition, induces a closure operator \begin{displaymath}
	\GL(R_{E}) \geq G \, \longmapsto \, [G]_{E} \leq \GL(R_{E}) 
\end{displaymath} on the subgroup lattice of the unit group of $R_{E}$ (Definition~\ref{definition:jordan}). Using Corollary~\ref{corollary:intro}, this operator will be identified as a source of extremely amenable topological groups as follows.

\begin{thm}[Theorem~\ref{theorem:stable.groups}]\label{theorem:amenable.subgroups} Let $E$ be a maximal nest in a non-discrete irreducible, continuous ring $R$. If a topological subgroup $G \leq \GL(R_{E})$ is amenable, then $[G]_{E}$ is extremely amenable. \end{thm}

For instance, this result directly entails the extreme amenability of the topological groups naturally arising from the Levitzki radicals of the considered stabilizer rings (Corollary~\ref{corollary:levitzki.extremely.amenable}). Combining Theorem~\ref{theorem:amenable.subgroups} with our continuous triangularization theorem (Theorem~\ref{theorem:invariant.flags}), we moreover deduce the following corollary, which---in view of the density of algebraic elements in the unit group of any non-discrete irreducible, continuous ring (Proposition~\ref{proposition:algebraic.units})---illustrates the abundance of extremely amenable groups in the context of continuous geometries.

\begin{cor}[Corollary~\ref{corollary:final}] Let $R$ be a non-discrete irreducible, continuous ring. Every element of $\GL(R)$ algebraic over $\cent(R)$ is contained in a locally solvable, extremely amenable topological subgroup of $\GL(R)$. \end{cor}

In order to point out a notable global consequence of the results described above, let us isolate the following dynamical property: a topological group $G$ will be called \emph{inert} if, for every continuous action of $G$ on a non-void compact Hausdorff space $X$, each element of $G$ admits a fixed point in $X$, i.e., \begin{displaymath}
	\forall g \in G \ \exists x \in X \colon \qquad gx \, = \, x .
\end{displaymath} Inertness is a reinforced form of negation of the existence of a free continuous action on a non-empty compact Hausdorff space. For instance, if a topological group $G$ is inert, then every continuous homomorphism from $G$ to any locally compact group must be constant, since the latter topological group---by work of William Veech~\cite[Theorem~2.2.1]{veech}---acts freely and continuously on its own Samuel compactification. Of course, every extremely amenable topological group is inert.

\begin{cor}[Corollary~\ref{corollary:inert.final}]\label{corollary:inert} Let $R$ be a non-discrete irreducible, continuous ring. Then the union of its extremely amenable topological subgroups is dense in $\GL(R)$. In particular, the topological group $\GL(R)$ is inert. \end{cor}

In general, the conclusion of Corollary~\ref{corollary:inert} cannot be strengthened to extreme amenability of the unit group: in fact, if a (discrete) group $G$ is not inner amenable in the sense of~\cite{effros} (for instance, if $G$ is a non-abelian free group), then the ring of operators affiliated with the group von Neumann algebra of $G$ constitutes a non-discrete irreducible, continuous ring whose unit group is non-amenable with respect to the topology generated by the corresponding rank metric, as proved by the present author in~\cite{SchneiderFactor}.

This article is organized as follows. In Section~\ref{section:concentration} we put together some general background on concentration of means in uniform spaces, and in Section~\ref{section:topological.groups} we set up the relevant terminology and notation concerning topological groups, (extreme) amenability, and convolution algebras. The subsequent Section~\ref{section:concentration.of.invariant.means} contains our main results about concentration of invariant means in topological groups. Pestov's Problem~\ref{problem} will be solved in Section~\ref{section:pestov}. In Section~\ref{section:continuous.geometries} we provide some background on (irreducible) continuous geometries and characterize their maximal chains. Our Section~\ref{section:continuous.rings} contains some material on continuous rings and their rank functions, as well as a characterization of the maximal nests in an irreducible continuous ring. In Section~\ref{section:algebraic.elements} we discuss von Neumann's density theorem for algebraic elements in non-discrete irreducible continuous rings, along with a slight variation thereof. The aforementioned continuous triangularization theorem is stated and proved in Section~\ref{section:triangularization}. The corresponding concept of nest envelopes is developed in Section~\ref{section:nest.envelopes}. The final Section~\ref{section:stability} is devoted to dynamical consequences of our results, including new examples of extremely amenable topological groups. In the Appendix~\ref{section:levitzki} we compile some relevant facts about nilpotency and the Levitzki radical in unital rings.

\section{Concentration in uniform spaces}\label{section:concentration}

The objective of this preliminary section is to put together some material on concentration of means in uniform spaces, which will be fundamental to our subsequent study of topological groups. For this purpose, we will suitably adapt well-established accounts on concentration of measure in metric spaces~\cite{MilmanSchechtman,ledoux,Pestov02a} and uniform spaces~\cite{Pestov02,PestovBook}.

Before getting to uniform spaces, let us briefly clarify some general terminology and notation used throughout the manuscript. Henceforth, if $E$ is a Banach space, then we let $E^{\ast}$ denote the Banach space of all bounded linear forms on $E$, and $\B(E)$ denote the unital Banach algebra of all bounded linear endomorphisms of $E$. Now, let $X$ be a set. Then $\mathscr{P}(X)$ denotes the power set of $X$, and $\Pfin (X)$ denotes the set of all finite subsets of $X$. Furthermore, we let $\ell^{\infty}(X)$ denote the unital Banach algebra of all bounded real-valued functions on $X$, equipped with the supremum norm \begin{displaymath}
	\Vert f \Vert_{\infty} \, \defeq \, \sup \{ \vert f(x) \vert \mid x \in X\} \qquad \left( f \in \ell^{\infty}(X) \right) .
\end{displaymath} The \emph{indicator function} of a subset $B \subseteq X$ will be denoted by \begin{displaymath}
	\chi_{B} \colon \, X \, \longrightarrow \, \{ 0,1 \}, \quad x \, \longmapsto \, \begin{cases}
																							\, 1 & \text{if } x \in B, \\
																							\, 0 & \text{otherwise.}
																						\end{cases}
\end{displaymath} Suppose that $d$ is a pseudo-metric on $X$. For every $r \in \R_{>0}$ and every $x \in X$, we let $\B_{d}(x,r) \defeq \{ y \in X \mid d(x,y)< r\}$. If $k \in \R_{\geq 0}$, then we define \begin{displaymath}
	\left. \Lip_{k}(X,d;S) \, \defeq \, \left\{ f \in S^{X} \, \right\vert \forall x,y \in X \colon \, \vert f(x) - f(y) \vert \leq k d(x,y) \right\}
\end{displaymath} for every subset $S \subseteq \R$, as well as \begin{displaymath}
	\Lip_{k}(X,d) \, \defeq \, \Lip_{k}(X,d; \R), \qquad \Lip_{k}^{\infty}(X,d) \, \defeq \, \Lip_{k}(X,d) \cap \ell^{\infty}(X) .
\end{displaymath} Given any subset $A \subseteq X$, let us define \begin{displaymath}
	\diam (A,d) \, \defeq \, \sup \{ d(x,y) \mid x,y \in A \} \, \in \, [0,\infty] .
\end{displaymath} For any subset $B \subseteq \R $, we put $\diam B \defeq \sup \{ \vert x-y \vert \mid x,y \in B \} \in [0,\infty]$.

\begin{remark}\label{remark:lipschitz.approximation} Let $(X,d)$ be a pseudo-metric space, let $\ell, \epsilon \in \R_{\geq 0}$, let $s,t \in \R$ with $s \leq t$, and let $f \colon X \to [s,t]$. Then the following holds: \begin{equation}\tag{$\ast$}\label{approximation}
	\begin{split}
		&\bigl( \forall x,y \in X \colon \, \vert f(x) - f(y) \vert \leq \ell d(x,y) + \epsilon \bigr) \\
		& \qquad \qquad \qquad \qquad \Longrightarrow \ \left( \exists g \in \Lip_{\ell}(X,d; [s,t]) \colon \, \Vert f-g \Vert_{\infty} \leq \epsilon \right) .
	\end{split}
\end{equation} Indeed, if $f$ satisfies the premise of~\eqref{approximation}, then \begin{displaymath}
	g \colon \, X \, \longrightarrow \, [s,t], \quad x \, \longmapsto \, \left( \inf\nolimits_{y \in X} f(y) + \ell d(x,y) \right) \wedge t
\end{displaymath} will verify the conclusion of~\eqref{approximation}. A proof of this very well-known fact is to be found, for instance, in~\cite[Lemma~5.2]{SchneiderThomRandomWalks}. \end{remark}

Furthermore, if $X$ is a compact Hausdorff space, then $\Cont(X)$ denotes the set of all continuous real-valued functions, which is well known to form a closed unital subalgebra of $\ell^{\infty}(X)$.  

We proceed to some basic terminology concerning uniform spaces. The reader is referred to~\cite{PachlBook} for a more comprehensive account on uniformities and associated spaces of functions and measures. A \emph{uniformity} on a set $X$ is a filter $\mathscr{E}$ on the set $X \times X$ such that \begin{itemize}
	\item[---] $\{ (x,x) \mid x \in X\} \subseteq E$ for every $E \in \mathscr{E}$,
	\item[---] $\{ (y,x) \mid (x,y) \in E \} \in \mathscr{E}$ for every $E \in \mathscr{E}$, and 
	\item[---] for every $E_{0} \in \mathscr{E}$ there exists $E_{1} \in \mathscr{E}$ such that \begin{displaymath}
				\qquad \{ (x,y) \in X \times X \mid \exists z \in X \colon \, (x,z), (z,y) \in E_{1} \} \, \subseteq \, E_{0} .
			\end{displaymath}
\end{itemize} The elements of a uniformity are usually referred to as \emph{entourages}. A \emph{uniform space} is a set $X$ together with a uniformity on $X$. If $\mathscr{E}$ is a uniformity on a set $X$ and $Y \subseteq X$, then ${\mathscr{E}\!\!\upharpoonright_{Y}} \defeq \{ E \cap (Y \times Y) \mid E \in \mathscr{E} \}$ constitutes a uniformity on $Y$ and is called the \emph{relative uniformity} (or \emph{subspace uniformity}) induced by $\mathscr{E}$ on $Y$. A \emph{uniform subspace} of a uniform space $X$ is a subset of $X$ equipped with the induced relative uniformity. As usual, any pseudo-metric space $(X,d)$ will be viewed as a uniform space, carrying the induced uniformity \begin{displaymath}
	\{ E \subseteq X \times X \mid \exists r \in \R_{>0} \forall x,y \in X \colon \, d(x,y) < r \Longrightarrow (x,y) \in E \} .
\end{displaymath} This particularly applies to $\R$ with respect to the Euclidean metric.

Let $X$ and $Y$ be two uniform spaces. A map $h \colon X \to Y$ is called \emph{uniformly continuous} if, for every entourage $F$ of $Y$, there exists an entourage $E$ of $X$ such that $\{ (h(x),h(y)) \mid (x,y) \in E\} \subseteq F$. Furthermore, a set $H \subseteq Y^{X}$ is said to be \emph{uniformly equicontinuous} if, for every entourage $F$ of $Y$, there exists an entourage $E$ of $X$ such that $\{ (h(x),h(y)) \mid (x,y) \in E, \, h \in H \} \subseteq F$.

Now, let $X$ be a uniform space. We let $\UCB(X)$ denote the commutative unital real Banach algebra of all bounded uniformly continuous real-valued functions on $X$. A \emph{mean} on $X$ is a mean on $\UCB(X)$, i.e., a (necessarily continuous) positive unital linear form on $\mathrm{UCB}(X)$. The collection $\Mean (X)$ of all means on $X$ constitutes a weak-$\ast$ closed subset of the closed unit ball of $\UCB(X)^{\ast}$, whence $\Mean (X)$ is compact with respect to the relative weak-$\ast$ topology (see, e.g., \cite[2.1, Theorem~1.8(i), p.~68]{AnalysisOnSemigroups}). The set $\Samuel(X)$ of all (necessarily positive and linear) unital ring homomorphisms from $\UCB(X)$ to~$\mathbb{R}$ is a weak-$\ast$ closed (thus compact) subspace of~$\Mean (X)$ and is called the \emph{Samuel compactification} of~$X$. The map $\eta_{X} \colon X \to \Samuel (X)$ defined by \begin{displaymath}
	\eta_{X}(x)(f) \, \defeq \, f(x) \qquad (x \in X, \, f \in \UCB(X))
\end{displaymath} is uniformly continuous, and its the image $\eta_{X}(X)$ is a dense subset of $\Samuel(X)$ (see, e.g., \cite[2.1, Theorem~1.8(iv), p.~68]{AnalysisOnSemigroups}). The induced map \begin{displaymath}
	\Cont (\Samuel (X)) \, \longrightarrow \, \UCB(X), \quad f \, \longmapsto \, f \circ \eta_{X}
\end{displaymath} constitutes an isometric isomorphism of unital Banach algebras, the inverse of which is given by \begin{displaymath}
	\UCB(X) \, \longrightarrow \, \Cont (\Samuel (X)), \quad f \, \longmapsto \, \overline{f} , 
\end{displaymath} where \begin{displaymath}
	\overline{f}(\xi) \, \defeq \, \xi(f) \qquad (f \in \UCB(X), \, \xi \in \Samuel (X)) 
\end{displaymath} (cf.~\cite[2.1, Corollary~1.9, p.~69]{AnalysisOnSemigroups}). In turn, the Riesz--Markov--Kakutani representation theorem\footnote{The Riesz--Markov--Kakutani representation theorem as stated in~\cite[IV.6, Theorem~5]{DunfordSchwartz} moreover establishes an isometric isomorphism between $\UCB (X)^{\ast}$ and the Banach space of all regular finite signed Borel measures on~$\Samuel(X)$ equipped with the total variation norm.} (see, e.g.,~\cite[IV.6, Theorem~3]{DunfordSchwartz} or~\cite[Theorem~P.31]{PachlBook}) asserts that, for each $\mu \in \Mean (X)$, there exists a unique regular Borel probability measure $\hat{\mu}$ on $\Samuel (X)$ such that \begin{displaymath}
	\forall f \in \Cont (\Samuel (X)) \colon \qquad \int f \, \mathrm{d}\hat{\mu} \, = \, \mu(f \circ \eta_{X}) .
\end{displaymath}

\begin{remark}\label{remark:subspace.mean} Let $X$ be a uniform space and let $Y \subseteq X$. Then the linear map \begin{displaymath}
	\iota_{Y,X} \colon \, \UCB(Y)^{\ast} \, \longrightarrow \, \UCB(X)^{\ast} , \quad \mu \, \longmapsto \, (f \mapsto \mu(f\vert_{Y}))
\end{displaymath} is an isometric embedding with regard to the respective supremum norms as well as a topological embedding with regard to the respective weak-$\ast$ topologies. This is a consequence of the extension theorem for uniformly continuous bounded real-valued functions (see, e.g.,~\cite[Theorem~2.22]{PachlBook}). It also follows that $\iota_{Y,X}(\Mean (Y))$ is a weak-$\ast$ compact (thus closed) subspace of $\Mean (X)$, and that $\iota_{Y,X}(\Samuel (Y))$ is such of $\Samuel (X)$. Furthermore, a straightforward calculation shows that $\iota_{Y,X}(\eta_{Y}(y)) = \eta_{X}(y)$ for all $y \in Y$. \end{remark}

\begin{lem}\label{lemma:subspace.mean} Let $X$ be a uniform space and let $Y \subseteq X$. Then \begin{displaymath}
	\iota_{Y,X} (\Mean (Y)) \, = \, \{ \mu \in \Mean (X) \mid \hat{\mu}(\iota_{Y,X}(\Samuel (Y))) = 1 \} .
\end{displaymath} \end{lem}

\begin{proof} Let us abbreviate $\iota \defeq \iota_{Y,X}$. 

($\subseteq$) Let $\mu \in \iota (\Mean(Y))$. Then there exists $\nu \in \Mean (Y)$ such that $\mu = \iota(\nu)$. Recall that $\iota(\Samuel(Y))$ is closed in $\Samuel(X)$ by Remark~\ref{remark:subspace.mean}. Hence, if $U$ is an open subset of $\Samuel (X)$ containing $\iota (\Samuel(Y))$, then Urysohn's lemma asserts the existence of some $f \in \Cont (\Samuel (X))$ with $0 \leq f \leq 1$, $f(\iota (\Samuel(Y))) \subseteq \{ 1 \}$ and $f(\Samuel(X)\setminus U) \subseteq \{ 0 \}$, wherefore \begin{displaymath}
	(f \circ \eta_{X})(y) \, = \, f(\eta_{X}(y)) \, \stackrel{\ref{remark:subspace.mean}}{=} \, f(\iota(\eta_{Y}(y))) \, = \, 1 
\end{displaymath} for all $y \in Y$, i.e., $(f \circ \eta_{X})\vert_{Y} = 1$, and thus \begin{displaymath}
	\hat{\mu}(U) \, \geq \, \int f \, \mathrm{d}\hat{\mu} \, = \, \mu(f\circ \eta_{X}) \, = \, \nu ((f \circ \eta_{X})\vert_{Y}) \, = \, 1 .
\end{displaymath} Since $\hat{\mu}$ is regular, this entails that $\hat{\mu}(\iota (\Samuel(Y))) = 1$. 

($\supseteq$) Let $\mu \in \Mean (X)$ with $\hat{\mu}(\iota(\Samuel (Y))) = 1$. We are going to show that \begin{equation}\tag{$\ast$}\label{factoring}
	\forall f \in \UCB(X) \colon \qquad f\vert_{Y} = 0 \ \Longrightarrow \ \mu(f) = 0 .
\end{equation} To this end, let $f \in \UCB(X)$ such that $f\vert_{Y} = 0$. Then \begin{displaymath}
	\overline{f}(\iota (\xi)) \, = \, \iota(\xi)(f) \, = \, \xi(f\vert_{Y}) \, = \, 0
\end{displaymath} for every $\xi \in \Samuel (Y)$, that is, $\overline{f}\vert_{\iota(\Samuel(Y))} = 0$. As $\hat{\mu}(\iota(\Samuel (Y))) = 1$, we conclude that \begin{displaymath}
	\mu(f) \, = \, \mu \left( \overline{f} \circ {\eta_{X}} \right)\, = \, \int \overline{f} \, \mathrm{d}\hat{\mu} \, = \, \int_{\iota(\Samuel(Y))} \overline{f} \, \mathrm{d}\hat{\mu} \, = \, 0 ,
\end{displaymath} which readily proves~\eqref{factoring}. According to~\eqref{factoring} and the surjectivity of the linear operator $\UCB(X) \to \UCB(Y), \, f \mapsto f\vert_{Y}$ (see, e.g.,~\cite[Theorem~2.22]{PachlBook}), there exists a unique linear map $\nu \colon \UCB(Y) \to \R$ such that $\iota(\nu) = \mu$. As $\mu(1) = 1$, it follows that $\nu(1) = 1$. Moreover, positivity of $\mu$ is easily seen to imply that $\nu$ is positive, too. Hence, $\nu \in \Mean (Y)$ and therefore $\mu \in \iota (\Mean (Y))$. \end{proof}

For later use (see Lemma~\ref{lemma:conditional.expectation} below), we include a brief remark about functoriality of the Samuel compactification. If $X$ and $Y$ are uniform spaces and $\phi \colon X \to Y$ is uniformly continuous, then the map $\Samuel (\phi) \colon \Samuel (X) \to \Samuel (Y)$ defined by \begin{displaymath}
	\Samuel (\phi)(\xi)(f) \, \defeq \, \xi (f \circ \phi) \qquad (\xi \in \Samuel (X), \, f \in \UCB(Y))
\end{displaymath} is continuous. With these assignments, $\Samuel$ constitutes a functor from the category of uniform spaces and uniformly continuous maps into the category of compact Hausdorff spaces and continuous maps. The transformation $\eta$ is natural in the following sense.

\begin{remark}\label{remark:natural} If $X$ and $Y$ are uniform spaces and $\phi \colon X \to Y$ is uniformly continuous, then $\Samuel (\phi) \circ \eta_{X} = \eta_{Y} \circ \phi$. \end{remark}

We now begin our study of concentration of means on uniform spaces.

\begin{lem}[Markov inequality]\label{lemma:markov.chebyshev.general} Let $X$ be a uniform space and let $\mu \in \Mean (X)$. For every $f \in \UCB(X)$ with $f \geq 0$, \begin{displaymath}
	\hat{\mu}(\{ \xi \in \Samuel(X) \mid \xi(f) \geq 1 \}) \, \leq \, \mu(f) .
\end{displaymath} \end{lem}

\begin{proof} If $f \in \UCB(X)$ and $f \geq 0$, then $\chi_{\{ \xi \in \Samuel(X) \mid \xi(f) \geq 1\}} \leq \overline{f}$ and thus \begin{displaymath}
	\hat{\mu}(\{ \xi \in \Samuel(X) \mid \xi(f) \geq 1 \}) \, \leq \, \int \overline{f} \, \mathrm{d}\hat{\mu} \, = \, \mu\left( \overline{f} \circ \eta_{X} \right) \, = \, \mu(f) . \qedhere
\end{displaymath} \end{proof}

\begin{lem}\label{lemma:convergence} Let $X$ be a uniform space. Furthermore, let $\mu \in \Mean (X)$, $f \in \UCB(X)$, and $\epsilon \in \R_{>0}$. Then the following hold. \begin{itemize}
	\item[$(1)$] $\hat{\mu}(\{ \xi \in \Samuel(X) \mid \vert \xi(f) - \mu(f) \vert \geq \epsilon \}) \leq \tfrac{1}{\epsilon^{2}}\mu\left( (f-\mu(f))^{2} \right)$.
	\item[$(2)$] $\mu\left( (f-\mu(f))^{2} \right) \leq (\diam f(X))^{2}\hat{\mu}(\{ \xi \in \Samuel(X) \mid \vert \xi(f) - \mu(f) \vert \geq \epsilon \}) + \epsilon^{2}$.
\end{itemize} \end{lem}

\begin{proof} (1) We observe that \begin{align*}
	\hat{\mu}(\{ \xi \in \mathrm{S}(X) \mid \vert \xi(f) - \mu(f) \vert \geq \epsilon \}) \, & = \, \hat{\mu}\left(\left\{ \xi \in \mathrm{S}(X) \left\vert \, (\xi(f) - \mu(f))^{2} \geq \epsilon^{2} \right\}\right) \right. \\
		& = \, \hat{\mu}\left(\bigl\{ \xi \in \mathrm{S}(X) \, \big\vert \, \xi\bigl(\tfrac{1}{\epsilon^{2}}(f - \mu(f))^{2}\bigr) \geq 1 \bigr\}\right) \\
		& \stackrel{\ref{lemma:markov.chebyshev.general}}{\leq} \, \mu \left(\tfrac{1}{\epsilon^{2}}(f - \mu(f))^{2}\right) \, = \, \tfrac{1}{\epsilon^{2}}\mu\left((f - \mu(f))^{2}\right) .
\end{align*}
	
(2) Our hypothesis entails that $\Mean (X) \ne \emptyset$ and thus $X \ne \emptyset$. Consequently, both $s \defeq \inf f(X) \in \R$ and $t \defeq \sup f(X) \in \R$. As $s\leq f \leq t$ and $\mu \in \Mean (X)$, it follows that $s \leq \mu(f) \leq t$. Hence, for every $x \in X$, \begin{displaymath}
	\left\lvert \overline{f}(\eta_{X}(x)) - \mu (f)\right\rvert \, = \, \vert f(x) - \mu(f) \vert \, \leq \, t-s \, = \, \diam f(X) .
\end{displaymath} Since $\eta_{X}(X)$ is dense in $\Samuel (X)$, this implies that \begin{equation}\tag{$\ast$}\label{diam}
	\forall \xi \in \Samuel (X) \colon \qquad \left\lvert \overline{f}(\xi) - \mu(f) \right\rvert \, \leq \, \diam f(X) .
\end{equation} Considering the closed subset \begin{displaymath}
	B \, \defeq \, \left\{ \xi \in \Samuel(X) \left\vert \, \left\lvert \overline{f}(\xi) - \mu(f) \right\rvert \geq \epsilon \right\} \! \right. \, = \, \{ \xi \in \Samuel(X) \mid \vert \xi(f) - \mu(f) \vert \geq \epsilon \}
\end{displaymath} of $\Samuel (X)$, we now conclude that \begin{align*}
	\mu \left( (f-\mu(f))^{2} \right) \, & = \, \int \overline{(f-\mu(f))^{2}} \, \mathrm{d}\hat{\mu} \, = \, \int \left(\overline{f}-\mu(f)\right)^{2} \, \mathrm{d}\hat{\mu} \\
		& = \, \int_{B} \left(\overline{f}-\mu(f)\right)^{2} \, \mathrm{d}\hat{\mu} + \int_{\Samuel (X) \setminus B} \left(\overline{f}-\mu(f)\right)^{2} \, \mathrm{d}\hat{\mu} \\
		&\stackrel{\eqref{diam}}{\leq} \, (\diam f(X))^{2}\hat{\mu}(B) + \epsilon^{2} .\qedhere
\end{align*} \end{proof}

\begin{prop}\label{proposition:uniform.convergence} Let $X$ be a uniform space, $(\mu_{i})_{i \in I}$ be a net in $\Mean (X)$, and $B$ be a norm-bounded subset of $\UCB(X)$. The following are equivalent. \begin{itemize}
	\item[$(1)$] $\sup\nolimits_{f \in B} \mu_{i}\left( (f-\mu_{i}(f))^{2} \right) \, \longrightarrow \, 0$ as $i \to I$.
	\item[$(2)$] For every $\epsilon \in \R_{>0}$, \begin{displaymath}
			\qquad \sup\nolimits_{f \in B} \hat{\mu}_{i}(\{ \xi \in \Samuel(X) \mid \vert \xi(f) - \mu_{i}(f) \vert \geq \epsilon \})  \, \longrightarrow \, 0 \quad (i \to I) .
		\end{displaymath}
\end{itemize} \end{prop}

\begin{proof} While the implication (1)$\Longrightarrow$(2) is due to Lemma~\ref{lemma:convergence}(1), the implication (2)$\Longrightarrow$(1) follows from Lemma~\ref{lemma:convergence}(2) and norm-boundedness of $B$. \end{proof}

The following fact casts concentration in uniform spaces (Definition~\ref{definition:concentration}) as a uniform version of convergence to multiplicative homomorphisms.

\begin{prop}\label{proposition:convergence} Let $X$ be a uniform space and $(\mu_{i})_{i \in I}$ be a net in $\Mean (X)$. The following are equivalent. \begin{itemize}
	\item[$(1)$] Every weak-$\ast$ accumulation point of $(\mu_{i})_{i \in I}$ belongs to $\Samuel(X)$.
	\item[$(2)$] For every $f \in \UCB(X)$, \begin{displaymath}
			\qquad \mu_{i}\left( (f-\mu_{i}(f))^{2} \right) \, \longrightarrow \, 0 \quad (i \to I) .
		\end{displaymath}
	\item[$(3)$] For every $f \in \UCB(X)$ and every $\epsilon \in \R_{>0}$, \begin{displaymath}
			\qquad \hat{\mu}_{i}(\{ \xi \in \Samuel(X) \mid \vert \xi(f) - \mu_{i}(f) \vert \geq \epsilon \})  \, \longrightarrow \, 0 \quad (i \to I) .
		\end{displaymath}
\end{itemize}\end{prop}

\begin{proof} (1)$\Longleftrightarrow$(2). By a general fact about arbitrary real (or complex) unital algebras (see, e.g.,~\cite[Lemma~2.1.5]{kaniuth}), \begin{displaymath}
	\Samuel (X) \, = \, \bigl\{ \mu \in \R^{\UCB(X)} \big\vert \ \mu \text{ linear, unital}, \, \forall f \in \UCB(X) \colon \mu\bigl(f^{2}\bigr) = \mu(f)^{2} \bigr\} . 
\end{displaymath} Also, if $\mu \in \R^{\UCB(X)}$ is linear and unital, then \begin{displaymath}
	\mu\bigl(f^{2}\bigr)-\mu(f)^{2}\! \, = \, \mu\bigl(f^{2}\bigr)-2\mu(f)^{2} + \mu(f)^{2}\! \, = \, \mu\bigl( f^{2} -2\mu(f)f + \mu(f)^{2} \bigr) \, = \, \mu\bigl((f-\mu(f))^{2}\bigr)
\end{displaymath} for every $f \in \UCB(X)$. Hence, \begin{equation}\tag{$\ast$}\label{samuel}
	\Samuel(X) \, = \, \bigl\{ \mu \in \R^{\UCB(X)} \big\vert \ \mu \text{ linear, unital}, \, \forall f \in \UCB(X) \colon \mu\bigl((f-\mu(f))^{2}\bigr) = 0 \bigr\} .
\end{equation} Moreover, for all $\mu, \nu \in \UCB(X)^{\ast}$ and $f \in \UCB(X)$, since \begin{align*}
	(f-\mu(f))^{2} - (f-\nu(f))^{2} \, &= \, ((f-\mu(f)) + (f-\nu(f)))((f-\mu(f)) - (f-\nu(f))) \\
	& = \, (2f-\mu(f)-\nu(f))(\nu(f)-\mu(f)) \\
	& = \, (\mu(f)+\nu(f)-2f)(\mu(f)-\nu(f))
\end{align*} and therefore \begin{align*}
	\left\lVert (f-\mu(f))^{2} - (f-\nu(f))^{2} \right\rVert_{\infty} \, &= \, \left\lVert \mu(f)+\nu(f)-2f \right\rVert_{\infty} \left\lvert \mu(f) - \nu(f) \right\rvert \\
	& \leq \, (\Vert \mu \Vert + \Vert \nu \Vert + 2) \left\lVert f \right\rVert_{\infty} \left\lvert \mu(f) - \nu(f) \right\rvert ,
\end{align*} we conclude that \begin{equation}\tag{$\ast\ast$}\label{samuel2}
	\begin{split}
		\left\lvert \mu \left( (f-\mu(f))^{2} \right) - \nu \left( (f-\nu(f))^{2} \right) \right\rvert \, &\leq \, \left\vert \mu\left( (f-\mu(f))^{2} \right) - \nu\left( (f-\mu(f))^{2} \right) \right\vert \\
		&\hspace{1mm} + \Vert \nu \Vert (\Vert \mu \Vert + \Vert \nu \Vert + 2)\left\lVert f \right\rVert_{\infty} \left\lvert \mu(f)-\nu(f) \right\rvert .
	\end{split}
\end{equation} Now, by~\eqref{samuel} and~\eqref{samuel2} and due to $(\mu_{i})_{i \in I}$ being norm-bounded, (2) implies (1). Conversely, to show that $\neg (2) \Longrightarrow \neg (1)$, suppose that there exist $f \in \UCB(X)$ and $\epsilon \in \R_{>0}$ such that the subset \begin{displaymath}
	J \, \defeq \, \bigl\{ i \in I \, \big\vert \, \mu_{i}\bigl( (f-\mu_{i}(f))^{2} \bigr) \geq \epsilon \bigr\} 
\end{displaymath} is cofinal in $I$. In particular, $J$ constitutes a directed set with respect to the preorder inherited from $I$. By weak-$\ast$ compactness of $\Mean (X)$, the net $(\mu_{j})_{j \in J}$ admits a weak-$\ast$ accumulation point $\mu \in \Mean (X)$, which then satisfies \begin{displaymath}
	\mu\left( (f-\mu(f))^{2}\right) \, \geq \, \epsilon
\end{displaymath} by~\eqref{samuel2}. Consequently, $\mu \notin \Samuel(X)$ by~\eqref{samuel}. Since $J$ is cofinal in $I$, the mean $\mu$ is a weak-$\ast$ accumulation point of $(\mu_{i})_{i \in I}$, confirming the negation of~(1).

(2)$\Longleftrightarrow$(3). This is due to Proposition~\ref{proposition:uniform.convergence}. \end{proof}

In order to specify a certain class of norm-bounded sets to be discussed in connection with Proposition~\ref{proposition:uniform.convergence}, let $X$ be a uniform space. Following~\cite[Definition~1.19]{PachlBook}, a subset $B \subseteq \mathrm{UCB}(X)$ will be called \emph{UEB}\footnote{which is short for \emph{uniformly equicontinuous, bounded}} if \begin{itemize}
	\item[---] $B$ is bounded in the supremum norm, and
	\item[---] $B$ is uniformly equicontinuous, i.e., for every $\epsilon \in \R_{>0}$ there exists an entourage $U$ of $X$ such that \begin{displaymath}
	\qquad \forall f \in B \ \forall (x,y) \in U \colon \qquad \vert f(x) - f(y) \vert \, \leq \, \epsilon .
\end{displaymath}
\end{itemize} The collection $\mathrm{UEB}(X)$ of all UEB subsets of $\mathrm{UCB}(X)$ constitutes a convex vector bornology on the vector space $\mathrm{UCB}(X)$. The set $\UEB(X)$ is an essential ingredient in Pachl's monograph~\cite{PachlBook}.

\begin{definition}\label{definition:concentration} Let $X$ be a uniform space. A net $(\mu_{i})_{i \in I}$ in $\Mean (X)$ is said to \emph{concentrate in $X$} if \begin{displaymath}
	\forall B \in \UEB (X) \colon \qquad \sup\nolimits_{f \in B} \mu_{i}\left( (f-\mu_{i}(f))^{2} \right) \, \longrightarrow \, 0 \quad (i \to I) .
\end{displaymath} \end{definition}

\begin{cor}\label{corollary:concentration} Let $X$ be a uniform space and let $(\mu_{i})_{i \in I}$ be a net in $\Mean (X)$. The following are equivalent. \begin{itemize}
	\item[$(1)$] $(\mu_{i})_{i \in I}$ concentrates in $X$.
	\item[$(2)$] For every $B \in \UEB(X)$ and every $\epsilon \in \R_{>0}$, \begin{displaymath}
			\qquad \sup\nolimits_{f \in B} \hat{\mu}_{i}(\{ \xi \in \Samuel(X) \mid \vert \xi(f) - \mu_{i}(f) \vert \geq \epsilon \})  \, \longrightarrow \, 0 \quad (i \to I) .
		\end{displaymath}
\end{itemize} \end{cor}

\begin{proof} This is due to Proposition~\ref{proposition:uniform.convergence}. \end{proof}

\section{Amenability and convolution algebras}\label{section:topological.groups}

In this section, we briefly collect some preliminary material concerning general topological groups and their convolution algebras from the literature~\cite{AnalysisOnSemigroups,PestovBook,PachlBook,GrigorchukDeLaHarpe}. The presentation is focused on aspects relevant to our subsequent discussion of amenability and extreme amenability.

Let $G$ be a topological group. For each $g \in G$, we define \begin{displaymath}
	\lambda_{g} \colon \, G \, \longrightarrow \, G, \quad x \, \longmapsto \, gx .
\end{displaymath} Let $\mathscr{U}(G)$ denote the neighborhood filter of the neutral element $e = e_{G} \in G$. Henceforth, $G$ is viewed as a uniform space, carrying the \emph{right uniformity} \begin{displaymath}
	\mathscr{E}_{\Rsh}(G) \, \defeq \, \left\{ E \subseteq G \times G \left\vert \, \exists U \in \mathscr{U}(G) \, \forall x,y \in G \colon \, xy^{-1}\! \in U \Longrightarrow \, (x,y) \in E \right\} \right.\! .
\end{displaymath} We put $\RUCB(G) \defeq \UCB(G,\mathscr{E}_{\Rsh}(G))$ and note that \begin{displaymath}
	\RUCB(G) \, = \, \left\{ f \in \ell^{\infty}(G) \left\vert \, \forall \epsilon \in \R_{>0} \, \exists U \in \mathscr{U}(G) \, \forall g \in U\colon \, \Vert f-(f\circ \lambda_{g}) \Vert_{\infty} \leq \epsilon \right\} . \right.
\end{displaymath} Moreover, let us observe that a norm-bounded subset $B \subseteq \ell^{\infty}(G)$ belongs to $\RUEB (G) \defeq \UEB (G,\mathscr{E}_{\Rsh}(G))$ if and only if \begin{displaymath}
	\forall \epsilon \in \R_{>0} \, \exists U \in \mathscr{U}(G) \, \forall g \in U \, \forall f \in B \colon \quad \Vert f-(f\circ \lambda_{g}) \Vert_{\infty} \leq \epsilon .
\end{displaymath} The topological group $G$ admits a continuous action on the compact Hausdorff space $\Mean (G) \defeq \Mean (G,\mathscr{E}_{\Rsh}(G))$ defined by \begin{displaymath}
	(g\mu)(f) \, \defeq \, \mu(f \circ \lambda_{g}) \qquad (g \in G, \, \mu \in \Mean (G), \, f \in \RUCB(G)) .
\end{displaymath} Furthermore, $\Samuel (G) \defeq \Samuel (G,\mathscr{E}_{\Rsh}(G))$ constitutes a $G$-invariant subset of $\Mean (G)$. The topological group $G$ is said to be \emph{amenable} (resp., \emph{extremely amenable}) if $\Mean (G)$ (resp., $\Samuel(G)$) contains a $G$-fixed point. A prominent characterization~\cite[Proposition~3.6]{GrigorchukDeLaHarpe} (resp.,~\cite[Theorem~1]{mitchell}) asserts that $G$ is amenable (resp., extremely amenable) if and only if every continuous\footnote{that is, jointly continuous} action of $G$ on a non-void compact Hausdorff space admits an invariant regular Borel probability measure (resp., a fixed point).

The reader is referred to~\cite{PestovBook,GrigorchukDeLaHarpe} for a more comprehensive account on (extreme) amenability of topological groups, while we confine ourselves to noting following basic characterization.

\begin{prop}\label{proposition:fremlin} Let $G$ be a topological group. The following are equivalent. \begin{itemize}
	\item[$(1)$] $G$ is extremely amenable.
	\item[$(2)$] For all $F \in \Pfin (\RUCB(G))$, $E \in \Pfin (G)$ and $\epsilon \in \R_{>0}$, there is $\mu \in \Mean (G)$ such that, for each $f \in F$, \begin{displaymath}
					\qquad \quad \, \sup\nolimits_{g \in E} \left\lvert \mu(f) - \mu(f \circ \lambda_{g}) \right\rvert \leq \epsilon , \quad \, \hat{\mu}(\{ \xi \in \Samuel (G) \mid \vert \xi(f) - \mu(f) \vert \geq \epsilon \}) \leq \epsilon .
				\end{displaymath}
\end{itemize} \end{prop}

\begin{proof} While the former clearly implies the latter, the converse implication follows by weak-$\ast$ compactness of $\Mean (G)$ and Proposition~\ref{proposition:convergence}. \end{proof}

Our next objective is to recollect a well-known construction of convolution algebras for general topological groups, following~\cite[Section~2.2]{AnalysisOnSemigroups} (see also~\cite[Lemma~3.2]{SchneiderThomRandomWalks} and~\cite[Section~9.2]{PachlBook}). Let $G$ be a topological group. If $\mu \in \RUCB(G)^{\ast}$, then the map $\Phi_{\mu} \colon \RUCB(G) \to \RUCB(G)$ defined by \begin{displaymath}
	(\Phi_{\mu}f)(g) \, \defeq \, \mu(f \circ \lambda_{g}) \qquad (f \in \RUCB(G), \, g \in G)
\end{displaymath} is a bounded linear operator with $\Vert \Phi_{\mu} \Vert = \Vert \mu \Vert$. The Banach space $\RUCB (G)^{\ast}$, equipped with the multiplication given by \begin{displaymath}
	\mu \nu \, \defeq \, \mu \circ \Phi_{\nu} \qquad (\mu,\nu \in \RUCB(G)^{\ast}) ,
\end{displaymath} constitutes a unital real Banach algebra, whose multiplicative unit is $\eta_{G}(e)$. Moreover, both $\Mean (G)$ and $\Samuel (G)$ are submonoids of the multiplicative monoid of $\RUCB (G)^{\ast}$. The following well-known facts are easily verified, too.

\begin{remark}\label{remark:convolution} Let $G$ be a topological group. Then the following hold. \begin{itemize}
	\item[$(1)$] The mapping \begin{displaymath}
		\qquad \RUCB(G)^{\ast} \, \longrightarrow \, \B (\RUCB(G)), \quad \mu \, \longmapsto \, \Phi_{\mu}
	\end{displaymath} is an isometric unital Banach algebra embedding.
	\item[$(2)$] $\Mean (G) = \{ \mu \in \RUCB(G)^{\ast} \mid \Phi_{\mu} \text{ positive and unital} \}$.
	\item[$(3)$] $\Samuel (G) = \{ \mu \in \RUCB(G)^{\ast} \mid \Phi_{\mu} \text{ unital ring homomorphism} \}$.
\end{itemize} \end{remark}

Now, if $H$ is a topological subgroup of a topological group $G$, then \begin{displaymath}
	\mathscr{E}_{\Rsh}(H) \, = \, \mathscr{E}_{\Rsh}(G)\!\!\upharpoonright_{H} ,
\end{displaymath} whence Remark~\ref{remark:subspace.mean} asserts that the linear map \begin{displaymath}
	\iota_{H,G}\colon \, \RUCB(H)^{\ast} \, \longrightarrow \, \RUCB(G)^{\ast}, \quad \mu \, \longmapsto \, (f \mapsto \mu(f\vert_{H})) 
\end{displaymath} is both an isometric embedding relative to the respective supremum norms and a topological embedding relative to the respective weak-$\ast$ topologies.

\begin{lem}\label{lemma:multiplicative.embedding} Let $G$ be a topological group and let $H \leq G$. Then \begin{displaymath}
	\iota_{H,G} \colon \, \RUCB(H)^{\ast} \, \longrightarrow \, \RUCB(G)^{\ast}
\end{displaymath} is a multiplicative monoid homomorphism. \end{lem}

\begin{proof} Let us abbreviate $\iota \defeq \iota_{H,G}$. First of all, we note that $\iota (\eta_{H}(e)) = \eta_{G}(e)$ due to Remark~\ref{remark:subspace.mean}. Now, let $\mu, \nu \in \RUCB(H)^{\ast}$. Verifying that $\iota (\mu\nu) = \iota (\mu) \iota(\nu)$ amounts to showing that $\iota (\mu \nu)(f) = (\iota(\mu)\iota(\nu))(f)$ for every $f \in \RUCB(G)$. For this purpose, let $f \in \RUCB(G)$. Then \begin{displaymath}
	(\Phi_{\nu}(f\vert_{H}))(h) \, = \, \nu (f\vert_{H} \circ \lambda_{h}) \, = \, \nu ((f \circ \lambda_{h})\vert_{H}) \, = \, \iota (\nu)(f \circ \lambda_{h}) \, = \, (\Phi_{\iota(\nu)}f)(h)
\end{displaymath} for every $h \in H$, that is, \begin{equation}\tag{$\ast$}\label{suboperator}
	\Phi_{\nu}(f\vert_{H}) \, = \, (\Phi_{\iota(\nu)}f)\vert_{H} .
\end{equation} Therefore, as desired, \begin{align*}
	\iota(\mu \nu)(f) \, &= \, (\mu \nu)(f\vert_{H}) \, = \, \mu (\Phi_{\nu}(f\vert_{H})) \, \stackrel{\eqref{suboperator}}{=} \, \mu((\Phi_{\iota(\nu)}f)\vert_{H}) \\
	& = \, \iota(\mu)(\Phi_{\iota(\nu)}f) \, = \, (\iota(\mu)\iota(\nu))(f) . \qedhere
\end{align*} \end{proof}

Let us agree on the following notational convention.

\begin{remark}\label{remark:convention} Let $G$ be a topological group and let $H \leq G$. Henceforth, based on Remark~\ref{remark:subspace.mean} and Lemma~\ref{lemma:multiplicative.embedding}, we will view $\RUCB(H)^{\ast}$ as a (norm-closed unital) subalgebra of $\RUCB(G)^{\ast}$ by identifying $\RUCB(H)^{\ast}$ with its image under $\iota_{H,G}$. Correspondingly, $\Mean (H)$ will be regarded as a (weak-$\ast$ closed) submonoid of $\Mean (G)$, and likewise $\Samuel (H)$ as such of $\Samuel (G)$. \end{remark}

For convenience, we set up additional terminology. Let $G$ be a topological group and let $H$ be a subgroup of $G$. A function $f \colon G \to \R$ is called \emph{$H$-right-invariant} if $f(xh) = f(x)$ for all $x \in G$ and $h \in H$. A mean $\mu \in \Mean (G)$ is called \emph{$H$-left-invariant} if $\mu(f \circ \lambda_{h}) = \mu(f)$ for all $f \in \RUCB(G)$ and $h \in H$.

\begin{lem}\label{lemma:invariance} Let $G$ be a topological group, let $H \leq G$ and let $\mu \in \Mean (G)$. The following are equivalent. \begin{itemize}
	\item[$(1)$] $\mu$ is $H$-left-invariant.
	\item[$(2)$] $\Phi_{\mu}f$ is $H$-right-invariant for every $f \in \RUCB(G)$.
	\item[$(3)$] $\nu\mu = \mu$ for every $\nu \in \Mean(H)$.
\end{itemize} \end{lem}

\begin{proof} (1)$\Longrightarrow$(2). If $f \in \RUCB(G)$, then \begin{displaymath}
	(\Phi_{\mu}f)(xh) \, = \, \mu(f \circ \lambda_{xh}) \, = \, \mu(f \circ \lambda_{x} \circ \lambda_{h}) \, \stackrel{(1)}{=} \, \mu(f \circ \lambda_{x}) \, = \, (\Phi_{\mu}f)(x)
\end{displaymath} for all $x \in G$ and $h \in H$, which means that $\Phi_{\mu}f$ is $H$-right-invariant.
	
(2)$\Longrightarrow$(3). If $\nu \in \Mean(H)$, then \begin{displaymath}
	(\nu\mu)(f) \, = \, \nu((\Phi_{\mu}f)\vert_{H}) \, \stackrel{(2)}{=} \, \nu((\Phi_{\mu}f)(e)\cdot 1) \, = \, (\Phi_{\mu}f)(e) \, = \, \mu(f)
\end{displaymath} for every $f \in \RUCB(G)$, i.e., $\nu\mu = \mu$.
	
(3)$\Longrightarrow$(1). For all $h \in H$ and $f \in \RUCB(G)$, \begin{displaymath}
	\mu(f \circ \lambda_{h}) \, = \, (\Phi_{\mu}f)(h) \, = \, \eta_{H}(h)((\Phi_{\mu}f)\vert_{H})\, = \, (\eta_{H}(h) \mu)(f) \, \stackrel{(3)}{=} \, \mu(f) ,
\end{displaymath} wherefore $\mu$ is $H$-left-invariant. \end{proof}

\begin{remark}\label{remark:ideal.of.invariant.means} Let $H$ be a subgroup of a topological group $G$. If $\mu \in \Mean (G)$ is $H$-left-invariant, then so is $\mu \nu$ for every $\nu \in \Mean (G)$. For instance, this easily follows from~\cite[2.2, Proposition~2.3(v)(1), p.~73]{AnalysisOnSemigroups}. \end{remark}

The discussion of invariance properties of functions naturally leads to coset spaces of topological groups. In order to clarify some related notation and terminology, let $H$ be a subgroup of a topological group $G$. We endow the set $G/H = \{ xH \mid x \in G\}$ with the \emph{right uniformity} \begin{displaymath}
	\left\{ E \subseteq (G/H) \times (G/H) \left\vert \, \exists U \in \mathscr{U}(G) \, \forall x,y \in G \colon \, xy^{-1}\! \in U \Longrightarrow \, (xH,yH) \in E \right\} , \right.\! 
\end{displaymath} which is the finest uniformity on $G/H$ such that \begin{displaymath}
	\pi_{H} \colon \, G \, \longrightarrow \, G/H, \quad x \, \longmapsto \, xH
\end{displaymath} is uniformly continuous with respect to the right uniformity on $G$ (see, for instance,~\cite[Lemma~6.2.4]{PestovBook}). Henceforth, whenever $G/H$ is considered as a uniform space, we will be referring to the right uniformity.

\begin{remark}\label{remark:quotient.spaces} Let $H$ be a subgroup of a topological group $G$. Then, \begin{displaymath}
	\{ f \in \RUCB(G) \mid f \text{ $H$-right-invariant} \} \, = \, \{ f \circ \pi_{H} \mid f \in \UCB(G/H)\} .
\end{displaymath} \end{remark} 

We conclude this preliminary section by recalling that, with respect to certain pseudo-metrics on a topological group, convolution operators induced by means preserve Lipschitz properties of functions. Let us clarify the relevant terminology. A pseudo-metric $d$ on a topological space~$X$ is called \emph{continuous} (resp., \emph{compatible}) if the topology generated by $d$ is contained in (resp., coincides with) the topology of $X$. A pseudo-metric $d$ on a group $G$ is called \emph{right-invariant} (resp., \emph{left-invariant}) if $d(xg,yg) = d(x,y)$ (resp., $d(gx,gy) = d(x,y)$) for all $g,x,y \in G$. A pseudo-metric on a group is said to be \emph{bi-invariant} if it is both left- and right-invariant.

\begin{remark}\label{remark:lipschitz.right.uniformly.continuous} Let $d$ be a continuous right-invariant pseudo-metric on a topological group $G$. For all $k,r \in \R_{\geq 0}$, \begin{displaymath}
	\Lip_{k}(G,d;[-r,r]) \, \in \, \RUEB(G) .
\end{displaymath} In particular, $\Lip_{1}^{\infty}(G,d) \subseteq \RUCB(G)$. \end{remark}

The following lemma is well known among experts (for an extensive generalization, see~\cite[Proof of Lemma~9.1]{PachlBook}).

\begin{lem}\label{lemma:lipschitz} Let $d$ be a continuous right-invariant pseudo-metric on a topological group $G$. If $f \in \Lip_{1}^{\infty}(G,d)$ and $\mu \in \Mean (G)$, then $\Phi_{\mu}f \in \Lip_{1}^{\infty}(G,d)$. \end{lem}

\begin{proof} If $f \in \Lip_{1}^{\infty}(G,d)$ and $\mu \in \Mean (G)$, then \begin{align*}
	\left\lvert (\Phi_{\mu}f)(g) - (\Phi_{\mu}f)(h) \right\rvert \, &= \, \left\lvert \mu (f \circ \lambda_{g}) - \mu (f \circ \lambda_{h}) \right\rvert \, = \, \left\lvert \mu((f \circ \lambda_{g})-(f \circ \lambda_{h})) \right\rvert \\
		&\leq \, \left\lVert (f\circ \lambda_{g}) - (f\circ \lambda_{h}) \right\rVert_{\infty} \, = \, \sup\nolimits_{x \in G} \vert f(gx) - f(hx) \vert \\
		&\leq \, \sup\nolimits_{x \in G} d(gx,hx) \, = \, d(g,h)
\end{align*} for all $g,h \in G$, i.e., $\Phi_{\mu}f \in \Lip_{1}^{\infty}(G,d)$. \end{proof}

\section{Concentration of invariant means}\label{section:concentration.of.invariant.means}

This section contains our main concentration result (Theorem~\ref{theorem:main.concentration}), along with some abstract dynamical consequences (Theorem~\ref{theorem:extreme.amenability} and Corollaries~\ref{corollary:extreme.amenability.1}--\ref{corollary:extreme.amenability.2}) and a resulting amplification technique (Proposition~\ref{proposition:amenable.folding}). The two principal ingredients in the proof of Theorem~\ref{theorem:main.concentration} are Azuma's martingale inequality (Theorem~\ref{theorem:azuma}) and a dynamical description of certain expectation operators (Lemma~\ref{lemma:conditional.expectation}).

We now recall some basic elements of probability theory relevant to the proof of Theorem~\ref{theorem:main.concentration}. Let $(X,\mathscr{B},\mu)$ be a probability space, let $\mathscr{B}_{0}$ be a sub-$\sigma$-algebra of $\mathscr{B}$, and let $f \in \mathscr{L}^{1}(X,\mathscr{B},\mu)$. A function $g \in \mathscr{L}^{1}(X,\mathscr{B}_{0},\mu\vert_{\mathscr{B}_{0}})$ is said to be a \emph{version of the conditional expectation} of $f$ relative to $\mathscr{B}_{0}$ if \begin{displaymath}
	\forall B \in \mathscr{B}_{0} \colon \qquad \int \chi_{B}\cdot f \, \mathrm{d}\mu \, = \, \int \chi_{B}\cdot g \, \mathrm{d}\mu .
\end{displaymath} The set of all versions of the conditional expectation of $f$ relative to $\mathscr{B}_{0}$ constitutes an element of $L^{1}(X,\mathscr{B}_{0},\mu\vert_{\mathscr{B}_{0}})$ (see, e.g.,~\cite[Chapter~23]{schilling}), which will be denoted by $\E_{\mu}(f\vert \, \mathscr{B}_{0})$. As a matter of course, $\E_{\mu}(f\vert \, \{ \emptyset, X \})$ takes the value $\E_{\mu}(f) \defeq \int f \, \mathrm{d}\mu$ $\mu$-almost everywhere. Moreover, if $f \in \mathscr{L}^{\infty}(X,\mathscr{B},\mu)$, then $\E_{\mu}(f\vert \, \mathscr{B}_{0})$ belongs to $L^{\infty}(X,\mathscr{B}_{0},\mu\vert_{\mathscr{B}_{0}})$. 

Let us specify some additional notation. If $h \colon X \to X'$ is any mapping and $\mathscr{B}'$ is a $\sigma$-algebra on $X'$, then $h^{-1}[\mathscr{B}'] \defeq \{ h^{-1}(B) \mid B \in \mathscr{B}' \}$ constitutes a $\sigma$-algebra on the set $X$, which we will refer to as the \emph{pull-back ($\sigma$-algebra)} of $\mathscr{B}'$ along $h$. Now, let $(X,\mathscr{B},\mu)$ be a probability space, $(X',\mathscr{B}')$ be a measurable space and $h \colon (X,\mathscr{B}) \to (X',\mathscr{B}')$ be a measurable map. As usual, the \emph{push-forward (measure)} of $\mu$ along $h$ is defined to be the probability measure \begin{displaymath}
	h_{\ast}(\mu) \colon \, \mathscr{B}'\! \, \longrightarrow \, [0,1], \quad B \, \longmapsto \, \mu\left(h^{-1}(B)\right) .
\end{displaymath} Furthermore, for every $f \in \mathscr{L}^{1}(X,\mathscr{B},\mu)$, we let \begin{displaymath}
	\E_{\mu}(f\vert \, h) \, \defeq \, \E_{\mu} \left(f \left\vert \, h^{-1}[\mathscr{B}']\right) . \right.
\end{displaymath}

\begin{thm}[Azuma~\cite{azuma}]\label{theorem:azuma} Let $(X,\mathscr{B},\mu)$ be a probability space, let $n \in \N$ and let $f \in \mathscr{L}^{\infty}(X,\mathscr{B},\mu)$. Consider a chain of sub-$\sigma$-algebras \begin{displaymath}
	\{ \emptyset,X\} \, = \, \mathscr{B}_{n} \, \subseteq \, \ldots \, \subseteq \, \mathscr{B}_{0} \, = \, \mathscr{B}
\end{displaymath} and let \begin{displaymath}
	d_{i} \, \defeq \, \left\lVert \E_{\mu}(f\vert \, \mathscr{B}_{i}) - \E_{\mu}(f\vert\,\mathscr{B}_{i+1}) \right\rVert_{\infty} \qquad (i \in \{ 0,\ldots,n-1 \}) .
\end{displaymath} Then, for every $\epsilon \in \R_{>0}$, \begin{displaymath}
	\mu\left(\left\{ x \in X \left\vert \, \left\lvert f(x) - \E_{\mu}(f) \right\rvert \geq \epsilon \right\}\right) \right. \! \, \leq \, 2 \exp \left( -\tfrac{\epsilon^{2}}{2 \sum\nolimits_{i=0}^{n-1} d_{i}^{2}} \right) .
\end{displaymath} \end{thm}

\begin{proof} See~\cite[4.1, Lemma~4.1]{ledoux} for the precise statement. \end{proof}

In order to deduce Theorem~\ref{theorem:main.concentration} from Theorem~\ref{theorem:azuma}, we need to connect the probabilistic tools outlined above with our dynamical setting, which happens via the following two lemmata. Concerning the notation used in Lemma~\ref{lemma:conditional.expectation}, we refer to both Remark~\ref{remark:natural} and Remark~\ref{remark:quotient.spaces} as well as the paragraphs preceding those two remarks. 

\begin{lem}\label{lemma:expectation} Let $G$ be a topological group and let $H$ be a topological subgroup of $G$. If $f \in \RUCB(G)$ is $H$-right-invariant, then \begin{displaymath}
	\forall \mu \in \Mean (H) \ \forall g \in \RUCB(G) \colon \qquad \Phi_{\mu}(f\cdot g) \, = \, f \cdot (\Phi_{\mu}g) .
\end{displaymath} \end{lem}

\begin{proof} Suppose that $f \in \RUCB(G)$ is $H$-right-invariant. Now, if $\mu \in \Mean (H)$ and $g \in \RUCB(G)$, then \begin{align*}
	\Phi_{\mu}(f\cdot g)(x) \, & = \, \mu (((f\cdot g)\circ \lambda_{x})\vert_{H}) \, = \, \mu((f\circ \lambda_{x})\vert_{H} \cdot (g\circ \lambda_{x})\vert_{H}) \\
		& = \, \mu(f(x) \cdot (g\circ \lambda_{x})\vert_{H}) \, = \, f(x) \cdot \mu((g\circ \lambda_{x})\vert_{H}) \\
		& = \, f(x) \cdot (\Phi_{\mu}g)(x) \, = \, (f \cdot (\Phi_{\mu}g))(x)
\end{align*} for every $x \in G$, i.e., $\Phi_{\mu}(f\cdot g) = f \cdot (\Phi_{\mu}g)$. \end{proof}

\begin{lem}\label{lemma:conditional.expectation} Let $G$ be a topological group, let $H$ be a topological subgroup of~$G$, let $\nu \in \Mean(H)$ be $H$-left-invariant, and let $\mu \in \Mean (G)$ with $\mu\nu = \mu$. Then, \begin{displaymath}
		\forall f \in \RUCB (G) \colon \qquad \! \left. \E_{\hat{\mu}}\left(\overline{f} \, \right\vert \mathrm{S}(\pi_{H})\right) \, = \, \overline{\Phi_{\nu}f} \, \text{ $\hat{\mu}$-almost everywhere} .
\end{displaymath} \end{lem}

\begin{proof} Let $f \in \RUCB(G)$. Being a member of $\Cont(\Samuel(G))$, the function $\overline{\Phi_{\nu}f}$ is integrable with respect to the Borel probability measure $\hat{\mu}$. By Lemma~\ref{lemma:invariance} and Remark~\ref{remark:quotient.spaces}, there is $f' \in \UCB (G/H)$ such that $\Phi_{\nu}f = f' \circ \pi_{H}$, whence \begin{displaymath}
	\Phi_{\nu}f \, = \, f' \circ \pi_{H} \, = \, \overline{f'} \circ \eta_{G/H} \circ \pi_{H} \, \stackrel{\ref{remark:natural}}{=} \, \overline{f'} \circ \Samuel(\pi_{H}) \circ \eta_{G} ,
\end{displaymath} i.e., $\overline{\Phi_{\nu}f} = \overline{f'} \circ \Samuel(\pi_{H})$. As $\overline{f'} \in \Cont(\Samuel(G/H))$, this entails that $\overline{\Phi_{\nu}f}$ is measurable with respect to the pull-back of the Borel $\sigma$-algebra of $\Samuel(G/H)$ along $\Samuel(\pi_{H})$. It thus remains to verify that \begin{equation}\tag{$\ast$}\label{game}
	\begin{split}
		&\forall B \subseteq \Samuel(G/H) \text{ Borel } \, \forall \epsilon \in \R_{> 0} \colon \\
		& \qquad \qquad \quad \left\lvert \int \chi_{\Samuel(\pi_{H})^{-1}(B)} \cdot \overline{f} \, \mathrm{d}\hat{\mu} - \int \chi_{\Samuel(\pi_{H})^{-1}(B)} \cdot \overline{\Phi_{\nu}f} \, \mathrm{d}\hat{\mu} \right\rvert \, \leq \, \epsilon .
	\end{split}
\end{equation} To this end, consider a Borel subset $B \subseteq \Samuel(G/H)$ and let $\epsilon \in \R_{>0}$. Since $\Samuel(\pi_{H})$ is a continuous map between compact Hausdorff spaces, regularity of $\hat{\mu}$ implies regularity of the push-forward Borel probability measure $\Samuel(\pi_{H})_{\ast}(\hat{\mu})$ on $\Samuel(G/H)$. Combining the regularity of $\Samuel(\pi_{H})_{\ast}(\hat{\mu})$ with a standard application of Urysohn's lemma, we find $h \in \Cont(\Samuel(G/H))$ such that \begin{displaymath}
	\int \vert \chi_{B} - h \vert \, \mathrm{d} \Samuel(\pi_{H})_{\ast}(\hat{\mu}) \, \leq \, \tfrac{\epsilon}{2\Vert f \Vert_{\infty} + 1} .
\end{displaymath} Consider $h' \defeq h \circ \Samuel(\pi_{H}) \circ \eta_{G} \in \RUCB(G)$ and note that $\overline{h'} = h \circ \Samuel(\pi_{H})$. As \begin{displaymath}
	h' \, = \, h \circ \Samuel(\pi_{H}) \circ \eta_{G} \, \stackrel{\ref{remark:natural}}{=} \, h \circ \eta_{G/H} \circ \pi_{H} ,
\end{displaymath} the function $h'$ is $H$-right-invariant by Remark~\ref{remark:quotient.spaces}. We conclude that \begin{align*}
	&\int (h \circ \Samuel(\pi_{H})) \cdot \overline{f} \, \mathrm{d} \hat{\mu} \, = \, \int \overline{h'\cdot f} \, \mathrm{d} \hat{\mu} \, = \, \mu\left( \overline{h'\cdot f} \circ {\eta_{G}} \right) \, = \, \mu (h'\cdot f) \, = \, (\mu\nu)(h'\cdot f) \\
		& \quad \qquad = \, \mu (\Phi_{\nu}(h'\cdot f)) \, \stackrel{\ref{lemma:expectation}}{=} \, \mu (h'\cdot (\Phi_{\nu}f)) \, = \, \mu \left((h \circ \Samuel(\pi_{H}) \circ \eta_{G}) \cdot \left(\overline{\Phi_{\nu}f} \circ {\eta_{G}} \right)\right) \\
		& \quad \qquad = \, \mu \left( \left((h \circ \Samuel(\pi_{H})) \cdot \overline{\Phi_{\nu}f}\right) \circ \eta_{G} \right) \, = \, \int (h \circ \Samuel(\pi_{H})) \cdot \overline{\Phi_{\nu}f} \, \mathrm{d} \hat{\mu} ,
\end{align*} and therefore \begin{align*}
	&\left\lvert \int \chi_{\Samuel(\pi_{H})^{-1}(B)} \cdot \overline{f} \, \mathrm{d} \hat{\mu} - \int \chi_{\Samuel(\pi_{H})^{-1}(B)} \cdot \overline{\Phi_{\nu}f} \, \mathrm{d} \hat{\mu} \right\rvert \\
	& \qquad \quad \leq \, \left\lvert \int \chi_{\Samuel(\pi_{H})^{-1}(B)} \cdot \overline{f} \, \mathrm{d} \hat{\mu} - \int  (h \circ \Samuel(\pi_{H})) \cdot \overline{f} \, \mathrm{d} \hat{\mu} \right\rvert \\
	& \qquad \qquad \qquad + \left\lvert \int (h \circ \Samuel(\pi_{H})) \cdot \overline{\Phi_{\nu}f} \, \mathrm{d} \hat{\mu} - \int \chi_{\Samuel(\pi_{H})^{-1}(B)} \cdot \overline{\Phi_{\nu}f} \, \mathrm{d} \hat{\mu} \right\rvert \\
	& \qquad \quad \leq \, \left( \left\lVert \overline{f} \right\rVert_{\infty} + \left\lVert \overline{\Phi_{\nu}f}\right\rVert_{\infty} \right) \cdot \int \left\lvert \chi_{\Samuel(\pi_{H})^{-1}(B)} - (h \circ \Samuel(\pi_{H})) \right\rvert \, \mathrm{d} \hat{\mu} \\
	& \qquad \quad = \, \left( \Vert f \Vert_{\infty} + \Vert \Phi_{\nu}f\Vert_{\infty} \right) \cdot \int \vert \chi_{B} - h \vert \, \mathrm{d} \Samuel(\pi_{H})_{\ast}(\hat{\mu}) \\
	& \qquad \quad \leq \, 2\Vert f \Vert_{\infty} \cdot \int \vert \chi_{B} - h \vert \, \mathrm{d} \Samuel(\pi_{H})_{\ast}(\hat{\mu}) \, \leq \, \epsilon .
\end{align*} This proves~\eqref{game} and hence completes the argument. \end{proof}

The quantitative statement of Theorem~\ref{theorem:main.concentration} requires some additional concepts (Definition~\ref{definition:amenable.length}). To clarify the relevant notation, let $d$ be a continuous right-invariant pseudo-metric on a topological group $G$. For every $g \in G$, \begin{displaymath}
	d^{g} \colon \, G \times G \, \longrightarrow \, \R_{\geq 0}, \quad (x,y) \, \longmapsto \, d(gx,gx)
\end{displaymath} is a continuous right-invariant pseudo-metric on $G$, too. Moreover, if $H$ is any subgroup of $G$, then \begin{displaymath}
	d_{G/H} \colon \, G/H \times G/H \, \longrightarrow \, \R_{\geq 0}, \quad (xH,yH) \, \longmapsto \, \inf\nolimits_{h \in H} d(x,yh)
\end{displaymath} constitutes a well-defined pseudo-metric on $G/H = \{ xH \mid x \in G \}$.

\begin{lem}\label{lemma:diameter} Let $G$ be a topological group and let $G_{0} \leq G_{1} \leq G$. Furthermore, let $d$ be a continuous right-invariant pseudo-metric on $G$ and let $f \in \Lip_{1}^{\infty}(G,d)$ be $G_{0}$-right-invariant. Then, \begin{displaymath}
	\sup \{ \vert f(x) - f(xy) \vert \mid x \in G, \, y \in G_{1} \} \, \leq \, \sup\nolimits_{g \in G}\diam \left( G_{1}/G_{0},d^{g}_{G/G_{0}}\right) .
\end{displaymath} In particular, if $\mu \in \Mean (G_{1})$, then \begin{displaymath}
	\left\lVert f - \Phi_{\mu}f \right\rVert_{\infty} \, \leq \, \sup\nolimits_{g \in G}\diam \left( G_{1}/G_{0},d^{g}_{G/G_{0}}\right) .
\end{displaymath} \end{lem}

\begin{proof} First, if $\epsilon \in \R_{>0}$, $x \in G$ and $y \in G_{1}$, then there exists $z \in G_{0}$ such that \begin{displaymath}
	d(x,xyz) \, \leq \, \sup\nolimits_{g \in G} \diam \left( G_{1}/G_{0},d^{g}_{G/G_{0}}\right) + \epsilon ,
\end{displaymath} whence $f$ being a $G_{0}$-right-invariant element of $\Lip_{1}^{\infty}(G,d)$ implies that \begin{align*}
	\vert f(x) - f(xy) \vert \, &= \, \vert f(x) - f(xyz) \vert \, \leq \, d(x,xyz) \, \leq \, \sup\nolimits_{g \in G} \diam \left( G_{1}/G_{0},d^{g}_{G/G_{0}}\right) + \epsilon .
\end{align*} Thus, as desired, \begin{equation}\tag{$\ast$}\label{difference}
	\sup \{ \vert f(x) - f(xy) \vert \mid x \in G, \, y \in G_{1} \} \, \leq \, \sup\nolimits_{g \in G}\diam \left( G_{1}/G_{0},d^{g}_{G/G_{0}}\right) .
\end{equation} In particular, if $\mu \in \Mean (G_{1})$, then \begin{align*}
	\left\lVert f - \Phi_{\mu}f \right\rVert_{\infty} \, &= \, \sup\nolimits_{x \in G} \left\lvert f(x) - \mu\left((f \circ \lambda_{x})\vert_{G_{1}}\right) \right\rvert \\
		& = \, \sup\nolimits_{x \in G} \left\lvert \mu\left((f(x) - (f \circ \lambda_{x}))\vert_{G_{1}}\right) \right\rvert \\
		& \leq \, \sup\nolimits_{x \in G} \left\lVert (f(x) - (f \circ \lambda_{x}))\vert_{G_{1}} \right\rVert_{\infty} \\
		& = \, \sup \{ \vert f(x) - f(xy) \vert \mid x \in G, \, y \in G_{1} \} \\
		& \stackrel{\eqref{difference}}{\leq} \, \sup\nolimits_{g \in G}\diam \left( G_{i+1}/G_{i},d^{g}_{G/G_{i}}\right) . \qedhere
\end{align*} \end{proof}

\begin{remark}\label{remark:biinvariant} Let $G$ be a topological group and let $G_{0} \leq G_{1} \leq G$. If $d$ is a continuous bi-invariant pseudo-metric on $G$, then \begin{displaymath}
	\sup\nolimits_{g \in G}\diam \left( G_{1}/G_{0},d^{g}_{G/G_{0}}\right) \, = \, \diam \left( G_{1}/G_{0},d_{G/G_{0}}\right) .
\end{displaymath} \end{remark}

The following concept is crucially inspired by the notion of length of a finite metric space by Schechtman~\cite{schechtman82} (see also~\cite[I, Definition~7.7]{MilmanSchechtman}) and its generalization for metric measure spaces by Pestov~\cite[Definition~4.3.16]{PestovBook}.

\begin{definition}\label{definition:amenable.length} Let $G$ be a topological group. Let $\mathscr{A}(G)$ denote the set of all amenable topological subgroups of $G$ and let \begin{displaymath}
	\mathscr{A}^{\ast}(G) \, \defeq \, \bigcup\nolimits_{n \in \N } \! \left. \left\{ (G_{0},\ldots,G_{n}) \in \mathscr{A}(G)^{n+1} \right\vert \{ e \} = G_{0} \leq \ldots \leq G_{n} = G \right\} .
\end{displaymath}  Let $\Delta (G)$ denote the set of all bounded, continuous, right-invariant pseudo-metrics on $G$. For every $d \in \Delta(G)$ and every $(G_{0},\ldots,G_{n}) \in \mathscr{A}^{\ast}(G)$, we define \begin{displaymath}
	\ell(G_{0},\ldots,G_{n}; d) \, \defeq \, \left(\sum\nolimits_{i=0}^{n-1} \left(\sup\nolimits_{g \in G} \diam \left(G_{i+1}/G_{i},d^{g}_{G/G_{i}}\right) \right)^{2} \right)^{1/2} .
\end{displaymath} If $d \in \Delta(G)$, then we will refer to \begin{displaymath}
	\ell(G,d) \, \defeq \, \inf \{ \ell(G_{0},\ldots,G_{n};d) \mid (G_{0},\ldots,G_{n}) \in \mathscr{A}^{\ast}(G) \} \, \in \, [0,\infty]
\end{displaymath} as the \emph{amenable length} of $(G,d)$. \end{definition}

\begin{remark}\label{remark:approximate.amenable.length} Let $G$ be a topological group. Then the following hold. \begin{itemize}
	\item[$(1)$] If $G$ is amenable, then $\ell(G,d) \leq \diam (G,d) < \infty$ for every $d \in \Delta (G)$. Conversely, if $\ell(G,d) < \infty$ for some $d \in \Delta (G)$, then $G$ is necessarily amenable, due to a formal triviality.
	\item[$(2)$] $\ell(G,t\cdot d) = t \cdot \ell(G,d)$ for all $d \in \Delta (G)$ and $t \in \R_{>0}$.
\end{itemize} \end{remark}

Everything is in place for the desired concentration inequality.

\begin{thm}\label{theorem:main.concentration} Let $G$ be an amenable topological group, $(G_{0},\ldots,G_{n}) \in \mathscr{A}^{\ast}(G)$. For each $i \in \{ 1,\ldots,n\}$, let $\nu_{i} \in \Mean (G_{i})$ be $G_{i}$-left-invariant, and define \begin{displaymath}
	\mu \, \defeq \, \nu_{n}\cdots \nu_{1} \, \in \, \Mean (G) .
\end{displaymath} For all $d \in \Delta(G)$, $f \in \Lip_{1}(G,d)$, and $\epsilon \in \R_{>0}$, \begin{displaymath}
	\hat{\mu}\left(\left\{ \xi \in \Samuel (G) \left\vert \, \left\lvert \xi(f) - \mu(f) \right\rvert \geq \epsilon \right\}\right) \right. \! \, \leq \, 2 \exp \left( -\tfrac{\epsilon^{2}}{2 \ell(G_{0},\ldots,G_{n};d)^{2}} \right) .
\end{displaymath}\end{thm}

\begin{proof} For convenience of notation, we put $\nu_{0} \defeq \eta_{G_{0}}(e) \in \Mean (G_{0}) \subseteq \Mean (G)$. For each $i \in \{ 1,\ldots,n \}$, by Remark~\ref{remark:ideal.of.invariant.means}, the mean \begin{displaymath}
	\mu_{i} \, \defeq \, \nu_{i}\cdots \nu_{0} \, \in \, \Mean (G_{i})
\end{displaymath} is $G_{i}$-left-invariant. In particular, $\mu = \mu_{n} \in \Mean (G)$ is $G$-left-invariant, and \begin{displaymath}
	\mu\mu_{i} \, = \, \nu_{n}\cdots \nu_{1}\nu_{i}\cdots \nu_{0} \, \stackrel{\ref{lemma:invariance}}{=} \, \nu_{n}\cdots \nu_{i+1}\nu_{i}\cdots \nu_{0} \, = \, \mu
\end{displaymath} for each $i \in \{ 0,\ldots,n\}$. Moreover, let $\mathscr{B}$ denote the Borel $\sigma$-algebra of $\Samuel (G)$ and, for each $i \in \{ 0,\ldots,n\}$, let $\mathscr{B}_{i}$ denote the pull-back of the Borel $\sigma$-algebra of $\Samuel(G/G_{i})$ along $\Samuel(\pi_{G_{i}})$. Note that \begin{displaymath}
	\{ \emptyset,\Samuel (G)\} \, = \, \mathscr{B}_{n} \, \subseteq \, \ldots \, \subseteq \, \mathscr{B}_{0} \, = \, \mathscr{B} .
\end{displaymath} Now, let $d \in \Delta (G)$ and note that $\Lip_{1}(G,d) = \Lip_{1}^{\infty}(G,d) \subseteq \RUCB(G)$ according to Remark~\ref{remark:lipschitz.right.uniformly.continuous}. Consider any $f \in \Lip_{1}(G,d)$. For every $i \in \{ 0,\ldots,n \}$, the function $f_{i} \defeq \Phi_{\mu_{i}}f$ belongs to $\Lip_{1}(G,d)$ thanks to Lemma~\ref{lemma:lipschitz} and is $G_{i}$-right-invariant by Lemma~\ref{lemma:invariance}. For every $i \in \{ 0,\ldots,n-1\}$, \begin{displaymath}
	\Phi_{\nu_{i+1}}(f_{i}) \, = \, \Phi_{\nu_{i+1}}(\Phi_{\mu_{i}}f) \, \stackrel{\ref{remark:convolution}(1)}{=} \, \Phi_{\nu_{i+1}\mu_{i}}f \, = \, \Phi_{\mu_{i+1}}f \, = \, f_{i+1}
\end{displaymath} and thus \begin{align*}
	\left\lVert \left. \E_{\hat{\mu}}\left(\overline{f} \, \right\vert \mathscr{B}_{i} \right) - \left. \E_{\hat{\mu}}\left(\overline{f} \, \right\vert \mathscr{B}_{i+1} \right) \right\rVert_{\hat{\mu},\infty} \, &= \, \left\lVert \left. \E_{\hat{\mu}}\left(\overline{f} \, \right\vert \mathrm{S}(\pi_{G_{i}})\right) - \left. \E_{\hat{\mu}}\left(\overline{f} \, \right\vert \mathrm{S}(\pi_{G_{i+1}})\right) \right\rVert_{\hat{\mu},\infty} \\
	& \hspace{-40mm} \stackrel{\ref{lemma:conditional.expectation}}{=} \, \left\lVert \overline{f_{i}} - \overline{f_{i+1}} \right\rVert_{\hat{\mu},\infty} \, \leq \, \left\lVert \overline{f_{i}} - \overline{f_{i+1}} \right\rVert_{\infty} \, = \, \left\lVert f_{i} - f_{i+1} \right\rVert_{\infty} \, = \, \left\lVert f_{i} - \Phi_{\nu_{i+1}}(f_{i}) \right\rVert_{\infty} \\
	& \hspace{-40mm} \stackrel{\ref{lemma:diameter}}{\leq} \, \sup\nolimits_{g \in G}\diam \left( G_{i+1}/G_{i},d^{g}_{G/G_{i}}\right) .
\end{align*} Consequently, for every $\epsilon \in \R_{>0}$, \begin{align*}
	\hat{\mu}\left(\left\{ \xi \in \Samuel (G) \left\vert \, \left\lvert \xi(f) - \mu(f) \right\rvert \geq \epsilon \right\}\right) \right. \! \, &= \, \hat{\mu}\left(\left\{ \xi \in \Samuel (G) \left\vert \, \left\lvert \overline{f}(\xi) - \E_{\hat{\mu}}\bigl(\overline{f}\bigr) \right\rvert \geq \epsilon \right\}\right) \right. \\
		& \stackrel{\ref{theorem:azuma}}{\leq} \, 2 \exp \left( -\tfrac{\epsilon^{2}}{2 \ell(G_{0},\ldots,G_{n};d)^{2}} \right) . \qedhere
\end{align*} \end{proof}

A topological group $G$ is said to be \emph{precompact} if, for every $U \in \mathscr{U}(G)$, there exists a finite subset $F \subseteq G$ such that $G = UF$. It is well known that, if $G$ is a precompact topological group, then $\Mean (G)$ contains a unique $G$-left-invariant mean (see~\cite[paragraph preceding Theorem~3.3]{mpu}). From this fact and our Theorem~\ref{theorem:main.concentration}, we readily deduce the following corollary, which---in view of the simplification suggested by Remark~\ref{remark:biinvariant}---entails a concentration inequality for bi-invariant metrics on compact groups by Milman and Schechtman~\cite[I, Theorem~7.12(i)]{MilmanSchechtman} (see also~\cite[Theorem~4.5.3]{PestovBook}).

\begin{cor}[\cite{MilmanSchechtman}, I, Theorem~7.12(i)]\label{corollary:precompact} Consider a chain of precompact topological groups $\{ e \} = G_{0} \leq \ldots \leq \, G_{n} = G$. Let $d$ be a continuous right-invariant pseudo-metric on $G$ and let $\mu \in \Mean (G)$ be the unique $G$-left-invariant mean. For every $f \in \Lip_{1}(G,d)$ and every $\epsilon \in \R_{>0}$, \begin{displaymath}
	\hat{\mu}\left(\left\{ \xi \in \Samuel (G) \left\vert \, \left\lvert \xi(f) - \mu(f) \right\rvert \geq \epsilon \right\}\right) \right. \! \, \leq \, 2 \exp \left( -\tfrac{\epsilon^{2}}{2 \ell(G_{0},\ldots,G_{n};d)^{2}} \right) .
\end{displaymath} \end{cor}

\begin{proof} For each $i \in \{ 1,\ldots,n\}$, we let $\mu_{i}$ denote the unique $G_{i}$-left-invariant element of $\Mean (G_{i})$. Since $\mu_{n}\cdots \mu_{1}$ is $G$-left-invariant by Remark~\ref{remark:ideal.of.invariant.means}, it follows that $\mu = \mu_{n}\cdots \mu_{1}$. Therefore, the desired conclusion is an immediate consequence of Theorem~\ref{theorem:main.concentration}. \end{proof}

Theorem~\ref{theorem:main.concentration} provides a method for proving extreme amenability. For this particular purpose, vanishing of the following slightly technical relative of the amenable length functional (Definition~\ref{definition:amenable.length}) will be sufficient.

\begin{definition}\label{definition:approximate.amenable.length} Let $G$ be a topological group. If $U \in \mathscr{U}(G)$ and $E \in \Pfin (G)$, then we define \begin{displaymath}
	\T_{G}(U,E) \, \defeq \, \{ (H,d) \mid H \in \mathscr{A}(G), \, d \in \Delta(H), \, E \subseteq UH, \, \B_{d}(e,1) \subseteq U \} .
\end{displaymath} Moreover, we let \begin{displaymath}
	\left. \ell(G) \, \defeq \, \sup \left\{  \inf\nolimits_{(H,d) \in \T_{G}(U,E)} \ell(H,d) \, \right\vert U \in \mathscr{U}(G), \, E \in \Pfin(G) \right\} \, \in \, [0,\infty] .
\end{displaymath} \end{definition}

\begin{remark}\label{remark:finite.amenable.length} Let $G$ be a topological group. \begin{itemize}
	\item[$(1)$] If $G$ is amenable, then $\ell(G) \leq 1$. Indeed, for every $U \in \mathscr{U}(G)$, there is $d \in \Delta(G)$ with $\B_{d}(e,1) \subseteq U$ and $\diam (G,d) \leq 1$~\cite[Theorem~P.3]{PachlBook}, whence amenability of $G$ entails that $(G,d) \in \T_{G}(U,E)$ and thus \begin{displaymath}
			\qquad \inf\nolimits_{(H,d') \in \T_{G}(U,E)} \ell(H,d') \, \leq \, \ell(G,d) \, \stackrel{\ref{remark:approximate.amenable.length}(1)}{=} \, \diam (G,d) \, \leq \, 1
		\end{displaymath} for each $E \in \Pfin(G)$, as desired.
	\item[$(2)$] If $\ell(G) < \infty$, then $G$ is amenable, since the former entails that \begin{displaymath}
			\qquad \forall U \in \mathscr{U}(G) \ \forall E \in \Pfin(G) \ \exists H \in \mathscr{A}(G) \colon \quad E \subseteq UH ,
		\end{displaymath} which implies amenability of $G$ by a straightforward argument using weak-$\ast$ compactness of $\Mean (G)$.
	\item[$(3)$] It follows by~(1) and~(2) that $\ell(G) \in [0,1] \cup \{ \infty\}$, and that $\ell(G) < \infty$ if and only if $G$ is amenable.
	\item[$(4)$] If $G$ is non-trivial and has no small subgroups, then $\ell(G) \geq 1$. In~fact, if $U \in \mathscr{U}(G)$ does not contain any non-trivial subgroup of $G$, then, for every $d \in \Delta(G)$ with $\B_{d}(e,1) \subseteq U$ and any chain of subgroups $\{ e \} = G_{0} \leq \ldots \leq G_{n} \leq G$ with $G_{n} \ne \{ e \}$, it follows that \begin{align*}
			\qquad \qquad &\left(\sum\nolimits_{i=0}^{n-1} \left(\sup\nolimits_{g \in G} \diam \left(G_{i+1}/G_{i},d^{g}_{i}\right) \right)^{2} \right)^{1/2} \\
			&\qquad \qquad \geq \, \diam \left(G_{j+1}/G_{j},d_{G/G_{j}}\right) \, = \, \diam \left(G_{j+1},d\right) \, \geq \, 1 ,
		\end{align*} where $j \defeq \max \{ i \in \{ 0,\ldots,n\} \mid G_{i}=\{ e \} \}$, which implies that $\ell(G)\geq 1$, due to non-triviality of $G$.
\end{itemize} \end{remark}

\begin{thm}\label{theorem:extreme.amenability} Let $G$ be a topological group such that $\ell (G) = 0$. Then $G$ is extremely amenable. \end{thm}

\begin{proof} We are going to verify condition~(2) of Proposition~\ref{proposition:fremlin}. To this end, let $F \in \Pfin(\RUCB(G))$, $E \in \Pfin (G)$ and $\epsilon \in (0,1]$. Then we find $U \in \mathscr{U}(G)$ such that $\Vert f - (f\circ\lambda_{u}) \Vert_{\infty} \leq \tfrac{\epsilon }{3}$ for all $u \in U$ and $f \in F$. Put $s \defeq \sup\nolimits_{f \in F} \Vert f \Vert_{\infty} +1$ and consider \begin{displaymath}
	\ell \, \defeq \, \tfrac{\epsilon}{s\sqrt{72\ln (2/ \epsilon)}} \, \in \, \R_{> 0} .
\end{displaymath} Since $\ell (G) = 0$, there is $(H,d) \in \T_{G}(U,E)$ with $\ell (H,d) \leq \ell$. By Theorem~\ref{theorem:main.concentration} and Remark~\ref{remark:ideal.of.invariant.means}, there exists an $H$-left-invariant mean $\mu \in \Mean (H)$ such that \begin{equation}\tag{$\ast$}\label{submean}
	\sup\nolimits_{f \in \Lip_{1}(H,d)} \hat{\mu}\left(\left\{ \xi \in \Samuel (H) \left\vert \, \left\lvert \xi(f) - \mu(f) \right\rvert \geq \tfrac{\epsilon}{6s} \right\}\right) \right. \! \, \leq \, 2 \exp \left( -\tfrac{(\epsilon /6s)^{2}}{2 \ell^{2}} \right) \, = \, \epsilon .
\end{equation} We now claim that, for every $f \in F$, \begin{equation}\tag{$\ast \ast$}\label{submean.claim}
	\sup\nolimits_{g \in E} \left\lvert \mu(f) - \mu(f \circ \lambda_{g}) \right\rvert \, \leq \, \epsilon , \quad \hat{\mu}(\{ \xi \in \Samuel (G) \mid \vert \xi(f) - \mu(f) \vert \geq \epsilon \}) \, \leq \, \epsilon .
\end{equation} To this end, let $f \in F$. First, if $g \in E$, then there exist $u \in U$ and $h \in H$ such that $g=uh$, whence $\mu$ being $H$-left-invariant implies that \begin{align*}
	\left\lvert \mu(f) - \mu(f \circ \lambda_{g}) \right\rvert \, &= \, \left\lvert \mu(f) - \mu(f \circ \lambda_{u} \circ \lambda_{h}) \right\rvert \, = \, \left\lvert \mu(f) - \mu(f \circ \lambda_{u}) \right\rvert \\
	& = \, \left\lvert \mu(f - (f \circ \lambda_{u})) \right\rvert  \, \leq \, \left\lVert f - (f \circ \lambda_{u}) \right\rVert_{\infty} \, \leq \, \tfrac{\epsilon}{3} \, \leq \, \epsilon .
\end{align*} Furthermore, our choices of $U$ and $d$ entail that \begin{align*}
	\vert f(x) - f(y) \vert \, = \, \left\lvert f(x) - f\left(yx^{-1}x\right) \right\rvert \, &\leq \, \max \left\{ \tfrac{\epsilon }{3}, \, 2sd\left(e,yx^{-1}\right) \right\} \\
	& \leq \, 2sd\left(e,yx^{-1}\right) + \tfrac{\epsilon }{3} \, = \, 2sd(x,y) + \tfrac{\epsilon }{3}
\end{align*} for all $x,y \in H$. Therefore, Remark~\ref{remark:lipschitz.approximation} asserts the existence of some function $f' \in \Lip_{2s}(H,d)$ such that $\Vert f\vert_{H} - f' \Vert_{\infty} \leq \tfrac{\epsilon }{3}$. We conclude that \begin{align*}
	&\hat{\mu}\left(\left\{ \xi \in \Samuel (G) \left\vert \, \left\lvert \xi(f) - \mu(f) \right\rvert \geq \epsilon \right\}\right) \right. \! \, \stackrel{\ref{lemma:subspace.mean}+\ref{remark:convention}}{=} \, \hat{\mu}\left(\left\{ \xi \in \Samuel (H) \left\vert \, \left\lvert \xi(f\vert_{H}) - \mu(f\vert_{H}) \right\rvert \geq \epsilon \right\}\right) \right. \! \\
	&\qquad \leq \hat{\mu}\left(\left\{ \xi \in \Samuel (H) \left\vert \, \left\lvert \xi(f') - \mu(f') \right\rvert \geq \tfrac{\epsilon}{3} \right\}\right) \right. \! \\
	&\qquad = \hat{\mu}\left(\left\{ \xi \in \Samuel (H) \left\vert \, \left\lvert \xi\left((2s)^{-1}f'\right) - \mu\left((2s)^{-1}f'\right) \right\rvert \geq \tfrac{\epsilon}{6s} \right\}\right) \right. \! \, \stackrel{\eqref{submean}}{\leq} \, \epsilon . 
\end{align*} This proves~\eqref{submean.claim} and hence completes the argument by Proposition~\ref{proposition:fremlin}. \end{proof}

For the sake of completeness, let us note that the proof of Theorem~\ref{theorem:extreme.amenability} provides a construction of an asymptotically left-invariant net of means concentrating with respect to the right uniformity (in the sense of Definition~\ref{definition:concentration}) on any topological group $G$ with $\ell(G) = 0$. We continue with a few consequences of Theorem~\ref{theorem:extreme.amenability}.

\begin{cor}\label{corollary:extreme.amenability.1} Let $G$ be a topological group. Suppose that $G$ admits a bounded compatible right-invariant metric $d$ as well as a set $\mathscr{H}$ of amenable topological subgroups such that \begin{itemize}
	\item[$(1)$] $(\mathscr{H},{\subseteq})$ is directed,
	\item[$(2)$] $G = \overline{\bigcup \mathscr{H}}$,
	\item[$(3)$] $\liminf_{H \in \mathscr{H}} \ell (H,{d\vert_{H \times H}}) = 0$.
\end{itemize} Then $\ell(G) = 0$. In particular, $G$ is extremely amenable. \end{cor}

\begin{proof} Consider any $\epsilon \in \R_{>0}$, $U \in \mathscr{U}(G)$ and $E \in \Pfin (G)$. Since $d$ generates the topology of $G$, there exists $t \in \R_{\geq 1}$ such that $U \supseteq \B_{d}(e,1/t) = \B_{t\cdot d}(e,1)$. Thanks to (1)--(3), we find $H \in \mathscr{H}$ such that $E \subseteq UH$ and $\ell (H,{d\vert_{H \times H}}) \leq \tfrac{\epsilon}{t}$. Hence, $(H,t\cdot d\vert_{H \times H}) \in \T_{G}(U,E)$ and therefore \begin{displaymath}
	\inf\nolimits_{(H',d') \in \T_{G}(U,E)} \ell(H',d') \, \leq \, \ell(H,{t\cdot d\vert_{H \times H}}) \, \stackrel{\ref{remark:approximate.amenable.length}(2)}{=} \, t\cdot\ell(H,{d\vert_{H \times H}}) \, \leq \, \epsilon .
\end{displaymath} This shows that $\ell (G) = 0$, which entails extreme amenability of $G$ according to Theorem~\ref{theorem:extreme.amenability}. \end{proof}

\begin{cor}\label{corollary:extreme.amenability.2} Let $G$ be a topological group and let $d \in \Delta (G)$ be generating the topology of $G$. If $\ell(G,d) = 0$, then $\ell(G) = 0$, thus $G$ is extremely amenable. \end{cor}

\begin{proof} This follows from Remark~\ref{remark:approximate.amenable.length}(1) and Corollary~\ref{corollary:extreme.amenability.1}. \end{proof}

The following sufficient criterion, which may be viewed as an amplification technique, will be used to prove Theorem~\ref{theorem:stable.groups}.

\begin{prop}\label{proposition:amenable.folding} Let $G$ be a topological group and let $d \in \Delta (G)$. Consider any dense subset $D \subseteq [0,1]$ with $\{0,1\} \subseteq D$. Suppose that there exists a family of continuous endomorphisms $\phi_{t} \colon G \to G$ $(t \in D)$ such that \begin{itemize}
	\item[$(1)$] $\phi_{0}(G) = \{ e \}$ and $\phi_{1}(G) = G$,
	\item[$(2)$] if $s,t \in D$ and $s\leq t$, then $\phi_{s}(G) \subseteq \phi_{t}(G)$, and
	\item[$(3)$] there exists $C \in \R_{\geq 0}$ such that \begin{displaymath}
						\qquad \forall s,t \in D \ \forall g,h \in G \colon \quad d(g\phi_{s}(h),g\phi_{t}(h)) \, \leq \, C\cdot \vert s-t \vert .
					\end{displaymath}
\end{itemize} If $G$ is amenable, then $\ell(G,d) = 0$. \end{prop}

\begin{proof} Suppose that $G$ is amenable. It suffices to show that $\ell(G,d) \leq \epsilon$ for every $\epsilon \in \R_{>0}$. To this end, let $\epsilon \in \R_{>0}$. Let $C \in \R_{\geq 0}$ be as in~(3), and consider \begin{displaymath}
	n \, \defeq \, \max \left\{ \left\lceil \tfrac{16C^{2}}{\epsilon^{2}} \right\rceil , 1 \right\} \, \in \, \N_{>0} .
\end{displaymath} Let $t_{0} \defeq 0 \in D$ and $t_{n} \defeq 1 \in D$. Since $\overline{D} = [0,1]$, for each $i \in \{ 1,\ldots,n-1\}$ we find some $t_{i} \in D \cap \left[\tfrac{i}{n+1},\tfrac{i+1}{n+1}\right]$. Evidently, $t_{0} \leq \ldots \leq t_{n}$ and \begin{equation}\tag{$\ast$}\label{density}
	\forall i \in \{ 0,\ldots,n-1\} \colon \qquad \vert t_{i+1} - t_{i} \vert \, \leq \, \tfrac{2}{n+1} .
\end{equation} For each $i \in \{ 0,\ldots,n\}$, the topological subgroup \begin{displaymath}
	G_{i} \, \defeq \, \phi_{t_{i}}(G) \, \leq \, G 
\end{displaymath} is the image of an amenable topological group under a continuous homomorphism, thus amenable by~\cite[Theorem~4.6]{rickert}. By~(1) and~(2), \begin{displaymath}
	\{ e \} \, = \, G_{0} \, \leq \, G_{1} \, \leq \, \ldots \, \leq \, G_{n-1} \, \leq \, G_{n} \, = \, G .
\end{displaymath} Thus, $(G_{0},\ldots,G_{n}) \in \mathscr{A}^{\ast}(G)$. Furthermore, if $i \in \{ 0,\ldots,n-1\}$, then \begin{align*}
	&d^{g}_{G/G_{i}}(\phi_{t_{i+1}}(x)G_{i},\phi_{t_{i+1}}(y)G_{i}) \, \leq \, d^{g}_{G/G_{i}}(\phi_{t_{i+1}}(x)G_{i},G_{i}) + d^{g}_{G/G_{i}}(G_{i},\phi_{t_{i+1}}(y)G_{i}) \\
	&\qquad \leq \, d(g\phi_{t_{i+1}}(x),g\phi_{t_{i}}(x)) + d(g\phi_{t_{i}}(y),g\phi_{t_{i+1}}(y)) \, \stackrel{(3)}{\leq} \, 2C\vert t_{i+1} - t_{i} \vert \, \stackrel{\eqref{density}}{\leq} \, \tfrac{4C}{n+1}
\end{align*} for all $g,x,y \in G$, whence \begin{displaymath}
	\sup\nolimits_{g \in G} \diam \left(G_{i+1}/G_{i},d^{g}_{G/G_{i}}\right) \, \leq \, \tfrac{4C}{n+1} .
\end{displaymath} Consequently, \begin{displaymath}
	\ell(G_{0},\ldots,G_{n}; d) \, = \, \left(\sum\nolimits_{i=0}^{n-1} \left(\sup\nolimits_{g \in G} \diam \left(G_{i+1}/G_{i},d^{g}_{G/G_{i}}\right)\right)^{2} \right)^{1/2} \! \, \leq \, \tfrac{4C}{\sqrt{n}} \, \leq \, \epsilon 
\end{displaymath} and therefore $\ell(G,d) \leq \epsilon$, as desired. \end{proof}

As a first application of Theorem~\ref{theorem:extreme.amenability} and Proposition~\ref{proposition:amenable.folding}, we recover the extreme amenability of the group of measurable maps with values in an amenable topological group~\cite[Theorem~1.1, (1)$\Longrightarrow$(3)]{PestovSchneider}.

\begin{exmpl}\label{example:pestov.schneider} Let $G$ be a topological group and consider the Lebesgue measure $\lambda$ on the closed real interval $[0,1]$. The set $L^{0}(G)$ of all equivalence classes of $\lambda$-almost continuous\footnote{A mapping $f \colon X \to Y$ from a compact Hausdorff space $X$ to a topological space $Y$ is called \emph{$\mu$-almost continuous}~\cite{Fremlin} with respect to a regular Borel probability measure $\mu$ on $X$ if, for every $\epsilon \in \R_{>0}$, there exists a closed subset $A \subseteq X$ with $\mu (X \setminus A) \leq \epsilon$ such that $f\vert_{A} \colon A \to Y$ is continuous. If the target space is metrizable, then $\mu$-almost continuity is equivalent to $\mu$-measurability~\cite[Theorem~2B]{Fremlin}.} functions from $[0,1]$ to $G$ up to equality $\lambda$-almost everywhere, endowed with the group structure inherited from $G$ and the corresponding topology of convergence in measure with respect to~$\lambda$, constitutes a topological group. For every $n \in \N_{>0}$, the mapping $\psi_{n} \colon G^{n} \to L^{0}(G)$ defined by \begin{displaymath}
	\psi_{n}(g)\vert_{[(i-1)/n,i/n)} \, \equiv \, g_{i} \qquad \left(g \in G^{n}, \, i \in \{ 1,\ldots,n \}\right)
\end{displaymath} is a continuous homomorphism. Due to closure properties of the class of amenable topological groups~\cite[Corollary~4.5, Theorems~4.6, 4.7, 4.8]{rickert}, if $G$ is amenable, then so is \begin{displaymath}
	L^{0}(G) \, = \, \overline{\bigcup \left\{ \! \left. \psi_{2^{n}}\bigl(G^{2^{n}}\bigr) \, \right\vert n \in \N \right\}} .
\end{displaymath} Furthermore, for every $t \in [0,1]$, the map $\phi_{t} \colon L^{0}(G) \to L^{0}(G)$ defined by \begin{displaymath}
	\phi_{t}(f)\vert_{[0,t]} \, = \, f\vert_{[0,t]}, \qquad \phi_{t}(f)\vert_{(t,1]} \, \equiv \, e \qquad \left( f \in L^{0}(G) \right)
\end{displaymath} is a continuous endomorphism. Evidently, $\phi_{0} \equiv e_{L^{0}(G)}$ and $\phi_{1} = \id_{L^{0}(G)}$. Also, $\phi_{s} \circ \phi_{t} = \phi_{s \wedge t}$ for all $s,t \in [0,1]$. In particular, if $s,t \in [0,1]$ and $s \leq t$, then $\phi_{s}(L^{0}(G)) \subseteq \phi_{t}(L^{0}(G))$. Finally, if $d \in \Delta (G)$, then \begin{align*}
	d_{\lambda} \colon \, L^{0}(G) \times L^{0}(G) \, &\longrightarrow \, \R_{\geq 0}, \\
	 (f,g) \, &\longmapsto \, \inf \{ \epsilon \in \R_{>0} \mid \lambda (\{ x \in [0,1] \mid d(f(x),g(x)) > \epsilon \}) \leq \epsilon \}
\end{align*} is a member of $\Delta \left(L^{0}(G)\right)$, and \begin{displaymath}
	d_{\lambda}(g\phi_{s}(f),g\phi_{t}(f)) \, \leq \, \vert s - t \vert
\end{displaymath} for all $s,t \in [0,1]$ and $f,g \in L^{0}(G)$. Thus, if $G$ is amenable, then
\begin{displaymath}
	\ell\left(L^{0}(G),{r^{-1}} \cdot d_{\lambda}\right) \, \stackrel{\ref{remark:approximate.amenable.length}(2)}{=} \, {r^{-1}} \cdot \ell\left(L^{0}(G),d_{\lambda}\right) \, \stackrel{\ref{proposition:amenable.folding}}{=} \, 0
\end{displaymath} for all $d \in \Delta (G)$ and $r \in \R_{>0}$. Since \begin{displaymath}
	\left. \left\{ \B_{d_{\lambda}}(e_{L^{0}(G)},r) \, \right\vert d \in \Delta (G), \, r \in \R_{>0} \right\} \, = \, \left. \! \left\{ \B_{r^{-1}\cdot d_{\lambda}}(e_{L^{0}(G)},1) \, \right\vert d \in \Delta (G), \, r \in \R_{>0} \right\}
\end{displaymath} is a neighborhood basis at the neutral element of $L^{0}(G)$, we conclude the following: if $G$ is amenable, then $\ell(L^{0}(G)) = 0$, wherefore $L^{0}(G)$ is extremely amenable by Theorem~\ref{theorem:extreme.amenability}. \end{exmpl}

We conclude this section by unraveling Definition~\ref{definition:amenable.length} for direct products of topological groups.

\begin{lem}\label{lemma:normal.filtrations} Let $G$ be a topological group and let $G_{0}, G_{1} \unlhd G$ with $G_{0} \subseteq G_{1}$. If $d$ is a continuous right-invariant pseudo-metric on $G$, then \begin{displaymath}
	\sup\nolimits_{g \in G}\diam \left( G_{1}/G_{0},d^{g}_{G/G_{0}}\right) \, = \, \diam \left( G_{1}/G_{0},d_{G/G_{0}}\right) .
\end{displaymath} \end{lem}

\begin{proof} Let $g \in G$. As $G_{0} \unlhd G$ and thus $gG_{0} = G_{0}g$, right invariance of $d$ entails that \begin{equation}\tag{$\ast$}\label{normal}
	\forall x,y \in G \colon \quad d_{G/G_{0}}(xgG_{0},ygG_{0}) \, = \, d_{G/G_{0}}(xG_{0},yG_{0}) .
\end{equation} Since $G_{1} \unlhd G$, if $x,y \in G_{1}$, then $gxg^{-1},gyg^{-1} \in G_{1}$ and therefore \begin{align*}
	d^{g}_{G/G_{0}}(xG_{0},yG_{0}) \, & = \, d_{G/G_{0}}(gxG_{0},gyG_{0}) \, = \, d_{G/G_{0}}\!\left(gxg^{-1}gG_{0},gyg^{-1}gG_{0}\right) \\
		& \stackrel{\eqref{normal}}{=} \,  d_{G/G_{0}}\!\left(gxg^{-1}G_{0},gyg^{-1}G_{0}\right) \, \leq \, \diam \left( G_{1}/G_{0},d_{G/G_{0}}\right) .
\end{align*}  Hence, $\diam \left( G_{1}/G_{0},d^{g}_{G/G_{0}}\right) \leq \diam \left( G_{1}/G_{0},d_{G/G_{0}}\right)$ as desired. \end{proof}

\begin{prop}\label{proposition:concentration.in.products} Let $n \in \N$. For each $i \in \{ 1,\ldots,n\}$, let $d_{i}$ be a right-invariant continuous pseudo-metric on a topological group $G_{i}$. Consider the topological group $G \defeq G_{1} \times \ldots \times G_{n}$ and its right-invariant continuous pseudo-metric \begin{displaymath}
	d \colon \, G \times G \, \longrightarrow \, \R, \quad (x,y) \, \longmapsto \, \sum\nolimits_{i=1}^{n} d_{i}(x_{i},y_{i}) .
\end{displaymath} For each $i \in \{ 0,\ldots,n \}$, consider \begin{displaymath}
	H_{i} \, \defeq \, G_{1} \times \ldots \times G_{i} \times \{ e \} \times \ldots \times \{ e \} \, \leq \, G .
\end{displaymath} Then the following hold. \begin{itemize}
	\item[$(1)$] For every $i \in \{ 0,\ldots,n-1 \}$, \begin{displaymath}
		\qquad \diam \left( H_{i+1}/H_{i},d_{G/H_{i}}\right) \, = \, \diam (G_{i+1},d_{i+1}) .
\end{displaymath}
	\item[$(2)$] If $G_{1},\ldots,G_{n}$ are amenable, then \begin{displaymath}
		\qquad \ell (G,d) \, \leq \, \ell(H_{0},\ldots,H_{n};d) \, = \, \left(\sum\nolimits_{i=1}^{n} \diam (G_{i},d_{i})^{2} \right)^{1/2} .
	\end{displaymath}
\end{itemize} \end{prop}

\begin{proof} (1) Let $i \in \{ 0,\ldots,n-1 \}$. If $x,y \in H_{i+1}$, then \begin{displaymath}
	z \, \defeq \, \left(y_{1}^{-1}x_{1},\ldots,y_{i}^{-1}x_{i},e,\ldots, e\right) \, \in \, H_{i}
\end{displaymath} and \begin{align*}
	d(x,yz) \, &= \, d((x_{1},\ldots,x_{i},x_{i+1},e,\ldots,e),(x_{1},\ldots,x_{i},y_{i+1},e,\ldots,e)) \\
		& = \, d_{i+1}(x_{i+1},y_{i+1}) \, \leq \, \diam (G_{i+1},d_{i+1}) ,
\end{align*} whence \begin{displaymath}
	d_{G/H_{i}}(xH_{i},yH_{i}) \, \leq \, d(x,yz) \, \leq \, \diam (G_{i+1},d_{i+1}) .
\end{displaymath} Conversely, if $x,y \in G_{i+1}$, then \begin{displaymath}
	x' \, \defeq \, (e,\ldots,e,x,e,\ldots,e) \, \in \, H_{i+1}, \qquad y' \, \defeq \, (e,\ldots,e,y,e,\ldots,e) \, \in \, H_{i+1},
\end{displaymath} and \begin{displaymath}
	d_{i+1}(x,y) \, \leq \, d_{i+1}(x,y) + \sum\nolimits_{j=1}^{i} d_{j}(z_{j},z_{j}') \, = \, d(x'z,y'z')
\end{displaymath} for all $z,z' \in H_{i}$, that is, \begin{displaymath}
	d_{i+1}(x,y) \, \leq \, d_{G/H_{i}}(x'H_{i},y'H_{i}) \, \leq \, \diam \left( H_{i+1}/H_{i},d_{G/H_{i}}\right) .
\end{displaymath}

(2) Note that $H_{i} \unlhd G$ for each $i \in \{ 0,\ldots,n\}$. Now, if $G_{1},\ldots,G_{n}$ are amenable topological groups, then so are $H_{0},\ldots,H_{n}$ by~\cite[Theorem~4.8]{rickert}, thus \begin{align*}
	\ell (G,d) \, \leq \, \ell(H_{0},\ldots,H_{n};d) \, &\stackrel{\ref{lemma:normal.filtrations}}{=} \, \left(\sum\nolimits_{i=0}^{n-1} \diam \left(H_{i+1}/H_{i},d_{G/H_{i}}\right)^{2} \right)^{1/2} \\
	& \stackrel{(1)}{=} \, \left(\sum\nolimits_{i=0}^{n-1} \diam (G_{i+1},d_{i+1})^{2} \right)^{1/2} . \qedhere
\end{align*} \end{proof}

\section{Solution to a problem by Pestov}\label{section:pestov}

This entire section is devoted to solving Problem~\ref{problem}. Keeping the notation of Problem~\ref{problem} for a brief moment, let us note that, if the topological groups $G_{n}$ $(n \in \N_{>0})$ are precompact, then the answer to Pestov's question is affirmative: in such case, for each $n \in \N_{>0}$ the topological group $H_{n}$ will be precompact, too, whence indeed \begin{align*}
	\mu_{n}\left( (f_{n} - \mu_{n}(f_{n}))^{2} \right) \, &\stackrel{\ref{lemma:convergence}(2)}{\leq} \, \hat{\mu}_{n}\left(\left\{ \xi \in \Samuel (H_{n}) \left\vert \, \vert \xi(f_{n}) - \mu_{n}(f_{n}) \vert \geq n^{-1/4}\right\} \right) \right. + \tfrac{1}{\sqrt{n}} \\
	&\stackrel{\ref{corollary:precompact}+\ref{proposition:concentration.in.products}(2)}{\leq} \, 2\exp\left( -\tfrac{n}{2\sqrt{n}} \right) + \tfrac{1}{\sqrt{n}} \, = \, 2\exp\left( -\tfrac{\sqrt{n}}{2}\right) + \tfrac{1}{\sqrt{n}} \, \stackrel{n\to\infty}{\longrightarrow} \, 0 
\end{align*} for every sequence $f_{n} \in \Lip_{1}(H_{n},d'_{n})$ $(n \in \N_{>0})$. However, the precompact case turns out to be very particular. As revealed by Corollary~\ref{corollary:pestov.2}, the answer to Pestov's general question is, in fact, negative. 

The core argument for solving Problem~\ref{problem} is contained in the following Theorem~\ref{theorem:pestov}. For the sake of convenience, we set up some additional notation. Let $X$ be a set and let $n \in \N_{>0}$. If $f \in \R^{X}$, then we define \begin{displaymath}
	f_{n,i} \colon \, X^{n} \, \longrightarrow \, \R, \quad x \, \longmapsto \, \tfrac{1}{n} \sum\nolimits_{j=1}^{i} f(x_{j})
\end{displaymath} for each $i \in \{ 0,\ldots,n \}$. If $d$ is a metric on $X$, then we consider the induced \emph{normalized Hamming metric} \begin{displaymath}
	d_{n} \colon \, X^{n} \times X^{n} \, \longrightarrow \, \R_{\geq 0}, \quad (x,y) \, \longmapsto \, \tfrac{1}{n} \sum\nolimits_{i=1}^{n} d(x_{i},y_{i}) ,
\end{displaymath} and note that $\{f_{n,0},\ldots,f_{n,n}\} \subseteq \Lip_{1}(X^{n},d_{n};[0,1])$ for each $f \in \Lip_{1}(X,d;[0,1])$. For any topological group $G$, let $\RUC(G,[0,1]) \defeq \RUCB(G) \cap [0,1]^{G}$.

\begin{thm}\label{theorem:pestov} Let $G$ be a topological group such that $\Samuel(G)$ contains two disjoint closed, $G$-invariant, non-empty subsets. Then there exists $f \in \RUC (G,[0,1])$ such that \begin{displaymath}
	\forall n \in \N_{>0} \, \exists \nu \in \Mean (G^{n}) \, \forall \mu \in \Mean (G^{n}) \colon \quad (\mu\nu)\left( (f_{n,n} - (\mu\nu)(f_{n,n}))^{2} \right) \, = \, \tfrac{1}{4}.
\end{displaymath} \end{thm}

\begin{proof} Fix any pair of closed, $G$-invariant, non-empty subsets $A_{0},A_{1} \subseteq \Samuel (G)$ such that $A_{0} \cap A_{1} = \emptyset$. Thanks to Urysohn's lemma, we find $h \in \Cont (\Samuel (G))$ with $0 \leq h \leq 1$, $h(A_{0}) = \{ 0 \}$ and $h(A_{1}) = \{ 1 \}$. Consider $f \defeq h \circ {\eta_{G}} \in \RUC (G,[0,1])$. Since $\eta_{G}(G)$ is dense in $\Samuel (G)$ and $h(\eta_{G}(x)) = f(x) = \eta_{G}(x)(f)$ for all $x \in G$, \begin{equation}\tag{1}\label{Samuel}
	\forall \xi \in \Samuel (G) \colon \qquad h(\xi) \, = \, \xi (f) .
\end{equation} For each $i \in \{ 0,1 \}$, as $A_{i} \ne \emptyset$, we may pick some $\xi_{i} \in A_{i}$. We observe that \begin{equation}\tag{2}\label{Samuel2}
	\forall i \in \{ 0,1 \} \colon \qquad \Phi_{\xi_{i}}f \, = \, i .
\end{equation} Indeed, if $i \in \{ 0,1\}$, then $A_{i}$ being $G$-invariant ensures that \begin{align*}
	\left(\Phi_{\xi_{i}}f \right)(g) \, & = \, \xi_{i}(f \circ \lambda_{g}) \, = \, (g \xi_{i})(f) \, \stackrel{\eqref{Samuel}}{=} \, h(g\xi_{i}) \, = \, i
\end{align*} for all $g \in G$, that is, $\Phi_{\xi_{i}}f = i$. Now, let $n\in \N_{>0}$. For every $j \in \{ 1,\ldots,n \}$, the map $\psi_{j} \colon G \to G^{n}$ defined by \begin{displaymath}
	\psi_{j}(g)_{k} \, \defeq \, \begin{cases}
										\, g & \text{if } k=j, \\
										\, e & \text{otherwise}
									\end{cases} \qquad (g \in G, \, k \in \{ 1,\ldots,n\})
\end{displaymath} is a continuous homomorphism, thus uniformly continuous with respect to the right uniformities on $G$ and $G^{n}$. It follows that, for each $i \in \{ 0,1 \}$ and every $j \in \{ 1,\ldots,n \}$, \begin{displaymath}
	\xi_{i,j} \, \defeq \, \Samuel (\psi_{j})(\xi_{i}) \colon \, \RUCB(G^{n}) \, \longrightarrow \, \R, \quad f \, \longmapsto \, \xi_{i}(f \circ {\psi_{j}})
\end{displaymath} is a well-defined member of $\Samuel (G^{n})$. We prove that \begin{equation}\tag{3}\label{Samuel3}
	\forall i \in \{ 0,1 \} \, \forall j \in \{ 0,\ldots,n\} \colon \qquad \Phi_{\xi_{i,n-j+1}\cdots \xi_{i,n}}(f_{n,n}) \, = \, f_{n,n-j} + \tfrac{j}{n}i .
\end{equation} To this end, let $i \in \{ 0,1 \}$. The proof proceeds by induction on $j \in \{ 0,\ldots,n\}$. Clearly, if $j = 0$, then \begin{displaymath}
	\Phi_{\xi_{i,n-j+1}\cdots \xi_{i,n}}(f_{n,n}) \, = \, \Phi_{\eta_{G}(e)}(f_{n,n}) \, \stackrel{\ref{remark:convolution}(1)}{=} \, f_{n,n} \, = \, f_{n,n-j} + \tfrac{j}{n}i .
\end{displaymath} For the inductive step, suppose that $\Phi_{\xi_{i,n-j+1}\cdots \xi_{i,n}}(f_{n,n}) \, = \, f_{n,n-j} + \tfrac{j}{n}i$ for some $j \in \{ 0,\ldots,n-1\}$. Then \begin{align*}
	\Phi_{\xi_{i,n-(j+1)+1}\cdots \xi_{i,n}}(f_{n,n}) (g) \, & \stackrel{\ref{remark:convolution}(1)}{=} \, \Phi_{\xi_{i,n-j}}\left( \Phi_{\xi_{i,n-j+1} \cdots \xi_{i,n-j+1}} (f_{n,n}) \right) (g) \\
	& = \, \Phi_{\xi_{i,n-j}}\left( f_{n,n-j} + \tfrac{j}{n}i \right)(g) \, \stackrel{\ref{remark:convolution}(2)}{=} \, \Phi_{\xi_{i,n-j}}(f_{n,n-j})(g) + \tfrac{j}{n}i \\
	& = \, \xi_{i,n-j}\left({f_{n,n-j}} \circ {\lambda_{g}}\right) + \tfrac{j}{n}i \, = \, \xi_{i}\left({f_{n,n-j}} \circ {\lambda_{g}} \circ {\psi_{n-j}}\right) + \tfrac{j}{n}i \\
	& = \, \xi_{i}\left( \tfrac{1}{n} \left( f(g_{1}) + \ldots + f(g_{n-j-1}) + \left(f \circ {\lambda_{g_{n-j}}}\right)\right) \right) + \tfrac{j}{n}i \\
	& = \, \tfrac{1}{n} \left( f(g_{1}) + \ldots + f(g_{n-j-1}) + \xi_{i}\left(f \circ {\lambda_{g_{n-j}}}\right)\right) + \tfrac{j}{n}i \\
	& = \, \tfrac{1}{n} \left( f(g_{1}) + \ldots + f(g_{n-j-1}) + (\Phi_{\xi_{i}}f)(g_{n-j}) \right) + \tfrac{j}{n}i \\
	& \stackrel{\eqref{Samuel2}}{=} \, \tfrac{1}{n} \left( f(g_{1}) + \ldots + f(g_{n-j-1}) + i \right) + \tfrac{j}{n}i \\
	& = \, \tfrac{1}{n} \left( f(g_{1}) + \ldots + f(g_{n-j-1}) \right) + \tfrac{j+1}{n}i \, = \, f_{n,n-(j+1)}(g) + \tfrac{j+1}{n}i 
\end{align*} for every $g \in G^{n}$, i.e., $\Phi_{\xi_{i,n-(j+1)+1}\cdots \xi_{i,n}}(f_{n,n}) \, = \, f_{n,n-(j+1)} + \tfrac{j+1}{n}i$. This completes our induction and hence proves~\eqref{Samuel3}. Finally, for each $i \in \{ 0,1\}$, we consider \begin{displaymath}
	\zeta_{i} \, \defeq \, \xi_{i,1}\cdots \xi_{i,n} \, \in \, \Samuel (G^{n}) 
\end{displaymath} and note that \begin{equation}\tag{4}\label{Samuel4}
	\Phi_{\zeta_{i}}(f_{n,n}) \, = \, \Phi_{\xi_{i,1}\cdots \xi_{i,n}}(f_{n,n}) \, \stackrel{\eqref{Samuel3}}{=} \, f_{n,0} + \tfrac{n}{n}i \, = \, i .
\end{equation} We claim that the mean \begin{displaymath}
	\nu \, \defeq \, \tfrac{1}{2}\left( \zeta_{0} + \zeta_{1} \right) \, \in \, \Mean (G^{n}) 
\end{displaymath} verifies the conclusion of the theorem, i.e., \begin{equation}\tag{5}\label{Samuel5}
	\forall \mu \in \Mean (G^{n}) \colon \qquad (\mu\nu)\left( (f_{n,n} - (\mu\nu)(f_{n,n}))^{2} \right) \, = \, \tfrac{1}{4}.
\end{equation} To prove this, let $\mu \in \Mean (G^{n})$. For each $i \in \{ 0,1 \}$, \begin{displaymath}
	\left(\mu \zeta_{i}\right)(f_{n,n}) \, = \, \mu\left( \Phi_{\zeta_{i}}(f_{n,n})\right) \, \stackrel{\eqref{Samuel4}}{=} \, \mu(i) \, = \, i .
\end{displaymath} Consequently, \begin{equation}\tag{6}\label{Samuel6}
	(\mu \nu)(f_{n,n}) \, = \, \left( \tfrac{1}{2}\left( \mu \zeta_{0} + \mu \zeta_{1} \right)\right)(f_{n,n}) \, = \, \tfrac{1}{2}\left( (\mu\zeta_{0})(f_{n,n}) + (\mu\zeta_{1})(f_{n,n}) \right) \, = \, \tfrac{1}{2}
\end{equation} and therefore \begin{align*}
	(\mu \nu)\left( (f_{n,n} - (\mu \nu)(f_{n,n}))^{2} \right) \, &\stackrel{\eqref{Samuel6}}{=} \, (\mu \nu)\left( \left(f_{n,n} - \tfrac{1}{2}\right)^{2} \right) \, = \, \left( \tfrac{1}{2}\left( \mu \zeta_{0} + \mu \zeta_{1} \right)\right)\left( \left(f_{n,n} - \tfrac{1}{2}\right)^{2} \right) \\	
	& = \, \tfrac{1}{2}\left( (\mu\zeta_{0})\left( \left(f_{n,n} - \tfrac{1}{2}\right)^{2} \right) + (\mu\zeta_{1})\left( \left(f_{n,n} - \tfrac{1}{2}\right)^{2} \right) \right) \\
	&  = \, \tfrac{1}{2}\left( \mu\left( \Phi_{\zeta_{0}}\left( \left(f_{n,n} - \tfrac{1}{2}\right)^{2} \right)\right) + \mu\left( \Phi_{\zeta_{1}}\left( \left(f_{n,n} - \tfrac{1}{2}\right)^{2} \right) \right) \right) \\
	& \stackrel{\ref{remark:convolution}(3)}{=} \, \tfrac{1}{2}\left( \mu\left( \left( \Phi_{\zeta_{0}}(f_{n,n}) - \tfrac{1}{2}\right)^{2} \right) + \mu\left( \left(\Phi_{\zeta_{1}}(f_{n,n}) - \tfrac{1}{2}\right)^{2} \right) \right) \\
	& \stackrel{\eqref{Samuel4}}{=} \, \tfrac{1}{2}\left( \mu\left( \left( 0 - \tfrac{1}{2}\right)^{2} \right) + \mu\left( \left(1 - \tfrac{1}{2}\right)^{2} \right) \right) \, = \, \tfrac{1}{4} .
\end{align*} This proves~\eqref{Samuel5} and thus completes the argument. \end{proof}

Contemplating the hypothesis of Theorem~\ref{theorem:pestov}, we now turn to a certain class of topological groups. Following Pachl~\cite{ambitable}, a topological group $G$ will be called \emph{ambitable}\footnote{This definition of ambitability is equivalent to the one in~\cite{ambitable} by~\cite[Lemma~19(2)]{ambitable}.} if \begin{displaymath}
	\forall B \in \RUEB (G) \, \exists f \in \RUCB (G) \colon \quad B \subseteq \{ \Phi_{\xi}f \mid \xi \in \Samuel(G) \} .
\end{displaymath} Precompact topological groups are non-ambitable~\cite[Theorem~2]{ambitable}. It is an open problem~\cite[Question~1]{ambitable} whether the converse is true as well, i.e., whether every topological group is either precompact or ambitable. Within the class of locally compact groups, this question has been answered in the affirmative~\cite[Corollary~15]{ambitable}. Another related result is the following.

\begin{remark}\label{remark:ambitable} Every separable topological group is either precompact or ambitable. This is a direct consequence of~\cite[Corollary~16]{ambitable}. \end{remark}

For our purposes, the following basic observation about ambitable topological groups will be relevant.

\begin{lem}\label{lemma:ambitable} If $G$ is an ambitable topological group, then $\Samuel (G)$ admits a continuum of pairwise disjoint, closed, $G$-invariant, non-empty subsets. \end{lem}

\begin{proof} Let $G$ be an ambitable topological group. In particular, as the set $B$ of all constant functions on $G$ with values in $[0,1]$ is an element of $\RUEB (G)$, there exists $f \in \RUCB (G)$ with $B \subseteq \{ \Phi_{\xi}f \mid \xi \in \Samuel (G) \}$. We observe that \begin{displaymath}
	T_{i} \, \defeq \, \{ \xi \in \Samuel (G) \mid \forall g \in G \colon \, (g\xi)(f) = i \} \qquad (i \in [0,1])
\end{displaymath} constitutes a family of pairwise disjoint, closed, $G$-invariant subsets of $\Samuel (G)$. It remains to show that $T_{i} \ne \emptyset$ for every $i \in [0,1]$. To this end, let $i \in [0,1]$. Our choice of $f$ asserts that there exists $\xi \in \Samuel (G)$ with $\Phi_{\xi}f \equiv i$, whence \begin{displaymath}
	(g\xi)(f) \, = \, \xi(f \circ \lambda_{g}) \, = \, (\Phi_{\xi}f)(g) \, = \, i 
\end{displaymath} for all $g \in G$, thus $\xi \in T_{i}$ and so $T_{i}\ne \emptyset$. This completes the argument. \end{proof}

Quantitative improvements of Lemma~\ref{lemma:ambitable} are known for non-compact, locally compact groups~\cite[Theorem~4.1]{LauMilnesPym}, as well as for non-precompact, separable topological groups~\cite[Theorem~1.3]{FerriStrauss} (see also~\cite[Theorem~3.8]{HindmanStrauss}). It is an open problem whether the Samuel compactification of every non-precompact topological contains two disjoint closed, invariant, non-empty subsets~\cite[paragraph after Theorem~3.8]{HindmanStrauss}.

Regarding the notation used in the following two corollaries, we refer to the paragraph preceding Theorem~\ref{theorem:pestov}.

\begin{cor}\label{corollary:pestov.1} Let $G$ be a second-countable, non-precompact topological group. Then there exists a sequence $\nu_{n} \in \Mean (G^{n})$ $(n \in \N_{>0})$ such that, for every compatible right-invariant metric $d$ on $G$, there exists $f \in \Lip_{1}(G,d;[0,1])$ such that \begin{displaymath}
	\inf\nolimits_{n \in \N_{>0}} \inf\nolimits_{\mu \in \Mean (G^{n})} (\mu \nu_{n})\left( (f_{n,n} - (\mu\nu_{n})(f_{n,n}))^{2} \right) \, > \,  0 .
\end{displaymath} \end{cor}

\begin{proof} Due to Remark~\ref{remark:ambitable} and Lemma~\ref{lemma:ambitable}, $\Samuel (G)$ contains a disjoint pair of closed, $G$-invariant, non-empty subsets. Thus, by Theorem~\ref{theorem:pestov}, there exist $f \in \RUC (G,[0,1])$ and a sequence $\nu_{n} \in \Mean (G^{n})$ $(n \in \N_{>0})$ such that \begin{equation}\tag{1}\label{pestov}
	\forall n \in \N_{>0} \, \forall \mu \in \Mean (G^{n}) \colon \qquad (\mu\nu_{n})\left( (f_{n,n} - (\mu\nu_{n})(f_{n,n}))^{2} \right) \, = \, \tfrac{1}{4}.
\end{equation} Now, let $d$ be a compatible right-invariant metric on $G$. Since $f$ is right-uniformly continuous and $d$ generates the topology of $G$, there is $\ell \in \R_{\geq 1}$ such that $\Vert f - (f\circ\lambda_{u}) \Vert_{\infty} \leq \tfrac{1}{32}$ for every $u \in \B_{d}(e,1/\ell)$, whence \begin{align*}
	\vert f(x) - f(y) \vert \, = \, \left\lvert f(x) - f\left(yx^{-1}x\right) \right\rvert \, &\leq \, \max \left\{ \tfrac{1}{32}, \, \ell d\left(e,yx^{-1}\right) \right\} \\
	& \leq \, \ell d\left(e,yx^{-1}\right) + \tfrac{1}{32} \, = \, \ell d(x,y) + \tfrac{1}{32}
\end{align*} for all $x,y \in G$. Due to Remark~\ref{remark:lipschitz.approximation}, we now find $h \in \Lip_{\ell}(G,d;[0,1])$ such that $\Vert f - h \Vert_{\infty} \leq \tfrac{1}{32}$, which readily entails that \begin{equation}\tag{2}\label{apple}
	\forall n \in \N_{>0} \colon \qquad \left\lVert f_{n,n} - h_{n,n} \right\rVert_{\infty} \leq \tfrac{1}{32} .
\end{equation} We claim that \begin{equation}\tag{3}\label{apple.too}
	\forall n \in \N_{>0} \, \forall \mu \in \Mean (G^{n}) \colon \qquad (\mu\nu_{n})\left( (h_{n,n} - (\mu\nu_{n})(h_{n,n}))^{2} \right) \, \geq \, \tfrac{1}{8}.
\end{equation} To prove this, let $n \in \N_{>0}$ and $\mu \in \Mean (G^{n})$. From $0\leq h_{n,n}+f_{n,n} \leq 2$, we infer that $0\leq (\mu\nu_{n})(h_{n,n}+f_{n,n}) \leq 2$ and thus $\Vert h_{n,n}+f_{n,n} - (\mu\nu_{n})(h_{n,n}+f_{n,n}) \Vert_{\infty} \leq 2$. Since \begin{align*}
	&(h_{n,n} - (\mu\nu_{n})(h_{n,n}))^{2} - (f_{n,n} - (\mu\nu_{n})(f_{n,n}))^{2} \\
	& \hspace{15mm} = \, (h_{n,n}+f_{n,n} - (\mu\nu_{n})(h_{n,n}+f_{n,n}))(h_{n,n}-f_{n,n} - (\mu\nu_{n})(h_{n,n}-f_{n,n})) ,
\end{align*} we conclude that \begin{align*}
	&\left\lVert (h_{n,n} - (\mu\nu_{n})(h_{n,n}))^{2} - (f_{n,n} - (\mu\nu_{n})(f_{n,n}))^{2} \right\rVert_{\infty} \\
	& \hspace{8mm} \leq \, \left\lVert h_{n,n}+f_{n,n} - (\mu\nu_{n})(h_{n,n}+f_{n,n}) \right\rVert_{\infty} \cdot \left\lVert h_{n,n}-f_{n,n} - (\mu\nu_{n})(h_{n,n}-f_{n,n}) \right\rVert_{\infty} \\
	& \hspace{8mm} \stackrel{\eqref{apple}}{\leq} \, 2\cdot \tfrac{1}{16} \, = \, \tfrac{1}{8} .
\end{align*} Consequently, \begin{displaymath}
	(\mu\nu_{n})\left( (h_{n,n} - (\mu\nu_{n})(h_{n,n}))^{2} \right) \, \geq \, (\mu\nu_{n})\left( (f_{n,n} - (\mu\nu_{n})(f_{n,n}))^{2} \right) - \tfrac{1}{8} \, \stackrel{\eqref{pestov}}{=} \, \tfrac{1}{8} .
\end{displaymath} This proves~\eqref{apple.too}. Finally, we observe that $g \defeq \ell^{-1}h \in \Lip_{1}(G,d;[0,1])$ and \begin{displaymath}
	\inf\nolimits_{n \in \N_{>0}} \inf\nolimits_{\mu \in \Mean (G^{n})} (\mu \nu_{n})\left( (g_{n,n} - (\mu\nu_{n})(g_{n,n}))^{2} \right) \, \stackrel{\eqref{apple.too}}{\geq} \, \tfrac{1}{8 \ell^{2}} \, > \,  0 . \qedhere
\end{displaymath} \end{proof}

We arrive at the announced negative solution to Problem~\ref{problem}, continuing to use the notation introduced in the paragraph just before Theorem~\ref{theorem:pestov}.

\begin{cor}\label{corollary:pestov.2} Let $G$ be a second-countable, non-precompact, amenable topological group. Then there exists a sequence of left-invariant means $\mu_{n} \in \Mean (G^{n})$ $(n \in \N_{>0})$ such that, for every compatible right-invariant metric $d$ on $G$, \begin{displaymath}
	\exists f \in \Lip_{1}(G,d;[0,1]) \colon \qquad \inf\nolimits_{n \in \N_{>0}} \mu_{n}\left( (f_{n,n} - \mu_{n}(f_{n,n}))^{2} \right) \, > \,  0 .
\end{displaymath} \end{cor}

\begin{proof} Since $G$ is amenable, for every $n \in \N_{>0}$ the topological group $G^{n}$ is amenable as well, due to~\cite[Theorem~4.8]{rickert}. For each $n \in \N_{>0}$, let us pick any $G^{n}$-left-invariant mean $\mu_{n} \in \Mean (G^{n})$. By Corollary~\ref{corollary:pestov.1}, there exists a sequence $\nu_{n} \in \Mean (G^{n})$ $(n \in \N_{>0})$ such that, for every compatible right-invariant metric $d$ on $G$, there exists $f \in \Lip_{1}(G,d;[0,1])$ with \begin{displaymath}
	\inf\nolimits_{n \in \N_{>0}} \inf\nolimits_{\mu \in \Mean (G^{n})} (\mu \nu_{n})\left( (f_{n,n} - (\mu\nu_{n})(f_{n,n}))^{2} \right) \, > \,  0 ,
\end{displaymath} whence, in particular, \begin{displaymath}
	\inf\nolimits_{n \in \N_{>0}} (\mu_{n} \nu_{n})\left( (f_{n,n} - (\mu_{n}\nu_{n})(f_{n,n}))^{2} \right) \, > \,  0 .
\end{displaymath} It thus remains to observe that, for each $n \in \N_{>0}$, the mean $\mu_{n}\nu_{n} \in \Mean (G^{n})$ is $G^{n}$-left-invariant by Remark~\ref{remark:ideal.of.invariant.means}. \end{proof}

\section{Continuous geometries and their maximal chains}\label{section:continuous.geometries}

This section marks the beginning of our study of continuous geometries. Beyond providing some background concerning the groundbreaking work of von Neumann~\cite{VonNeumannBook} (see also Maeda's monograph~\cite{MaedaBook}), the main purpose of this section is to prove a characterization of the maximal chains in irreducible continuous geometries by means of their dimension functions (Theorem~\ref{theorem:maximal.flags}). We start off by clarifying some relevant terminology.

Let $P$ be a partially ordered set. Let $\Max (P)$ denote the set of all maximal elements of $P$. A \emph{chain} in $P$ is a subset $C \subseteq P$ such that the restriction of the partial order of $P$ to $C$, i.e., its intersection with $C \times C$, constitutes a linear order on $C$. A \emph{maximal chain} in $P$ is a maximal element of the partially ordered set $(\{ C \subseteq P \mid C \text{ chain in } P \},{\subseteq})$. By the Hausdorff maximal principle, a well-known equivalent of the axiom of choice (see, e.g.,~\cite[2.7, p.~61]{DevlinBook}), every chain in $P$ is contained in a maximal chain of $P$.

By a \emph{lattice} we mean a partially ordered set $L$ in which every pair of elements $x,y \in L$ admits both a (necessarily unique) supremum $x\vee y \in L$ and a (necessarily unique) infimum $x\wedge y \in L$. Equivalently, lattices may be characterized as algebraic structures with two commutative and associative binary operations satisfying the two absorption laws. A \emph{complete lattice} is a partially ordered set $L$ such that every subset $S \subseteq L$ has a (necessarily unique) supremum $\bigvee S \in L$. If $L$ is a complete lattice, then every $S \subseteq L$ admits a (necessarily unique) infimum $\bigwedge S \in L$, too.

\begin{remark}\label{remark:maximal.chains.complete} Let $L$ be a complete lattice. If $C \subseteq L$ is a maximal chain in $L$, then both $\bigvee S \in C$ and $\bigwedge S \in C$ for every $S \subseteq C$. \end{remark} 

As any equational class of algebraic structures, the collection of all lattices admits direct products. A lattice $L$ is said to be \emph{(directly) irreducible} if $\vert L \vert \geq 2$ and $L$ is not isomorphic to a direct product of two lattices of cardinality at least two. A lattice is said to be \emph{bounded} if it possesses both a (necessarily unique) greatest element and a (necessarily unique) least element. If $L$ is a bounded lattice, then $1 = 1_{L}$ denotes the greatest element and $0 = 0_{L}$ denotes the least element of $L$. Evidently, any complete lattice is bounded. A lattice $L$ is called \begin{itemize}
	\item[---] \emph{complemented} if $L$ is bounded and, for every $x \in L$, there exists $y \in L$ such that $x\wedge y = 0$ and $x \vee y = 1$,
	\item[---] \emph{relatively complemented} if, for all $a,b,x \in L$ with $a \leq x \leq b$, there exists $y \in L$ such that $x\wedge y = a$ and $x \vee y = b$,
	\item[---] \emph{modular} if, for all $x,y,z \in L$, \begin{displaymath}
					\qquad x \leq y \ \, \Longrightarrow \ \, x \vee (y \wedge z) = y \wedge (x \vee z) .
\end{displaymath} \end{itemize} Clearly, every relatively complemented, bounded lattice is complemented. For modular lattices, the converse is true as well:

\begin{lem}[\cite{VonNeumannBook}, \cite{BirkhoffBook}]\label{lemma:relatively.complemented} Every complemented, modular lattice is relatively complemented. \end{lem}

\begin{proof} See~\cite[I.I, Theorem~1.3, p.~5]{VonNeumannBook}, or~\cite[VIII.1, Theorem~1]{BirkhoffBook}. \end{proof}

Let $L$ be a bounded lattice. Two elements $x,y \in L$ are said to be \emph{perspective} in $L$ and we write $x \sim y$ if there exists $z \in L$ such that $x \wedge z = y \wedge z = 0$ and $x \vee z = y \vee z = 1$. For $x,y \in L$, let us define \begin{displaymath} 
	x \precsim y \quad :\Longleftrightarrow \quad \exists x' \in L \colon \ x \sim x' \leq y .
\end{displaymath} Let $n \in \N$. A tuple $(x_{1},\ldots,x_{n}) \in L^{n}$ is said to be \emph{independent} in $L$ if \begin{displaymath}
	\forall I,J \subseteq \{ 1,\ldots,n \} \colon \quad I \cap J = \emptyset \ \, \Longrightarrow \ \, \left( \bigvee\nolimits_{\! \! \! i \in I} x_{i} \right) \wedge \left( \bigvee\nolimits_{\! \! \! j \in J} x_{j} \right) = 0 .
\end{displaymath} Suppose that $L$ is complemented and modular. Then $L$ is said to \emph{have order} $n$ if there exists an independent tuple $(x_{1},\ldots,x_{n}) \in L^{n}$ of pairwise perspective elements of $L$ such that $\bigvee_{i=1}^{n} x_{i} = 1$. In general, a complemented, modular lattice may have multiple orders (see~\cite[II.III, Note after Definition~3.2, p.~93]{VonNeumannBook}, or~\cite[VIII.1, Definition~1.2, p.~170]{MaedaBook}).

A complete lattice $L$ is said to be \emph{continuous}\footnote{This notion admits a characterization in terms of directed subsets~\cite[Proposition~13.1]{GoodearlBook}.} if, for every chain $C \subseteq L$ and every element $x \in L$, \begin{displaymath}
	x \wedge \bigvee C \, = \, \bigvee \{ x \wedge y \mid y \in C\}, \qquad x \vee \bigwedge C \, = \, \bigwedge \{ x \vee y \mid y \in C\} .
\end{displaymath} A \emph{continuous geometry} is a continuous, complete, complemented, modular lattice. Subsequently, we recall some elements from von Neumann's comprehensive treatment~\cite{VonNeumannBook} of such objects.

We turn to certain real-valued functions on lattices. To this end, let $L$ be a lattice. A function $\delta \colon L \to [0,1]$ is called \emph{modular} if \begin{displaymath}
	\forall x,y \in L \colon \quad \delta (x\vee y) + \delta (x \wedge y) \, = \, \delta (x) + \delta (y) ,
\end{displaymath} $\delta$ is called \emph{positive} if \begin{displaymath}
	\forall x,y \in L \colon \quad x \leq y \ \, \Longrightarrow \ \, \delta (x) \leq \delta (y) ,
\end{displaymath} and $\delta$ is called \emph{strictly positive} if \begin{displaymath}
	\forall x,y \in L \colon \quad x < y \ \, \Longrightarrow \ \, \delta (x) < \delta (y) .
\end{displaymath} Suppose now that $L$ is a bounded lattice. A function $\delta \colon L \to [0,1]$ will be called a \emph{pseudo-dimension function} if $\delta$ is modular, positive, and moreover satisfies $\delta (0_{L}) = 0$ and $\delta (1_{L}) = 1$. A \emph{dimension function} on $L$ is a strictly positive pseudo-dimension function on $L$. For the sake of completeness, we recall that any lattice admitting a strictly positive, modular function is necessarily modular~\cite[I.6, Satz~6.1, p.~46]{MaedaBook}.

\begin{lem}[\cite{BirkhoffBook,MaedaBook}]\label{lemma:birkhoff} Let $L$ be a lattice and let $\delta \colon L \to [0,1]$ be modular and positive. Then \begin{displaymath}
		d_{\delta} \colon \, L \times L \, \longrightarrow \, [0,1], \quad (x,y) \, \longmapsto \, \delta(x\vee y) - \delta(x\wedge y)
	\end{displaymath} is a pseudo-metric on $L$. Furthermore, the following hold.\begin{itemize}
	\item[$(1)$] $\delta$ strictly positive $\ \Longleftrightarrow \ $ $d_{\delta}$ metric on $L$.
	\item[$(2)$] Let $a \in L$. Then, \begin{displaymath}
	\qquad \forall x,y \in L \colon \quad d_{\delta}(a \wedge x,a \wedge y) + d_{\delta}(a \vee x,a \vee y) \, \leq \, d_{\delta}(x,y) .
\end{displaymath} In particular, the maps \begin{displaymath}
	\qquad L \, \longrightarrow \, L, \quad x \, \longmapsto \, a \wedge x, \qquad L \, \longrightarrow \, L, \quad x \, \longmapsto \, a \vee x
\end{displaymath} are $1$-Lipschitz, hence continuous, with respect to $d_{\delta}$. 
	\item[$(3)$] For all $x,y \in L$, \begin{displaymath}
					\qquad d_{\delta}(x,y) \, = \, 2\delta (x\vee y) -\delta(x) -\delta(y) \, = \, \delta(x) +\delta(y)-2\delta(x\wedge y) .
				\end{displaymath}
\end{itemize} \end{lem}

\begin{proof} Except for~(3), these claims are proved in~\cite[V.7, Lemma on p.~76]{BirkhoffBook} and~\cite[I.6, Satz~6.2, p.~46]{MaedaBook}.

(3) For all $x,y \in L$, modularity of $\delta$ implies that \begin{align*}
	d_{\delta}(x,y) \, &= \, \delta(x\vee y) - \delta(x\wedge y) \, = \, \delta(x\vee y) - (\delta(x)+\delta(y)-\delta(x\vee y)) \\
	& = \, 2\delta(x\wedge y)-\delta(x) -\delta(y) , \\
	d_{\delta}(x,y) \, &= \, \delta(x\vee y) - \delta(x\wedge y) \, = \, (\delta(x)+\delta(y)-\delta(x\wedge y)) - \delta (x\wedge y) \\
	& = \, \delta(x) +\delta(y)-2\delta(x\wedge y) . \qedhere
\end{align*} \end{proof}

The following is one of the central results of~\cite{VonNeumannBook}.

\begin{thm}[\cite{VonNeumannBook}]\label{theorem:dimension} Every irreducible continuous geometry possesses a unique dimension function. \end{thm}

\begin{proof} Existence is due to~\cite[I.VI, Theorem~6.9, p.~52]{VonNeumannBook}, while uniqueness is due to~\cite[I.VII, Corollary~1 on p.~60]{VonNeumannBook}. (Alternatively, see~\cite[V.2, Satz~2.1]{MaedaBook} for the former assertion, and~\cite[V.2, Satz~2.3, p.~120]{MaedaBook} for the latter.) \end{proof}

\begin{definition}\label{definition:dimension.function} Let $L$ be an irreducible continuous geometry. Let us define $\delta_{L} \colon L \to [0,1]$ to be the unique dimension function on $L$ and let \begin{displaymath}
	d_{L} \defeq d_{\delta_{L}} \colon \, L \times L \, \longrightarrow \, [0,1], \quad (x,y) \, \longmapsto \, \delta_{L}(x\vee y) - \delta_{L}(x\wedge y)
\end{displaymath} denote the metric associated to $\delta_{L}$ via Lemma~\ref{lemma:birkhoff}. We call $L$ \emph{discrete} if the topology on $L$ generated by $d_{L}$ is discrete. \end{definition}

\begin{remark}[\cite{VonNeumannBook}, I.VII, Theorem~7.3, p.~58]\label{remark:range}
Let $L$ be any irreducible continuous geometry and let us consider $D \defeq \delta_{L}(L)$. As established by von Neumann~\cite[I.VII, Lemma~7.3, p.~56]{VonNeumannBook},
the following hold: \begin{itemize}
	\item[$(1)$] $0,1 \in D$.
	\item[$(2)$] If $s,t \in D$ and $s \leq t$, then $t-s \in D$.
	\item[$(3)$] $\sup C \in D$ for every countable subset $C \subseteq D$. 
\end{itemize} Due to an abstract argument~\cite[I.VII, Lemma~7.4, p.~57]{VonNeumannBook}, these properties imply that either $D = [0,1]$, or there exists $n \in \N_{>0}$ with \begin{displaymath}
	D \! \left. \, = \, \left\{ \tfrac{k}{n} \, \right\vert k \in \{ 0,\ldots,n \} \right\} .
\end{displaymath} It follows that $L$ is non-discrete if and only if $\delta_{L}(L) = [0,1]$. \end{remark}

\begin{remark}\label{remark:perspective} Let $L$ be a bounded lattice. \begin{itemize}
	\item[$(1)$] Let $\delta \colon L \to [0,1]$ be modular. Then \begin{displaymath}
		\qquad \forall x,y \in L \colon \quad x \sim y \ \, \Longrightarrow \ \, \delta (x) = \delta (y) .
	\end{displaymath} Indeed, if $x,y \in L$ are perspective, then there exists $z \in L$ such that $x \wedge z = y \wedge z = 0$ and $x \vee z = y \vee z = 1$, whence \begin{displaymath}
		\qquad \delta(x) \, = \, \delta(x\vee z) + \delta(x\wedge z) - \delta(z) \, = \, \delta(y\vee z) + \delta(y\wedge z) - \delta(z) \, = \, \delta (y) .
\end{displaymath} 
	\item[$(2)$] If $\delta \colon L \to [0,1]$ is modular and positive, then~(1) entails that \begin{displaymath}
		\qquad \forall x,y \in L \colon \quad x \precsim y \ \, \Longrightarrow \ \, \delta (x) \leq \delta (y) .
	\end{displaymath}
	\item[$(3)$] If $L$ is an irreducible continuous geometry, then \begin{displaymath}
		\qquad \forall x,y \in L \colon \quad x \precsim y \ \, \Longleftrightarrow \ \, \delta_{L} (x) \leq \delta_{L} (y) ,
	\end{displaymath} by a result of von Neumann~\cite[I.VI, Theorem~6.9(iii)'', p.~52]{VonNeumannBook} (see also~\cite[V.2, Satz~2.1(2°), p.~118]{MaedaBook}).
\end{itemize} \end{remark}

The following continuity property of dimension functions of irreducible continuous geometries will be relevant for the proof of Theorem~\ref{theorem:maximal.flags}.

\begin{prop}[\cite{MaedaBook}]\label{proposition:dimension.continuity} Let $L$ be an irreducible continuous geometry. If $C$ is a chain in $L$, then \begin{displaymath}
	\delta_{L}\left(\bigvee C\right) \, = \, \sup \delta_{L}(C), \qquad \delta_{L}\left(\bigwedge C\right) \, = \, \inf \delta_{L}(C) .
\end{displaymath} \end{prop}

\begin{proof} This fact is stated in~\cite[V.2, Satz~2.1, p.~118]{MaedaBook} and deduced from a more general result~\cite[V.1, Satz~1.8, p.~117]{MaedaBook}. \end{proof}

We arrive at the central observation of this section.

\begin{thm}\label{theorem:maximal.flags} Let $L$ be an irreducible continuous geometry. A chain $C \subseteq L$ is maximal if and only if $\delta_{L}(C) = \delta_{L}(L)$. \end{thm}

\begin{proof} ($\Longleftarrow$) Let $C$ be a chain in $L$ with $\delta_{L}(C) = \delta_{L}(L)$. Suppose that $C'$ is a chain in $L$ with $C \subseteq C'$. By assumption, for every $x \in C'$ there exists $y \in C$ with $\delta_{L}(y) = \delta_{L}(x)$, whence $\delta_{L}$ being strictly positive and $C'$ being a chain in $L$ imply that $x=y$. Thus, $C = C'$. This shows maximality of $C$.

($\Longrightarrow$) Suppose that $C$ is a maximal chain in $L$. To prove that $\delta_{L}(C) = \delta_{L}(L)$, let $a \in L$. Define \begin{displaymath}
	C_{0} \, \defeq \, \{ x \in C \mid \delta_{L}(x) \leq \delta_{L}(a) \}, \qquad C_{1} \, \defeq \, \{ x \in C \mid \delta_{L}(x) \geq \delta_{L}(a) \} .
\end{displaymath} Since $C$ is a maximal chain in $L$, Remark~\ref{remark:maximal.chains.complete} asserts that both $x_{0} \defeq \bigvee C_{0}$ and $x_{1} \defeq \bigwedge C_{1}$ belong to $C$. By $C_{0}$ and $C_{1}$ being chains in $L$, \begin{displaymath}
	\delta_{L}(x_{0}) \, \stackrel{\ref{proposition:dimension.continuity}}{=} \, \sup\nolimits_{x \in C_{0}} \delta_{L}(x) \, \leq \, \delta_{L}(a) \, \leq \, \inf\nolimits_{x \in C_{1}} \delta_{L}(x) \, \stackrel{\ref{proposition:dimension.continuity}}{=} \, \delta_{L}(x_{1}) .
\end{displaymath} Evidently, $x_{0} \leq x_{1}$. As $L$ is relatively complemented by Lemma~\ref{lemma:relatively.complemented}, there exists $y \in L$ such that $x_{0} \vee y = x_{1}$ and $x_{0} \wedge y = 0$. Thanks to Remark~\ref{remark:range}(2), we find $z \in L$ with $\delta_{L}(z) = \delta_{L}(a)-\delta_{L}(x_{0})$. Since \begin{displaymath}
	\delta_{L}(y) \, = \, \delta_{L}(x_{1}) - \delta_{L}(x_{0}) \, \geq \, \delta_{L}(a) - \delta_{L}(x_{0}) \, = \, \delta_{L}(z) ,
\end{displaymath} Remark~\ref{remark:perspective}(3) asserts that $z \precsim y$, i.e., there exists $z' \in L$ with $z \sim z' \leq y$. Define $x' \defeq x_{0} \vee z'$. Then $x_{0} \leq x_{0} \vee z' = x' =  x_{0} \vee z' \leq x_{0} \vee y = x_{1}$, which entails that $C \cup \{ x' \}$ is a chain in $L$. Hence, $C=C\cup \{ x'\}$ by maximality of the chain $C$, that is, $x' \in C$. Since $x_{0} \wedge z' \leq x_{0} \wedge y = 0$ and therefore $x_{0} \wedge z' = 0$, we conclude that \begin{displaymath}
	\delta_{L}(x') \, = \, \delta_{L}(x_{0}) + \delta_{L}(z') \, \stackrel{\ref{remark:perspective}(1)}{=} \, \delta_{L}(x_{0}) + \delta_{L}(z) \, = \, \delta_{L}(a) .
\end{displaymath} Consequently, $\delta_{L}(a)=\delta_{L}(x') \in \delta_{L}(C)$ as desired. \end{proof}

\begin{cor}\label{corollary:maximal.flags} Let $L$ be a non-discrete irreducible continuous geometry. A chain $C \subseteq L$ is maximal if and only if $\delta_{L}(C) = [0,1]$. \end{cor}

\begin{proof} This follows by Theorem~\ref{theorem:maximal.flags} and Remark~\ref{remark:range}. \end{proof}

\section{Continuous rings and their maximal nests}\label{section:continuous.rings}

This section revolves around regular and continuous rings, rank functions, and direct finiteness. In addition to recollecting some fundamental results by von Neumann~\cite{VonNeumannBook} (see also~\cite{MaedaBook,GoodearlBook}), our main objective is to provide a description of maximal chains in an irreducible continuous geometry in terms of maximal nests in the underlying ring (Theorem~\ref{theorem:maximal.nests}).

Before proceeding to regular rings, we briefly recall some general facts about idempotent elements in semigroups. To this end, let $S$ be a semigroup and let $\Ed(S) \defeq \{ e \in S \mid ee = e \}$. Then the relation given by \begin{displaymath}
	e \leq f \quad : \Longleftrightarrow \quad ef = fe = e \qquad (e,f \in \Ed(S))
\end{displaymath} constitutes a partial order on $\Ed(S)$.

\begin{lem}\label{lemma:topological.semigroups} Let $S$ be a Hausdorff topological semigroup. The following hold. \begin{itemize}
	\item[$(1)$] $\Ed(S)$ is closed in $S$.
	\item[$(2)$] If $E$ is a chain in $(\Ed(S),\leq)$, then so is $\overline{E}$.
\end{itemize} \end{lem}

\begin{proof} Since $S$ is a Hausdorff space, $R_{0} \defeq \{ (s,s) \mid s \in S\}$ is closed in $S \times S$ and $R_{1} \defeq \{ (s,s,s) \mid s \in S\}$ is closed in $S \times S \times S$. As $\phi \colon S \to S \times S, \, t \mapsto (t,tt)$ and \begin{displaymath}
	\psi \colon \, S\times S \, \longrightarrow \, S \times S \times S, \quad (s,t) \, \longmapsto \, (s,st,ts) 
\end{displaymath} are continuous maps, thus $\Ed(S) = \phi^{-1}(R_{0})$ is closed in $S$ and, moreover, \begin{displaymath}
	{\leq} \, = \, (\Ed(S) \times \Ed(S)) \cap \psi^{-1}(R_{1})
\end{displaymath} is closed in $S \times S$. Since $S \times S \to S \times S, \, (s,t) \mapsto (t,s)$ is a homeomorphism, the latter entails that $\geq$ is closed in $S \times S$ as well. Consequently, if $E$ is a chain in $(\Ed(S),\leq)$, then \begin{displaymath}
	\overline{E} \times \overline{E} \, = \, \overline{E \times E} \, \subseteq \, \overline{{\leq} \cup {\geq}} \, = \, {\leq} \cup {\geq} ,
\end{displaymath} so that $\overline{E}$ is a chain in $(\Ed(S),\leq)$, too. \end{proof}

Let $R$ be a ring. Adopting the notation above with reference to the multiplicative semigroup of $R$, let us consider the partially ordered set $(\Ed(R),{\leq})$. Furthermore, two elements $e,f \in \Ed(R)$ will be called \emph{orthogonal} and we will write $e \perp f$ if $ef = fe = 0$.

\begin{remark}\label{remark:idempotent.difference} Let $R$ be a ring. If $e,f \in \Ed(R)$ and $f \leq e$, then $e-f \in \Ed(R)$ and $f \perp (e-f) \leq e$. \end{remark}

Given a unital ring $R$, we consider the set $\lat(R) \defeq \{ aR \mid a \in R \}$, partially ordered by inclusion, the \emph{center} \begin{displaymath}
	\cent(R) \, \defeq \, \{ a \in R \mid \forall b \in R \colon \, ab = ba \} ,
\end{displaymath} which constitutes a commutative unital subring of $R$, as well as \begin{displaymath}
	\GL(R) \, \defeq \, \{ a \in R \mid \exists b \in R \colon \, ab = ba = 1 \} ,
\end{displaymath} the (multiplicative) group of \emph{units} of $R$. A unital ring $R$ is called \emph{(von Neumann) regular} if \begin{displaymath}
	\forall a \in R \ \exists b \in R \colon \qquad aba \, = \, a .
\end{displaymath} 

\begin{remark}[\cite{VonNeumannBook}, II.II, Theorem~2.2, p.~70]\label{remark:regular} A unital ring $R$ is regular if and only if, for every $a \in R$, there exists $e \in \Ed(R)$ such that $aR = eR$. \end{remark} 

According to~\cite[II.II, Theorem~2.4, p.~72]{VonNeumannBook}, if $R$ is a regular ring, then $(\lat(R),{\subseteq})$ is a complemented, modular lattice, in which \begin{displaymath}
	I\vee J \, = \, I+J, \quad I\wedge J \, = \, I\cap J \qquad (I,J \in \lat(R)) .
\end{displaymath} Conversely, another fundamental theorem due to von Neumann~\cite[II.XIV, Theorem~14.1, p.~208]{VonNeumannBook} asserts that, if a complemented, modular lattice $L$ has an order greater than or equal to four, then there exists an (up to isomorphism unique) regular ring $R$ with $L \cong \lat(R)$.

Recall that a ring is said to be \emph{(directly) irreducible} if it is non-zero and not isomorphic to a direct product of two non-zero rings. A unital ring $R$ is irreducible if and only if $R$ is non-zero and $\Ed(R) \cap \cent(R) = \{ 0,1 \}$ (see, e.g.,~\cite[VI.1, Satz~1.10, p.~139]{MaedaBook}). Furthermore, irreducibility of a regular ring and the corresponding lattice are linked in the following natural way.

\begin{thm}[\cite{VonNeumannBook}]\label{theorem:irreducibility} Let $R$ be a regular ring. The following are equivalent. \begin{itemize}
	\item[$(1)$] $R$ is irreducible.
	\item[$(2)$] $\cent(R)$ is a field.
	\item[$(3)$] $\lat(R)$ is irreducible.
\end{itemize} \end{thm}

\begin{proof} The equivalence of~(1) and~(2) is due to~\cite[II.II, Theorem~2.7, p.~75]{VonNeumannBook}, the equivalence of~(1) and~(3) is due to~\cite[II.II, Theorem~2.9, p.~76]{VonNeumannBook}.  \end{proof}

Before proceeding to pseudo-rank functions, let us note a general order-theoretic property of regular rings.

\begin{lem}\label{lemma:intermediate.idempotent} Let $R$ be a regular ring, let $e_{0},e_{1} \in \Ed(R)$ with $e_{0} \leq e_{1}$. If $I \in \lat(R)$ and $e_{0}R \subseteq I \subseteq e_{1}R$, then there exists $f \in \Ed(R)$ with $e_{0} \leq f \leq e_{1}$ and $I = fR$. \end{lem}

\begin{proof} Let $I \in \lat(R)$ with $e_{0}R \subseteq I \subseteq e_{1}R$. Then $(e_{1}-e_{0})x = e_{1}x-e_{0}x = x-e_{0}x \in I$ for every $x \in I$, that is, $(e_{1}-e_{0})I \subseteq I$. As $(e_{1}-e_{0})I \in \lat(R)$ and $R$ is regular, Remark~\ref{remark:regular} asserts the existence of some $f_{0} \in \Ed(R)$ such that $(e_{1}-e_{0})I = f_{0}R$. Consider $f \defeq e_{0}+f_{0}(e_{1}-e_{0}) \in R$. Since $f_{0} \in (e_{1}-e_{0})I$ and $e_{1}-e_{0} \in \Ed(R)$ by Remark~\ref{remark:idempotent.difference}, we infer that \begin{equation}\tag{1}\label{idem0}
	(e_{1}-e_{0})f_{0} \, = \, f_{0} ,
\end{equation} whence \begin{equation}\tag{2}\label{idem1}
	e_{0}f_{0} \, \stackrel{\eqref{idem0}}{=} \, e_{0}(e_{1}-e_{0})f_{0} \, \stackrel{\ref{remark:idempotent.difference}}{=} \, 0
\end{equation} and \begin{equation}\tag{3}\label{idem2}
	e_{1}f_{0} \, \stackrel{\eqref{idem0}}{=} \, e_{1}(e_{1}-e_{0})f_{0} \, \stackrel{\ref{remark:idempotent.difference}}{=} \, (e_{1}-e_{0})f_{0} \, \stackrel{\eqref{idem0}}{=} \, f_{0} .
\end{equation} We conclude that \begin{align*}
	ff \, &= \, (e_{0}+f_{0}(e_{1}-e_{0}))(e_{0}+f_{0}(e_{1}-e_{0})) \, \stackrel{\eqref{idem1}+\ref{remark:idempotent.difference}}{=} \, e_{0} + f_{0}(e_{1}-e_{0})f_{0}(e_{1}-e_{0}) \\
	& \stackrel{\eqref{idem0}}{=} \, e_{0} + f_{0}(e_{1}-e_{0}) \, = \, f , \\
	e_{0}f \, &= \, e_{0} + e_{0}f_{0}(e_{1}-e_{0}) \, \stackrel{\eqref{idem1}}{=} \, e_{0} \, \stackrel{\ref{remark:idempotent.difference}}{=} \,  e_{0} + f_{0}(e_{1}-e_{0})e_{0} \, = \, fe_{0}, \\
	e_{1}f \, &= \, e_{1}e_{0} + e_{1}f_{0}(e_{1}-e_{0}) \, \stackrel{\eqref{idem2}}{=} \, e_{0} + f_{0}(e_{1}-e_{0}) \, = \, f \, = \, e_{0} + f_{0}(e_{1}-e_{0}) \\
	& \stackrel{\ref{remark:idempotent.difference}}{=} \,  e_{0}e_{1} + f_{0}(e_{1}-e_{0})e_{1} \, = \, fe_{1},
\end{align*} which means that $f \in \Ed(R)$ and $e_{0} \leq f \leq e_{1}$. Furthermore, \begin{displaymath}
	fR \, \subseteq \, e_{0}R + f_{0}R \, = \, e_{0}R + (e_{1}-e_{0})I \, \subseteq \, I .
\end{displaymath} Conversely, if $x \in I$, then $(e_{1}-e_{0})x \in (e_{1}-e_{0})I = f_{0}R$ and therefore \begin{displaymath}
	x \, = \, e_{1}x \, = \, e_{0}x + (e_{1}-e_{0})x \, = \, e_{0}x + f_{0}(e_{1}-e_{0})x \, = \, fx \, \in \, fR .
\end{displaymath} This shows that $fR = I$ and thus completes the proof. \end{proof}

We recall some background material concerning pseudo-rank functions on regular rings from~\cite[Chapter~16]{GoodearlBook}. To this end, let $R$ be a regular ring. A map $\rho \colon R \to [0,1]$ is said to be a \emph{pseudo-rank function} on $R$ if \begin{itemize}
	\item[---] $\rho(1)=1$,
	\item[---] $\rho(ab) \leq \min \{ \rho(a),\rho(b)\}$ for all $a,b \in R$, and 
	\item[---] $\rho(e+f) = \rho(e) + \rho(f)$ for any two orthogonal $e,f \in \Ed(R)$.
\end{itemize} As a consequence of third condition, $\rho(0) = 0$ for any pseudo-rank function $\rho \colon R \to [0,1]$. A \emph{rank function} on $R$ is a pseudo-rank function $\rho \colon R \to [0,1]$ such that $\rho(a) > 0$ for every $a \in R\setminus \{ 0 \}$.

\begin{lem}[\cite{VonNeumannBook}]\label{lemma:rank.estimates} Let $R$ be a regular ring and let $\rho \colon R \to [0,1]$ be a pseudo-rank function. Then the following hold. \begin{itemize}
	\item[$(1)$] $\rho(a+b) \leq \rho(a) + \rho(b)$ for all $a,b \in R$.
	\item[$(2)$] If $e,f \in \Ed(R)$ and $f \leq e$, then $\rho(e-f) = \rho(e) -\rho(f)$.
\end{itemize} \end{lem}

\begin{proof} (1) Proofs of this are to be found in~\cite[II.XVIII, Corollary on p.~231]{VonNeumannBook}, \cite[VI.5, Hilfssatz~5.1(3°), p.~153]{MaedaBook}, as well as~\cite[Proposition~16.1(d)]{GoodearlBook}.

(2) As $e-f \in \Ed(R)$ and $f \perp (e-f)$ by Remark~\ref{remark:idempotent.difference}, \begin{displaymath}
	\rho (e) \, = \, \rho(f+(e-f)) \, = \, \rho(f)+\rho(e-f) . \qedhere
\end{displaymath} \end{proof}

\begin{lem}[\cite{VonNeumannBook}]\label{lemma:rank.metric} Let $R$ be a regular ring and let $\rho \colon R \to [0,1]$ be a pseudo-rank function. Then \begin{displaymath}
	d_{\rho} \colon \, R \times R \, \longrightarrow \, [0,1], \quad (a,b) \, \longmapsto \, \rho(a-b)
\end{displaymath} is a pseudo-metric on $R$. Moreover, the following hold. \begin{itemize}
	\item[$(1)$] $\rho$ rank function on $R$ $\ \Longleftrightarrow \ $ $d_{\rho}$ metric on $R$.
	\item[$(2)$] For all $a,b,c,d \in R$, \begin{align*}
					\qquad d_{\rho} (a+b,c+d) \, &\leq \, d_{\rho} (a,c) + d_{\rho} (b,d) , \\
					\qquad d_{\rho} (ab,cd) \, &\leq \, d_{\rho} (a,c) + d_{\rho} (b,d) .
				\end{align*}
	\item[$(3)$] If $a,b \in R$, then $\vert \rho(a) - \rho(b) \vert \leq d_{\rho}(a,b)$.
\end{itemize} \end{lem}

\begin{proof} Proofs are to be found in~\cite[II.XVIII, Lemma~18.1, pp.~231--232]{VonNeumannBook}, as well as in~\cite[VI.5, Satz~5.1, p.~154]{MaedaBook}. \end{proof}

\begin{remark}\label{remark:topological.ring} Let $R$ be a regular ring and let $\rho \colon R \to [0,1]$ be a pseudo-rank function. It follows from Lemma~\ref{lemma:rank.metric}(2) that $R$ constitutes a topological ring with respect to the \emph{$\rho$-topology}, i.e., the topology on $R$ generated by the pseudo-metric $d_{\rho}$. Furthermore, Lemma~\ref{lemma:rank.metric}(2) entails that $d_{\rho}\vert_{\GL(R) \times \GL(R)}$ is a bi-invariant pseudo-metric on $\GL(R)$, whence $\GL(R)$ is a topological group with respect to the relative $\rho$-topology. \end{remark}

Let $R$ be a regular ring. A rank function $\rho \colon R \to [0,1]$ is called \emph{complete} if the metric space $(R,d_{\rho})$ is complete. Furthermore, $R$ is said to be a \emph{complete rank ring} if $R$ admits a complete rank function.

\begin{lem}\label{lemma:from.rank.to.dimension} Let $R$ be a regular ring and let $\rho \colon R \to [0,1]$ be a pseudo-rank function. Then \begin{displaymath}
	\delta_{\rho} \colon \, \lat(R) \, \longrightarrow \, [0,1], \quad aR \, \longmapsto \, \rho(a)
\end{displaymath} is a well-defined pseudo-dimension function on $\lat(R)$. Also, the following hold. \begin{itemize}
	\item[$(1)$] $\rho$ rank function on $R$ $\ \Longleftrightarrow \ $ $\delta_{\rho}$ dimension function on $\lat(R)$.
	\item[$(2)$] If $I \in \lat(R)$ and $a \in R$, then $\delta_{\rho}(aI) \leq \min \{ \rho(a), \delta_{\rho}(I) \}$.
	\item[$(3)$] If $I \in \lat(R)$ and $a,b \in R$, then $d_{\delta_{\rho}}(aI,bI) \leq 2 \min\{ \rho(a-b), \delta_{\rho}(I) \}$.
\end{itemize} \end{lem}

\begin{proof} The statement about $\delta_{\rho}$ being a pseudo-dimension function, as well as item~(1), follow by the argument in~\cite[VI.5, Satz~5.2, p.~154]{MaedaBook}.

(2) For all $a,b \in R$, \begin{displaymath}
	\delta_{\rho}(a(bR)) \, = \, \delta_{\rho}((ab)R) \, = \, \rho(ab) \, \leq \, \min \{ \rho(a) , \rho(b) \} \, = \, \min \{ \rho(a), \delta_{\rho}(bR) \} .
\end{displaymath}

(3) Let $I \in \lat(R)$ and $a,b \in R$. Without loss of generality, we may assume that $\delta_{\rho}(bI) \leq \delta_{\rho}(aI)$. It is straightforward to verify that $aI+bI = bI + (a-b)I$, which by modularity of $\delta_{\rho}$ entails that \begin{equation}\tag{$\ast$}\label{dimension}
	\delta_{\rho}(aI+bI) \, \leq \, \delta_{\rho}(bI) + \delta_{\rho}((a-b)I) .
\end{equation} Hence, \begin{align*}
	d_{\delta_{\rho}}(aI,bI) \, & \stackrel{\ref{lemma:birkhoff}(3)}{=} \, 2\delta_{\rho}(aI+bI) - \delta_{\rho}(aI) -\delta_{\rho}(bI) \, \leq \, 2(\delta_{\rho}(aI+bI) - \delta_{\rho}(bI)) \\
	&\qquad \stackrel{\eqref{dimension}}{\leq} \, 2 \delta_{\rho}((a-b)I) \, \stackrel{(2)}{\leq} \, 2\min\{ \rho(a-b), \delta_{\rho}(I) \} . \qedhere
\end{align*}\end{proof}

A \emph{continuous ring} is a regular ring $R$ such that $(\lat(R),{\subseteq})$ is a continuous geometry. Another fundamental result of von Neumann's work~\cite{VonNeumannBook} is the following characterization of irreducible, continuous rings.

\begin{thm}[\cite{VonNeumannBook}]\label{theorem:continuous.rank.rings} Let $R$ be a ring. \begin{itemize}
	\item[$(1)$] If $R$ is a complete rank ring, then $R$ is a continuous ring.	
	\item[$(2)$] If $R$ is an irreducible, continuous ring, then\footnote{Definition~\ref{definition:dimension.function} applies, since the continuous geometry $\lat(R)$ is irreducible by Theorem~\ref{theorem:irreducibility}.} \begin{displaymath}
					\qquad \rho_{R} \colon \, R \, \longrightarrow \, [0,1], \quad a \, \longmapsto \, \delta_{\lat(R)}(aR)
				\end{displaymath} is the unique rank function on $R$ and, moreover, $\rho_{R}$ is complete.
	\item[$(3)$] Suppose that $R$ is irreducible. Then $R$ is a continuous ring if and only if $R$ is a complete rank ring.
\end{itemize} \end{thm}

\begin{proof} (1) This is implicit in~\cite[II.XVIII, Proof of Theorem~18.1, p.~237]{VonNeumannBook}, but also stated and proved explicitly in~\cite[VI.5, Satz~5.3, p.~156]{MaedaBook}.
	
(2) If $R$ is an irreducible, continuous ring, then $\rho_{R}$ is a rank function on $R$ by~\cite[II.XVII, Theorem~17.1, p.~224]{VonNeumannBook}, unique as such by~\cite[II.XVII, Theorem~17.2, p.~226]{VonNeumannBook}, and complete by~\cite[II.XVII, Theorem~17.4, p.~230]{VonNeumannBook}. (An alternative reference to the proofs is~\cite[VII.2, pp.~162--165]{MaedaBook}.)

(3) This is due to~\cite[II.XVIII, Theorem~18.1, p.~237]{VonNeumannBook} (see also~\cite[VII.2, Satz~2.2, p.~165]{MaedaBook}), but it also follows from~(1) and~(2), of course. \end{proof}

\begin{definition}\label{definition:intrinsic.metric} Let $R$ be an irreducible, continuous ring. Then we let \begin{displaymath}
	d_{R} \defeq d_{\rho_{R}} \colon \, R \times R \, \longrightarrow \, [0,1], \quad (a,b) \, \longmapsto \, \rho_{R}(a-b)
\end{displaymath} denote the metric associated to $\rho_{R}$ via Lemma~\ref{lemma:rank.metric}. We call $R$ \emph{discrete} if the topology on $R$ generated by $d_{R}$ is discrete. \end{definition}

\begin{remark}\label{remark:range.2} Let $R$ be an irreducible, continuous ring. By Theorem~\ref{theorem:continuous.rank.rings}(2) and Remark~\ref{remark:range}, it follows that $R$ is non-discrete if and only if $\rho_{R}(R) = [0,1]$. \end{remark}

With the following lemma, we recollect some basic results concerning rank functions from the literature: for instance, see~\cite[Lemma~16.2(c)]{GoodearlBook} for the first, and~\cite[Corollary~6(1)]{Handelman76} or~\cite[Proposition~16.11(b)]{GoodearlBook} for the second item of Lemma~\ref{lemma:rank.functions.vs.units}. For the sake of convenience, let us recall that unital ring $R$ is said to be \emph{directly finite} (or \emph{Dedekind-finite}) if \begin{displaymath}
\forall a,b \in R \colon \quad ab = 1 \ \Longrightarrow \ ba=1 .
\end{displaymath}

\begin{lem}[\cite{Handelman76,GoodearlBook}]\label{lemma:rank.functions.vs.units} Let $R$ be a regular ring admitting a rank function $\rho \colon R \to [0,1]$. Then the following hold. \begin{itemize}
	\item[$(1)$] If $e \in \Ed(R)$ and $\rho(e) = 1$, then $e=1$.
	\item[$(2)$] $R$ is directly finite.
	\item[$(3)$] $\GL(R) = \{ a \in R \mid \rho(a) = 1 \}$.
	\item[$(4)$] $\GL(R)$ is closed in $R$ with respect to $d_{\rho}$.
\end{itemize}\end{lem}

\begin{proof} (1) If $e \in \Ed(R)$ and $\rho(e) = 1$, then $e \perp (1-e) \in \Ed(R)$ by Remark~\ref{remark:idempotent.difference}, hence $1 = \rho(1) = \rho(e+(1-e)) = \rho(e)+\rho(1-e) = 1 + \rho(1-e)$ and thus $\rho(1-e) = 0$, wherefore $1-e=0$, i.e., $e=1$.

(2) Let $a,b \in R$ with $ab=1$. Clearly, $ba \in \Ed(R)$. Furthermore, \begin{displaymath}
	1 \, = \, \rho(1) \, = \, \rho(abab) \, \leq \, \rho(ba)
\end{displaymath} and hence $\rho(ba) = 1$, which by~(1) implies that $ba=1$.

(3) Evidently, if $a \in \GL(R)$, then there exists $b \in R$ with $ab=1$, wherefore $1=\rho(1) = \rho(ab) \leq \rho(a)$ and thus $\rho(a) =1$. In order to verify the converse, let $a \in R$ such that $\rho(a)=1$. As $R$ is regular, there exists $b \in R$ such that $aba = a$. Since therefore both $ab \in \Ed(R)$ and $1=\rho(a) =\rho(aba) \leq \rho(ab)$, thus $\rho(ab) = 1$, item~(1) asserts that $ab=1$. Hence, $a \in \GL(R)$ by~(2). 

(4) As Lemma~\ref{lemma:rank.metric}(3) asserts that $\rho \colon R \to [0,1]$ is $1$-Lipschitz, hence continuous, with respect to $d_{\rho}$, this follows from~(3). \end{proof}

Given a unital ring $R$, let us now turn to the relation between (maximal) chains in $(\lat(R),{\subseteq})$ and those in $(\Ed(R),{\leq})$.

\begin{definition} Let $R$ be a unital ring. A \emph{nest}\footnote{This terminology has been coined by Ringrose~\cite{ringrose} in the realm of operator algebras.} in $R$ is a chain in $(\Ed(R),{\leq})$. Let $\Nest (R)$ denote the set of all nests in $R$, and let $\Nestmax (R) \defeq \Max (\Nest (R),{\subseteq})$. By a \emph{flag} over $R$, we mean a chain in $(\lat(R),{\subseteq})$. The set of all flags over $R$ will be denoted by $\Flag (R)$. Finally, we let $\Flagmax (R) \defeq \Max (\Flag (R),{\subseteq})$. \end{definition}

\begin{remark}\label{remark:lambda.map} Let $R$ be a unital ring and let $E \in \Nest (R)$. \begin{itemize}
	\item[$(1)$] The map $\Lambda \colon (\Ed(R),{\leq}) \to (\lat(R),{\subseteq}), \, e \mapsto eR$ is easily seen to be monotone. In particular, $\Lambda (E) = \{ eR \mid e \in E \}$ is a flag over $R$.
	\item[$(2)$] For any two elements $e,f \in \Ed(R)$ with $e\leq f$ and $eR = fR$, it follows that $e = ef = f$. Hence, the restriction $\Lambda\vert_{E}$ is injective.
\end{itemize} \end{remark}

\begin{lem}\label{lemma:maximal.nests} Let $R$ be a complete rank ring and let $F \in \Flag(R)$. If $E$ is a maximal chain in $(\Lambda^{-1}(F),{\leq})$, then $\Lambda (E) = F$. \end{lem}

\begin{proof} Suppose that $\rho \colon R \to [0,1]$ is a complete rank function. Let $E$ be a maximal chain in $(\Lambda^{-1}(F),{\leq})$. To show that $\Lambda (E) = F$, let us consider any element $I \in F$. Define \begin{align*}
	E_{0} \, \defeq \, \{ e \in E \mid eR \subseteq I \}, \qquad  &t_{0} \, \defeq \, \sup \{ \rho(e) \mid e\in E_{0}\} , \\
	E_{1} \, \defeq \, \{ e \in E \mid I \subseteq eR \} , \qquad &t_{1} \, \defeq \, \inf \{ \rho(e) \mid e\in E_{1}\} .
\end{align*} Since $E$ is a nest in $R$, we find sequences $(e_{0,n})_{n \in \N} \in E^{\N}$ and $(e_{1,n})_{n \in \N} \in E^{\N}$ such that, for every $n \in \N$, \begin{itemize}
	\item[$(1)$] $e_{0,n} \in E_{0}$ and $e_{1,n} \in E_{1}$,
	\item[$(2)$] $e_{0,n} \leq e_{0,n+1}$ and $e_{1,n+1} \leq e_{1,n}$,
	\item[$(3)$] $\rho(e_{0,n}) \geq t_{0}-\tfrac{1}{n+1}$ and $\rho(e_{1,n}) \leq t_{1}+\tfrac{1}{n+1}$.
\end{itemize} For each $i \in \{ 0,1\}$ and for any $m,n \in \N$, we infer that \begin{displaymath}
	d_{\rho}(e_{i,m},e_{i,n}) \, = \, \rho(e_{i,m}-e_{i,n}) \, \stackrel{(2)+\ref{lemma:rank.estimates}(2)}{=} \, \vert \rho(e_{i,m}) - \rho(e_{i,n}) \vert \, \stackrel{(1)+(3)}{\leq} \, \max \left\{ \tfrac{1}{m+1}, \tfrac{1}{n+1} \right\} .
\end{displaymath} Therefore, both $(e_{0,n})_{n \in \N}$ and $(e_{1,n})_{n \in \N}$ are Cauchy sequences in $(R,d_{\rho})$. Completeness of the metric space $(R,d_{\rho})$ thus asserts the existence of the limits $e_{0} \defeq \lim\nolimits_{n \to \infty} e_{0,n} \in R$ and $e_{1} \defeq \lim\nolimits_{n \to \infty} e_{1,n} \in R$. By Lemma~\ref{lemma:rank.metric}(1)$+$(2) and Lemma~\ref{lemma:topological.semigroups}(1), it follows that $\{ e_{0},e_{1}\}\subseteq \Ed(R)$. For each $i \in \{ 0,1\}$, \begin{displaymath}
	\rho(e_{i}) \, \stackrel{\ref{lemma:rank.metric}(3)}{=} \, \lim\nolimits_{n\to \infty} \rho(e_{i,n}) \, \stackrel{(1)+(3)}{=} \, t_{i} .
\end{displaymath} Since $\rho\vert_{\Ed(R)} \colon (\Ed(R),\leq) \to ([0,1],\leq)$ is monotone due to Lemma~\ref{lemma:rank.estimates}(2) and we have $E \cup \{ e_{0},e_{1}\} \in \Nest (R)$ due to Lemma~\ref{lemma:rank.metric}(1)$+$(2) and Lemma~\ref{lemma:topological.semigroups}(2), we conclude that \begin{equation}\tag{4}\label{nest}
	(\forall f \in E_{0} \colon \, f \leq e_{0}), \qquad e_{0} \leq e_{1}, \qquad (\forall f \in E_{1} \colon \, e_{1} \leq f ) .
\end{equation} Furthermore, \begin{align*}
	e_{0}R \, \stackrel{\ref{lemma:from.rank.to.dimension}(3)}{=} \, \lim\nolimits_{n \to \infty} e_{0,n}R \, &= \, \lim\nolimits_{n \to \infty} I \cap e_{0,n}R \\
	& \stackrel{\ref{lemma:from.rank.to.dimension}+\ref{lemma:birkhoff}(2)}{=} \, I \cap \left( \lim\nolimits_{n \to \infty} e_{0,n}R\right) \, \stackrel{\ref{lemma:from.rank.to.dimension}(3)}{=} \, I \cap e_{0}R , \\
	e_{1}R \, \stackrel{\ref{lemma:from.rank.to.dimension}(3)}{=} \, \lim\nolimits_{n \to \infty} e_{1,n}R \, &= \, \lim\nolimits_{n \to \infty} I + e_{1,n}R \\
	& \stackrel{\ref{lemma:from.rank.to.dimension}+\ref{lemma:birkhoff}(2)}{=} \, I + \left( \lim\nolimits_{n \to \infty} e_{1,n}R\right) \, \stackrel{\ref{lemma:from.rank.to.dimension}(3)}{=} \, I + e_{1}R ,
\end{align*} i.e., $e_{0}R \subseteq I \subseteq e_{1}R$. By Lemma~\ref{lemma:intermediate.idempotent}, there exists $f \in \Ed(R)$ such that $e_{0} \leq f \leq e_{1}$ and $I = fR$. The former, combined with~\eqref{nest}, implies that $E \cup \{ f\} \in \Nest (R)$, while the latter then entails that $\Lambda (E \cup \{ f \}) \subseteq F$. Hence, maximality of $E$ asserts that $E = E \cup \{ f \}$, that is, $f \in E$. Thus, $I = fR \in \Lambda (E)$ as desired. \end{proof}

\begin{prop}\label{proposition:flags.vs.nests} Let $R$ be a complete rank ring. Then \begin{displaymath}
	\Nest(R) \, \longrightarrow \, \Flag(R), \quad E \, \longmapsto \, \Lambda (E)
\end{displaymath} is a well-defined surjection. \end{prop}

\begin{proof} This is an immediate consequence of Remark~\ref{remark:lambda.map}(1), Lemma~\ref{lemma:maximal.nests}, and the Hausdorff maximal principle. \end{proof}

\begin{thm}\label{theorem:maximal.nests} Let $R$ be a complete rank ring. Then \begin{displaymath}
	\Nestmax(R) \, \longrightarrow \, \Flagmax(R), \quad E \, \longmapsto \, \Lambda (E)
\end{displaymath} is a well-defined surjection. \end{thm}

\begin{proof} To establish well-definedness, we need to show that $\Lambda (E) \in \Flagmax(R)$ for every $E \in \Nestmax(R)$. To this end, let $E \in \Nestmax(R)$. Clearly, $\Lambda (E) \in \Flag (R)$ by Remark~\ref{remark:lambda.map}(1). Suppose that $F \in \Flag (R)$ with $\Lambda (E) \subseteq F$, i.e., $E \subseteq \Lambda^{-1}(F)$. Due to the Hausdorff maximal principle, there exists a maximal chain $E'$ in $(\Lambda^{-1}(F),{\leq})$ with $E \subseteq E'$. By Lemma~\ref{lemma:maximal.nests}, thus $\Lambda (E') = F$. Moreover, as $E'$ is a nest in $R$, maximality of $E$ asserts that $E=E'$, hence $\Lambda (E) = \Lambda (E') = F$. This proves that $\Lambda (E) \in \Flagmax (R)$.

To prove surjectivity, consider any $F \in \Flagmax(R)$. By the Hausdorff maximal principle, there exists a maximal chain $E$ in $(\Lambda^{-1}(F),{\leq})$. Since $\Lambda (E) = F$ due to Lemma~\ref{lemma:maximal.nests}, it remains to verify that $E$ is a maximal nest in $R$. For this purpose, let $E' \in \Nest (R)$ with $E \subseteq E'$. Then $F = \Lambda (E) \subseteq \Lambda (E') \in \Flag (R)$, thus $\Lambda (E) = F = \Lambda (E')$ by maximality of $F$, whence $E=E'$ due to Remark~\ref{remark:lambda.map}(2). This entails that $E \in \Nestmax(R)$, as desired. \end{proof}

\begin{cor}\label{corollary:maximal.nests} Let $R$ be an irreducible, continuous ring. A nest $E$ in $R$ is maximal if and only if $\rho_{R}(E) = \rho_{R}(R)$. \end{cor}

\begin{proof} Recall that $R$ is a complete rank ring thanks to Theorem~\ref{theorem:continuous.rank.rings}(2). For every $E \in \Nest(R)$, we conclude that \begin{align*}
	E \in \Nestmax (R) \ \ &\stackrel{\ref{remark:lambda.map}(2)+\ref{theorem:maximal.nests}}{\Longleftrightarrow} \ \ \Lambda (E) \in \Flagmax (R) \\
	&\stackrel{\ref{theorem:irreducibility}+\ref{theorem:maximal.flags}}{\Longleftrightarrow} \ \ \delta_{\lat(R)}(\Lambda (E)) = \delta_{\lat(R)}(\lat(R)) \ \ \stackrel{\ref{theorem:continuous.rank.rings}(2)}{\Longleftrightarrow} \ \ \rho_{R}(E) = \rho_{R}(R) . \qedhere
\end{align*} \end{proof}

\begin{cor}\label{corollary:maximal.nests.2} Let $R$ be a non-discrete irreducible, continuous ring. \begin{itemize}
	\item[$(1)$] A nest $E$ in $R$ is maximal if and only if $\rho_{R}(E) = [0,1]$.
	\item[$(2)$] If $E \in \Nestmax (R)$, then $\rho_{R}\vert_{E} \colon (E,{\leq}) \to ([0,1],{\leq})$ is an isomorphism of linearly ordered sets.
\end{itemize} \end{cor}

\begin{proof} (1) This follows by Corollary~\ref{corollary:maximal.nests} and Remark~\ref{remark:range.2}.

(2) Let $E \in \Nestmax (R)$. Since $\Lambda\vert_{E} \colon E \to \lat(R), \, e \mapsto eR$ is monotone and injective by Remark~\ref{remark:lambda.map} and $\delta_{\lat(R)}$ is strictly positive, it follows that \begin{displaymath}
	\rho_{R}\vert_{E} \, \stackrel{\ref{theorem:continuous.rank.rings}(2)}{=} \, {\delta_{\lat(R)}} \circ {\Lambda\vert_{E}}
\end{displaymath} is monotone and injective. Combined with~(1), this implies the claim. \end{proof}

\section{Algebraic elements in continuous rings}\label{section:algebraic.elements}

The purpose of this section is to recall another remarkable result by von Neumann~\cite{VonNeumann37,Halperin62} (see Theorem~\ref{theorem:neumann} and Corollary~\ref{corollary:neumann} below) and deduce a slight variation thereof (Proposition~\ref{proposition:algebraic.units}). For this purpose, let us establish some notation. If $R$ is any commutative unital ring, then we let \begin{displaymath}
	\Irr (R) \, \defeq \, \{ a \in R\setminus (\GL(R) \cup \{ 0\}) \mid \forall b,c \in R\colon \, a = bc \, \Longrightarrow \, \{ b,c\} \cap \GL(R) \ne \emptyset \} 
\end{displaymath} denote the set of all \emph{irreducible} elements of $R$.

Let now $K$ be a field and let $R$ be any unital $K$-algebra\footnote{In this article, by definition, any algebra over a field is required to be associative.}. Denote by $K[X]$ the polynomial ring over $K$ and by $\deg \colon K[X]\setminus \{ 0 \} \to \N$ the usual degree function. Moreover, recall that, for each $a \in R$, the induced evaluation map \begin{displaymath}
	K[X] \, \longrightarrow \, R, \quad p = \sum\nolimits_{i=0}^{n} k_{i}X^{i} \, \longmapsto \, p^{R}(a) \defeq \sum\nolimits_{i=0}^{n} k_{i}a^{i}
\end{displaymath} is a unital ring homomorphism. An element $a \in R$ is called \emph{algebraic over $K$} if there exists $p \in K[X]\setminus \{ 0 \}$ such that $p^{R}(a) = 0$. An element $a \in R$ is said to be \emph{purely transcendental over $K$} if $p^{R}(a) \in \GL(R)$ for every $p \in K[X]\setminus \{ 0 \}$. For every $a \in R$, we let \begin{displaymath}
	\Qalg_{R}(a) \, \defeq \, \left\{ p \in \Irr (K[X]) \left\vert \, p^{R}(a) \notin \GL(R) \right\} \right.
\end{displaymath} and \begin{displaymath}
	\Qalg_{R}^{\ast}(a) \, \defeq \, \{ p_{1}\cdots p_{n} \mid n \in \N_{>0}, \, p_{1},\ldots,p_{n} \in \Qalg_{R}(a) \} ,
\end{displaymath} i.e., $\Qalg_{R}^{\ast}(a)$ is the multiplicative subsemigroup of $K[X]$ generated by $\Qalg_{R}(a)$. Since $K[X]$ constitutes a unique factorization domain, an element $a \in R$ is purely transcendental if and only if $\Qalg_{R}(a) = \emptyset$.

\begin{remark} Let $R$ be an irreducible, continuous ring. Then $\cent(R)$ is a field due to Theorem~\ref{theorem:irreducibility}, whence $R$ naturally constitutes a unital $\cent(R)$-algebra. Thanks to Lemma~\ref{lemma:rank.functions.vs.units}(3), the following hold: \begin{itemize}
	\item[$(1)$] $\Qalg_{R}(a) = \left\{ p \in \Irr (\cent(R)[X]) \left\vert \, \rho_{R}(p^{R}(a)) < 1 \right\} \right.$ for every $a \in R$.
	\item[$(2)$] An element $a \in R$ is purely transcendental over $\cent(R)$ if and only if $\rho_{R}(p^{R}(a)) = 1$ for every (irreducible) $p \in \cent(R)[X]\setminus \{ 0 \}$.
\end{itemize} \end{remark}

While the following result due to von Neumann was announced in~\cite{VonNeumann37}, the details of von Neumann's unpublished proof were edited and made available to the public later by Halperin~\cite{Halperin62}.

\begin{thm}[\cite{VonNeumann37}; \cite{Halperin62}, Theorem~8.1]\label{theorem:neumann} Suppose that $R$ is a non-discrete irreducible, continuous ring. Let $a \in R$ and $\epsilon \in \R_{>0}$. \begin{itemize}
	\item[$(1)$] If $a$ is not purely transcendental over $\cent(R)$, then there exist $p \in \Qalg^{\ast}_{R}(a)$ and $b \in R$ such that \begin{displaymath}
			\qquad \rho_{R}(a-b) < \epsilon , \qquad p^{R}(b) = 0.
		\end{displaymath}
	\item[$(2)$] If $a$ is purely transcendental over $\cent(R)$, then for every $p \in \cent(R)[X]\setminus \{ 0 \}$ with $\deg p > 1/\epsilon$ there exists $b \in R$ such that \begin{displaymath}
			\qquad \rho_{R}(a-b) < \epsilon , \qquad p^{R}(b) = 0.
		\end{displaymath}
\end{itemize} \end{thm}

\begin{cor}[\cite{VonNeumann37,Halperin62}]\label{corollary:neumann} If $R$ is a non-discrete irreducible, continuous ring, then $\{ a \in R \mid a \textit{ algebraic over } \cent(R) \}$ is dense in $R$ with respect to $d_{R}$. \end{cor}

The chief objective of this section is to deduce Proposition~\ref{proposition:algebraic.units}. For this purpose, the following lemma will be useful.

\begin{lem}\label{lemma:neumann} Let $K$ be a field, let $R$ be a unital $K$-algebra, and let $a \in R$. \begin{itemize}
	\item[$(1)$] If $a \in \GL(R)$, then $\Qalg^{\ast}_{R}(a) \subseteq K[X]\cdot X + (K\setminus \{ 0 \})$.
	\item[$(2)$] If $p^{R}(a) = 0$ for some $p \in K[X]\cdot X + (K\setminus \{ 0 \})$, then $a \in \GL(R)$.
\end{itemize} \end{lem}

\begin{proof} (1) Suppose that $a \in \GL(R)$. We first show that \begin{equation}\tag{$\ast$}\label{irreducible}
	\Qalg_{R}(a) \, \subseteq \, K[X]\cdot X + (K\setminus \{ 0 \}) .
\end{equation} To this end, let $p \in \Qalg_{R}(a)$. Since $K[X] = K[X] \cdot X + K$, there exist $q \in K[X]$ and $c \in K$ such that $p = q\cdot X + c$. We argue that $c \ne 0$: indeed, if $c=0$, then irreducibility of $p=q\cdot X$ would necessitate that $q \in K\setminus \{ 0 \}$, and hence $p^{R}(a) = qa \in \GL(R)$, a contradiction. Thus, \begin{displaymath}
	p \, = \, q \cdot X + q \, \in \, K[X]\cdot X + (K\setminus \{ 0 \}) .
\end{displaymath} This proves~\eqref{irreducible}. As $K[X]\cdot X + (K\setminus \{ 0 \})$ is a subsemigroup of the multiplicative semigroup of $K[X]$, assertion~\eqref{irreducible} entails that $\Qalg^{\ast}_{R}(a) \subseteq K[X]\cdot X + (K\setminus \{ 0 \})$. 
	
(2) Assume that $p^{R}(a) = 0$ for some $p \in K[X]\cdot X + (K\setminus \{ 0 \})$. Then there exist $q \in K[X]$ and $c \in K\setminus \{ 0 \}$ such that $p = q\cdot X + c$. In turn, $0=p^{R}(a) = q^{R}(a) a + c$. Since $c \in K\setminus \{ 0 \}$, we may define $b \defeq -c^{-1}q^{R}(a) \in R$. For every $t \in K[X]$, \begin{displaymath}
	a \cdot t^{R}(a) \, = \, (X \cdot t)^{R}(a) \, = \, (t \cdot X)^{R}(a) \, = \, t^{R}(a) \cdot a .
\end{displaymath} In particular, $ab = ba$. Furthermore, \begin{displaymath}
	ba \, = \, -c^{-1}q^{R}(a)a \, = \, 1
\end{displaymath} and therefore $a \in \GL(R)$, as desired. \end{proof}

\begin{prop}\label{proposition:algebraic.units} If $R$ is a non-discrete irreducible, continuous ring, then the set $\{ a \in \GL(R) \mid a \textit{ algebraic over } \cent(R) \}$ is dense in $\GL(R)$ with respect to $d_{R}$. \end{prop}

\begin{proof} Suppose that $R$ is a non-discrete irreducible, continuous ring. Consider any $a \in \GL(R)$ and $\epsilon \in \R_{>0}$. We are going to prove the existence of some $\cent(R)$-algebraic element $b \in \GL(R)$ such that $\rho_{R}(a-b) < \epsilon$. The argument proceeds by case analysis.
	
\textit{Case 1: $a$ is not purely transcendental over $\cent(R)$}. By Theorem~\ref{theorem:neumann}(1), there exist $p \in \Qalg^{\ast}_{R}(a)$ and $b \in R$ such that $\rho_{R}(a-b) < \epsilon$ and $p^{R}(b) = 0$. According to Lemma~\ref{lemma:neumann}(1), as $a \in \GL(R)$, it follows that $p \in \cent(R)[X] \cdot X + (\cent(R)\setminus \{ 0 \})$. Since $p^{R}(b) = 0$, thus Lemma~\ref{lemma:neumann}(2) asserts that $b \in \GL(R)$, as desired.
	
\textit{Case 2: $a$ is purely transcendental over $\cent(R)$}. Consider $m \defeq \lceil 1/\epsilon\rceil + 1$ and \begin{displaymath}
	p \, \defeq \, X^{m} - 1 \, \in \, \cent(R)[X]\setminus \{ 0 \} .
\end{displaymath} By Theorem~\ref{theorem:neumann}(2), there exists $b \in R$ such that $\rho_{R}(a-b) < \epsilon$ and $p^{R}(b) = 0$. Since $p \in \cent(R)[X] \cdot X + (\cent(R)\setminus \{ 0 \})$, another application of Lemma~\ref{lemma:neumann}(2) yields that $b \in \GL(R)$, which completes our case analysis. \end{proof}

\section{Continuous triangularization}\label{section:triangularization}

The purpose of this section is to prove that, if $R$ is a non-discrete irreducible, continuous ring, then any subring of $R$ that includes and is finite-dimensional over the center $\cent(R)$ stabilizes a maximal flag over $R$, with respect to the natural action of $R$ on $\lat(R)$ by multiplication from the left. This result, Theorem~\ref{theorem:invariant.flags}, particularly applies to every single algebraic element of such ring and thus may be combined with the material of Section~\ref{section:algebraic.elements}, as substantiated by Corollary~\ref{corollary:algebraic.units}. Upon constructing a suitable family of continuous closure operators on the space of principal right ideals (Lemma~\ref{lemma:hull.2} and Lemma~\ref{lemma:hull.3}), the proof of Theorem~\ref{theorem:invariant.flags} ultimately relies on an application of the classical intermediate value theorem.

Before defining stabilizer rings (Definition~\ref{definition:stabilizer}), let us point out a general fact, noted in the context of operator algebras by Ringrose~\cite[Lemma~3.1]{ringrose}.

\begin{lem}[cf.~\cite{ringrose}, Lemma~3.1]\label{lemma:ringrose} Let $R$ be a unital ring, and let $e \in \Ed(R)$, $I \defeq eR$, $a \in R$. Then the following are equivalent. \begin{itemize}
	\item[$(1)$] $aI \subseteq I$.
	\item[$(2)$] $eae = ae$.
	\item[$(3)$] $(1-e)a(1-e) = (1-e)a$.
\end{itemize}\end{lem}

\begin{proof} (2)$\Longleftrightarrow$(3). This equivalence follows immediately from the fact that \begin{displaymath}
	(1-e)a(1-e)-(1-e)a \, = \, a-ae-ea+eae-a+ea \, = \, eae-ae  .
\end{displaymath}
	
(1)$\Longleftrightarrow$(2). If $aI\subseteq I$, then $ae \in aI \subseteq I$, thus $eae = ae$. Conversely, if $eae = ae$, then $ab = aeb = eaeb \in eR = I$ for all $b \in I$, i.e., $aI \subseteq I$. \end{proof}

\begin{definition}\label{definition:stabilizer} Let $R$ be a unital ring. For every subset $F \subseteq \lat(R)$, we define \begin{displaymath}
	R_{F} \, \defeq \, \{ a \in R \mid \forall I \in F \colon \, aI \subseteq I \} .
\end{displaymath} For every subset $E \subseteq \Ed(R)$, we define \begin{displaymath}
	R_{E} \, \defeq \, \{ a \in R \mid \forall e \in E \colon \, eae = ae \} .
\end{displaymath} \end{definition}

\begin{prop}\label{proposition:nest.subring} Let $R$ be a unital ring. \begin{itemize}
	\item[$(1)$] If $F \subseteq \lat(R)$, then $R_{F}$ is a unital subring of $R$.
	\item[$(2)$] If $E \subseteq \Ed(R)$, then $R_{E} = R_{\{ eR \mid e \in E \}}$ is a unital subring of $R$.
	\item[$(3)$] If $E \in \Nest (R)$, then $E \subseteq R_{E}$.
	\item[$(4)$] Suppose that $R$ is regular and that $\rho \colon R \to [0,1]$ is a rank function. For any $E \subseteq \Ed(R)$ and any $F \subseteq \lat(R)$, the sets $R_{E}$ and $R_{F}$ are closed in $R$ with respect to the $\rho$-topology.
\end{itemize} \end{prop}

\begin{proof} (1) This follows by a straightforward calculation.
	
(2) If $E \subseteq \Ed(R)$, then $R_{E} = R_{\{ eR \mid e \in E \}}$ due to Lemma~\ref{lemma:ringrose}, thus $R_{E}$ is a unital subring of $R$ by~(1).

(3) Let $E \in \Nest (R)$. In order to verify that $E \subseteq R_{E}$, let $f \in E$. If $e \in E$, then $E$ being a chain in $(\Ed(R),{\leq})$ implies that $e\leq f$ and therefore $efe = ee =  e = fe$, or $f\leq e$ and hence $efe=fe$. Thus, $f \in R_{E}$ as desired.

(4) First, let $E \subseteq \Ed(R)$. As $d_{\rho}$ is a metric by Lemma~\ref{lemma:rank.metric}(1), the $\rho$-topology is a Hausdorff topology on $R$, whence $D \defeq \{ (x,x) \mid x \in R\}$ is closed in $R \times R$. Moreover, for each $e \in E$, the map $\xi_{e} \colon R \to R \times R, \, a \mapsto (eae,ae)$ is continuous by Lemma~\ref{lemma:rank.metric}(2) (see also Remark~\ref{remark:topological.ring}). Thus, $R_{E} = \bigcap_{e \in E} \xi_{e}^{-1}(D)$ is closed in $R$. In conclusion, if $F \subseteq \lat(R)$, then \begin{displaymath}
	R_{F} \, \stackrel{\ref{lemma:ringrose}+\ref{remark:regular}}{=} \, R_{\{ e \in \Ed(R) \mid eR \in F\}}
\end{displaymath} is closed in $R$. \end{proof}

Preparing our later study of unit groups of stabilizer rings, we recollect a well-known persistence property of the class of directly finite rings. Let us recall that, if $R$ is a unital ring and $e \in \Ed(R)$, then $eRe = \{ eae \mid a \in R\}$ is a subring of $R$, with multiplicative unit $e$.

\begin{lem}\label{lemma:directly.finite} Let $R$ be a directly finite ring, $e \in \Ed(R)$. Then the following hold. \begin{itemize}
	\item[$\mathrm{(A)}$] $eRe$ is directly finite.
	\item[$\mathrm{(B)}$] Consider $I \defeq eR$ and let $a \in \GL(R)$. The following are equivalent. \begin{itemize}
		\item[$(1)$] $aI = I$.
		\item[$(2)$] $eae = ae$.
		\item[$(3)$] $ea^{-1}e = a^{-1}e$.
	\end{itemize}
\end{itemize} \end{lem} 

\begin{proof} $\mathrm{(A)}\,$ If $a,b \in eRe$ and $ab=e$, then \begin{displaymath}
	(a+1-e)(b+1-e) \, = \, ab + a(1-e) + (1-e)b + (1-e)^{2} \, = \, e + 1 - e \, = \, 1 ,
\end{displaymath} whence direct finiteness of $R$ implies that \begin{displaymath}
	1 \, = \, (b+1-e)(a+1-e) \, = \, ba + b(1-e) + (1-e)a + (1-e)^{2} \, = \, ba + 1-e ,
\end{displaymath} i.e., $ba = e$. This shows that $eRe$ is directly finite.

$(\mathrm{B})\,$ First of all, we claim that \begin{equation}\tag{$\ast$}\label{inverse}
	\forall a \in \GL(R) \colon \quad eae = ae \ \Longrightarrow \ ea^{-1}e = a^{-1}e .
\end{equation} Indeed, if $a \in \GL(R)$ and $eae = ae$, then \begin{displaymath}
	\left(ea^{-1}e\right)(eae) \, = \, ea^{-1}eae \, = \, ea^{-1}ae \, = \, ee \, = \, e,
\end{displaymath} which by $\mathrm{(A)}$ entails that \begin{displaymath}
	e \, = \, (eae)\left(ea^{-1}e\right) \, = \, eaea^{-1}e \, = \, aea^{-1}e
\end{displaymath} and hence $a^{-1}e = ea^{-1}e$. We proceed to verifying the desired equivalences. So, consider $I \defeq eR$ and let $a \in \GL(R)$. Of course, (2)$\Longleftrightarrow$(3) thanks to~\eqref{inverse}, and (1)$\Longrightarrow$(2) by Lemma~\ref{lemma:ringrose}. In order to prove that (2)$\Longrightarrow$(1), suppose that $eae = ae$. Then also $ea^{-1}e = a^{-1}e$ due to~\eqref{inverse}. Invoking Lemma~\ref{lemma:ringrose} twice, we see that $aI \subseteq I = aa^{-1}I \subseteq aI$, thus $aI = I$. \end{proof}

\begin{prop}\label{proposition:units.in.nest.stabilizers} Let $R$ be a directly finite ring. For every $E\in \Nest(R)$, \begin{displaymath}
	\GL(R_{E}) \, = \, R_{E} \cap \GL(R) .
\end{displaymath} \end{prop}

\begin{proof} This is a direct consequence of Lemma~\ref{lemma:directly.finite}$\mathrm{(B)}$. \end{proof}

We now begin to work towards Theorem~\ref{theorem:invariant.flags}, whose proof will involve the intermediate value theorem. To accommodate this application of the intermediate value theorem, we introduce a certain family of closure operators (Lemma~\ref{lemma:hull.2}) and examine their continuity properties (Lemma~\ref{lemma:hull.3}). Let us start off with a basic fact.

\begin{lem}\label{lemma:hull.1} Let $R$ be a unital ring and let $S$ be a subring of $R$ containing $\cent(R)$. If $I$ is any right ideal of $R$ and $T$ is any generating subset of the $\cent(R)$-module $S$, then $\sum\nolimits_{s \in S} sI = \sum\nolimits_{t \in T} tI$. \end{lem}

\begin{proof} Suppose that $T \subseteq S$ generates the $\cent(R)$-module $S$ and consider any right ideal $I \subseteq R$. Evidently, $\sum_{t \in T} tI \subseteq \sum_{s \in S} sI$. Conversely, if $s \in S$, then there exist $n \in \N$, $a_{1},\ldots,a_{n} \in \cent(R)$ and $t_{1},\ldots,t_{n} \in T$ with $s = \sum_{i=1}^{n} a_{i}t_{i}$, which entails that \begin{displaymath}
	sI \, \subseteq \, \sum\nolimits_{i=1}^{n} a_{i}t_{i}I \, = \, \sum\nolimits_{i=1}^{n} t_{i}Ia_{i} \, \subseteq \, \sum\nolimits_{i=1}^{n} t_{i}I \, \subseteq \, \sum\nolimits_{t \in T} tI .
\end{displaymath} Hence, $\sum\nolimits_{s \in S} sI = \sum\nolimits_{t \in T} tI$ as desired. \end{proof}

\begin{definition} If $R$ is a unital ring and $M$ is a (left) $R$-module, then we let \begin{displaymath}
	\dim_{R}(M) \, \defeq \, \inf \{ \vert T \vert \mid T \text{ generating subset of the $R$-module } M\} \, \in \, {\N} \cup {\{ \infty \}} .
\end{displaymath} \end{definition}

\begin{remark}\label{remark:algebraicity} Let $K$ be a field, let $R$ be a unital $K$-algebra, let $a \in R$ and let $S$ denote the unital subalgebra of $R$ generated by $a$. Then $a$ is algebraic over $K$ if and only if $\dim_{K}(S) < \infty$. \end{remark}

\begin{lem}\label{lemma:hull.2} Let $R$ be a regular ring and let $S$ be a subring of $R$ with $\cent(R) \subseteq S$ and $\dim_{\cent(R)}(S) < \infty$. Then \begin{displaymath}
	\Gamma_{S} \colon \, \lat(R) \, \longrightarrow \, \lat(R) , \quad I \, \longmapsto \, \sum\nolimits_{s \in S} sI 
\end{displaymath} is well defined. Furthermore, the following hold.  \begin{itemize}
	\item[$(1)$] $s\Gamma_{S}(I) \subseteq \Gamma_{S}(I)$ for all $I \in \lat(R)$ and $s \in S$.
	\item[$(2)$] $\Gamma_{S}(I+J) = \Gamma_{S}(I)+\Gamma_{S}(J)$ for all $I,J \in \lat(R)$.
	\item[$(3)$] If $I,J \in \lat(R)$ and $I \subseteq J$, then $\Gamma_{S}(I) \subseteq \Gamma_{S}(J)$.
	\item[$(4)$] $I \subseteq \Gamma_{S}(I) = \Gamma_{S}(\Gamma_{S}(I))$ for every $I \in \lat(R)$.
\end{itemize} \end{lem}

\begin{proof} Since $\lat(R)$ is closed under finite sums~\cite[II.II, Theorem~2.3, p.~71]{VonNeumannBook} and $\dim_{\cent(R)}(S) < \infty$, Lemma~\ref{lemma:hull.1} asserts that $\Gamma_{S} \colon \lat(R) \to \lat(R)$ is well defined. 
	
(1) Since $S$ is closed under multiplication, if $s \in S$ and $I \in \lat(R)$, then \begin{displaymath}
	s\Gamma_{S}(I) \, = \, \sum\nolimits_{t \in S} stI \, \subseteq \, \sum\nolimits_{t \in S} tI \, = \, \Gamma_{S}(I) .
\end{displaymath}
	
(2) For any two $I,J \in \lat(R)$, \begin{align*}
	\Gamma_{S}(I+J) \, &= \, \sum\nolimits_{s \in S} s(I+J) \, = \, \sum\nolimits_{s \in S} sI+sJ \\
		& = \, \left( \sum\nolimits_{s \in S} sI \right) + \left( \sum\nolimits_{s \in S} sJ \right) \, = \, \Gamma_{S}(I) + \Gamma_{S}(J) .
\end{align*}
	
(3) If $I,J \in \lat(R)$ and $I \subseteq J$, then $J = I+J$ and therefore \begin{displaymath}
	\Gamma_{S}(J) \, = \, \Gamma_{S}(I+J) \, \stackrel{(2)}{=} \, \Gamma_{S}(I) + \Gamma_{S}(J) ,
\end{displaymath} hence $\Gamma_{S}(I) \subseteq \Gamma_{S}(J)$.

(4) For every $I \in \lat(R)$, since $1_{R} \in \cent(R) \subseteq S$, \begin{displaymath}
	I \, \subseteq \, \Gamma_{S}(I) \, \stackrel{(1)}{=} \, \sum\nolimits_{s \in S} s\Gamma_{S}(I) \, = \, \Gamma_{S}(\Gamma_{S}(I)) . \qedhere
\end{displaymath} \end{proof}

\begin{lem}\label{lemma:hull.3} Let $R$ be a regular ring and let $S$ be a subring of $R$ with $\cent(R) \subseteq S$ and $\dim_{\cent(R)}(S) < \infty$. Furthermore, let $\rho \colon R \to [0,1]$ be a pseudo-rank function, and consider\footnote{cf.~Lemma~\ref{lemma:from.rank.to.dimension}} $\delta \defeq \delta_{\rho} \colon \lat(R) \to [0,1]$. Then, for all $I,J \in \lat(R)$,  \begin{itemize}
	\item[$(1)$] $\delta(\Gamma_{S}(I)) \leq \dim_{\cent(R)}(S)\cdot \delta(I)$,
	\item[$(2)$] $d_{\delta}(\Gamma_{S}(I),\Gamma_{S}(J)) \leq \dim_{\cent(R)}(S)\cdot d_{\delta}(I,J)$.
\end{itemize} \end{lem}

\begin{proof} (1) Pick any finite subset $T \subseteq S$ generating the $\cent(R)$-module $S$ such that $\vert T \vert = \dim_{\cent(R)}(S)$. For every $I \in \lat(R)$, \begin{displaymath}
	\delta (\Gamma_{S}(I)) \, \stackrel{\ref{lemma:hull.1}}{=} \, \delta \left(\sum\nolimits_{t \in T} tI \right) \, \stackrel{\ref{lemma:from.rank.to.dimension}}{\leq} \, \sum\nolimits_{t\in T}\delta(tI) \, \stackrel{\ref{lemma:from.rank.to.dimension}(2)}{\leq} \, \vert T \vert \cdot \delta(I) \, = \, \dim_{\cent(R)}(S) \cdot \delta (I) .
\end{displaymath}
	
(2) We first prove that \begin{equation}\tag{$\ast$}\label{hulldistance}
	\forall I,J \in \lat(R) \colon \quad I \subseteq J \ \Longrightarrow \ d_{\delta}(\Gamma_{S}(I),\Gamma_{S}(J)) \, \leq \, \dim_{\cent(R)}(S)\cdot d_{\delta}(I,J) .
\end{equation} To this end, let $I,J \in \lat(R)$ with $I \subseteq J$. Since the lattice $\lat(R)$ is complemented and modular, thus relatively complemented by Lemma~\ref{lemma:relatively.complemented}, there exists $I' \in \lat(R)$ such that $J = I+I'$ and $I \cap I' = \{ 0 \}$. Hence, $\delta(J) = \delta(I) + \delta(I')$ by Lemma~\ref{lemma:from.rank.to.dimension} and therefore \begin{equation}\tag{$\ast \ast$}\label{hulldistance2}
	\delta(I') \, = \, \delta(J) - \delta(I) \, = \, d_{\delta}(I,J) . 
\end{equation} Since $\Gamma_{S}(I) \subseteq \Gamma_{S}(J)$ by Lemma~\ref{lemma:hull.2}(3), we conclude that \begin{align*}
	d_{\delta}(\Gamma_{S}(I), \Gamma_{S}(J)) \, & = \, \delta(\Gamma_{S}(J)) - \delta(\Gamma_{S}(I)) \, \stackrel{\ref{lemma:hull.2}(2)}{=} \, \delta(\Gamma_{S}(I) + \Gamma_{S}(I')) - \delta(\Gamma_{S}(I)) \\
		&\stackrel{\ref{lemma:from.rank.to.dimension}}{\leq} \, \delta(\Gamma_{S}(I')) \, \stackrel{(1)}{\leq} \, \dim_{\cent(R)}(S)\cdot \delta(I') \, \stackrel{\eqref{hulldistance2}}{=} \, \dim_{\cent(R)} (S)\cdot d_{\delta}(I,J) .
\end{align*} This proves~\eqref{hulldistance}. Consequently, for all $I,J \in \lat(R)$, \begin{align*}
	d_{\delta}(\Gamma_{S}(I),\Gamma_{S}(J)) \, & \stackrel{\ref{lemma:from.rank.to.dimension}+\ref{lemma:birkhoff}}{\leq} \, d_{\delta}(\Gamma_{S}(I),\Gamma_{S}(I+J)) + d_{\delta}(\Gamma_{S}(I+J),\Gamma_{S}(J)) \\
		&\stackrel{\eqref{hulldistance}}{\leq} \, \dim_{\cent(R)}(S) \cdot d_{\delta}(I,I+J) + \dim_{\cent(R)}(S) \cdot d_{\delta}(I+J,J) \\
		&= \, \dim_{\cent(R)}(S) \cdot \left(d_{\delta}(I,I+J) + d_{\delta}(I+J,J)\right) \\
		&= \, \dim_{\cent(R)}(S) \cdot \left(2\delta(I+J) - \delta(I) - \delta(J)\right) \\
		& \stackrel{\ref{lemma:from.rank.to.dimension}+\ref{lemma:birkhoff}(3)}{=} \, \dim_{\cent(R)}(S)\cdot d_{\delta}(I,J) . \qedhere
\end{align*} \end{proof}

We arrive at the desired continuous triangularization theorem.

\begin{thm}[continuous triangularization]\label{theorem:invariant.flags} Let $R$ be a non-discrete irreducible, continuous ring, and let $S$ be a subring of $R$ containing the center $\cent(R)$. If $\dim_{\cent(R)}(S) < \infty$, then \begin{displaymath}
	\exists F \in \Flagmax (R) \colon \quad S \, \subseteq \, R_{F} .
\end{displaymath} \end{thm}

\begin{proof} Note that $\delta \defeq \delta_{\lat(R)} = \delta_{\rho_{R}}$ by Theorem~\ref{theorem:continuous.rank.rings}(2). Moreover, we observe that $\Flagmax (R) \ne \emptyset$ by the Hausdorff maximal principle. Pick any $F \in \Flagmax (R)$. Due to Corollary~\ref{corollary:maximal.flags} and strict positivity of $\delta$, the map $\delta\vert_{F} \colon F \to [0,1]$ is bijective. Now, suppose that $\dim_{\cent(R)}(S) < \infty$ and consider \begin{displaymath}
	F' \, \defeq \, \Gamma_{S}(F) \, = \, \{ \Gamma_{S}(I) \mid I \in F \} .
\end{displaymath} Due to Lemma~\ref{lemma:hull.2}(3) and $F$ being a flag over $R$, it follows that $F'$ constitutes a flag over $R$, too. Furthermore, we observe that the function \begin{displaymath}
	f \colon \, [0,1] \, \longrightarrow \, [0,1], \quad t \, \longmapsto \, \delta\left(\Gamma_{S}\left((\delta\vert_{F})^{-1}(t)\right)\right)
\end{displaymath} is continuous: indeed, for any $s,t \in [0,1]$, \begin{align*}
	\vert f(s) - f(t) \vert \, & = \, \left\lvert \delta\left(\Gamma_{S}\left((\delta\vert_{F})^{-1}(s)\right)\right) - \delta\left(\Gamma_{S}\left((\delta\vert_{F})^{-1}(t)\right)\right)\right\rvert \\
		& \stackrel{F' \in \Flag (R)}{=} \, d_{\delta}\left(\Gamma_{S}\left((\delta\vert_{F})^{-1}(s)\right),\Gamma_{S}\left((\delta\vert_{F})^{-1}(t)\right)\right) \\
		& \stackrel{\ref{lemma:hull.3}(2)}{\leq} \, \dim_{\cent(R)}(S) \cdot d_{\delta}\left((\delta\vert_{F})^{-1}(s),(\delta\vert_{F})^{-1}(t)\right) \\
		& \stackrel{F \in \Flag (R)}{=} \, \dim_{\cent(R)}(S) \cdot \left\lvert \delta \left((\delta\vert_{F})^{-1}(s)\right) - \delta\left((\delta\vert_{F})^{-1}(t)\right) \right\rvert \\
		& = \, \dim_{\cent(R)}(S) \cdot \vert s-t \vert .
\end{align*} Since moreover \begin{align*}
	f(0) \, &= \, \delta\left(\Gamma_{S}\left((\delta\vert_{F})^{-1}(0)\right)\right) \, = \, \delta(\Gamma_{S}(\{ 0_{R} \})) \, = \, \delta (\{ 0_{R} \}) \, = \, 0 , \\
		f(1) \, &= \, \delta\left(\Gamma_{S}\left((\delta\vert_{F})^{-1}(1)\right)\right) \, = \, \delta(\Gamma_{S}(R)) \, = \, \delta (R) \, = \, 1 ,
\end{align*} the intermediate value theorem asserts that $f$ is surjective. Hence, for every $t \in [0,1]$, there exists $s \in [0,1]$ with $f(s) = t$, that is, \begin{displaymath}
	t \, = \, \delta\left(\Gamma_{S}\left((\delta\vert_{F})^{-1}(s)\right)\right) \, \in \, \delta (F') .
\end{displaymath} Thus, $\delta\vert_{F'} \colon F' \to [0,1]$ is surjective, wherefore $F' \in \Flagmax (R)$ according to Corollary~\ref{corollary:maximal.flags}. Finally, $S \subseteq R_{F'}$ by Lemma~\ref{lemma:hull.2}(1). \end{proof}

Theorem~\ref{theorem:invariant.flags} is to be compared with the classical triangularization theorem of linear algebra, which asserts that an endomorphism of a finite-dimensional vector space is triangularizable if and only if its characteristic polynomial splits into linear factors. Since the endomorphism ring of any finite-dimensional vector space (equivalently, any finite-dimensional matrix ring over a field) constitutes a discrete irreducible, continuous ring (see~\cite[IX.2, Satz~2.2, p.~185]{MaedaBook}), the non-discreteness assumption in Theorem~\ref{theorem:invariant.flags} is clearly essential: while any square matrix over a field $K$ is algebraic over $K$ by the Cayley--Hamilton theorem, every field that is not algebraically closed admits a non-triangularizable square matrix (such as the companion matrix of any irreducible, non-linear, monic polynomial).

\begin{cor}\label{corollary:invariant.nests} Let $R$ be a non-discrete irreducible, continuous ring. \begin{itemize}
	\item[$(1)$] Let $S$ be a subring of $R$ containing $\cent(R)$. If $\dim_{\cent(R)}(S) < \infty$, then \begin{displaymath}
			\qquad \exists E \in \Nestmax (R) \colon \quad S \, \subseteq \, R_{E} .
		\end{displaymath}
	\item[$(2)$] Let $G \subseteq \GL(R)$ and let $S$ be the subring of $R$ generated by $G \cup \cent(R)$. If $\dim_{\cent(R)}(S) < \infty$, then \begin{displaymath}
			\qquad \exists E \in \Nestmax (R) \colon \quad G \, \subseteq \, \GL(R_{E}) .
		\end{displaymath}
\end{itemize} \end{cor}

\begin{proof} (1) Suppose that $\dim_{\cent(R)}(S) < \infty$. According to Theorem~\ref{theorem:invariant.flags}, there exists $F \in \Flagmax (R)$ such that $S \subseteq R_{F}$. By Theorem~\ref{theorem:continuous.rank.rings}(2) and Theorem~\ref{theorem:maximal.nests}, we find $E \in \Nestmax (R)$ with $F = \{ eR \mid e \in E\}$, whence \begin{displaymath}
	S \, \subseteq \, R_{F} \, \stackrel{\ref{proposition:nest.subring}(2)}{=} \, R_{E} .
\end{displaymath}

(2) Note that $R$ is directly finite by Theorem~\ref{theorem:continuous.rank.rings}(2) and Lemma~\ref{lemma:rank.functions.vs.units}(2). Now, if $\dim_{\cent(R)}(S) < \infty$, then by~(2) there is $E \in \Nestmax (R)$ with $S \subseteq R_{E}$, thus \begin{displaymath}
	G \, \subseteq \, R_{E} \cap \GL(R) \, \stackrel{\ref{proposition:units.in.nest.stabilizers}}{=} \, \GL(R_{E}) . \qedhere
\end{displaymath} \end{proof}

Combining the above with Proposition~\ref{proposition:algebraic.units}, we arrive at the following.

\begin{cor}\label{corollary:algebraic.units} If $R$ is a non-discrete irreducible, continuous ring, then \begin{displaymath}
	\overline{\bigcup\nolimits_{E \in \Nestmax (R)} R_{E}} \, = \, R , \qquad \overline{\bigcup\nolimits_{E \in \Nestmax (R)} \GL(R_{E})} \, = \, \GL(R)
\end{displaymath} with respect to the metric $d_{R}$. \end{cor}

\begin{proof} The first assertion follows by Corollary~\ref{corollary:neumann}, Remark~\ref{remark:algebraicity}, and Corollary~\ref{corollary:invariant.nests}(1). The second is a consequence of Proposition~\ref{proposition:algebraic.units}, Remark~\ref{remark:algebraicity}, and Corollary~\ref{corollary:invariant.nests}(2) on one hand, and Lemma~\ref{lemma:rank.functions.vs.units}(4) on the other. \end{proof}

\section{Nest envelopes}\label{section:nest.envelopes}

In view of the abundance of maximal nests in non-discrete irreducible, continuous rings illustrated by Corollary~\ref{corollary:algebraic.units}, it seems natural to study the subgroup structure of the unit groups of the corresponding stabilizer rings. As will be shown in the course of this section, if $E$ is any nest in a unital ring $R$, then the subgroup lattice of $\GL(R_{E})$ admits a natural closure operator (Definition~\ref{definition:jordan}), which is inspired by the classical (multiplicative) Jordan--Chevalley decomposition. We will examine the algebraic characteristics of this construction relevant to the dynamical considerations of the subsequent Section~\ref{section:stability}. We start off with a few preparatory notes.

\begin{remark}\label{remark:ring.homs} Let $R$ be a unital ring. \begin{itemize}
	\item[$(1)$] Let $n \in \N$. If $e_{1},\ldots,e_{n} \in \Ed(R)$ are pairwise orthogonal idempotents with $\sum_{i=1}^{n} e_{i} = 1$, then \begin{displaymath}
		\qquad \prod\nolimits_{i=1}^{n} e_{i}Re_{i} \, \longrightarrow \, R, \quad (a_{1},\ldots,a_{n}) \, \longmapsto \, a_{1}+\ldots +a_{n}
	\end{displaymath} is a unital ring homomorphism.
	\item[$(2)$] Let $S$ be a unital ring and let $\phi \colon R \to S$ be a unital ring homomorphism. Then $\tilde{\phi} \colon \GL(R) \to \GL(S), \, a \mapsto \phi(a)$ is a well-defined group homomorphism and $\Ker \tilde{\phi} = \GL(R) \cap (1+\Ker\phi)$.
\end{itemize} \end{remark}

\begin{lem}\label{lemma:induced.units} Let $R$ be a unital ring and let $e \in \Ed(R)$. Then \begin{displaymath}
	\GL(eRe) \, \longrightarrow \, \GL(R), \quad a \, \longmapsto \, a +1-e
\end{displaymath} is a group homomorphism. \end{lem}

\begin{proof} Note that $e\perp (1-e) \in \Ed(R)$ by Remark~\ref{remark:idempotent.difference}. Now, let us consider the ring $S \defeq eRe \times (1-e)R(1-e)$. According to Remark~\ref{remark:ring.homs}, \begin{displaymath}
	\phi \colon \, \GL(S) \, \longrightarrow \, \GL(R), \quad (a,b) \, \longmapsto \, a + b
\end{displaymath} is a well-defined group homomorphism. Furthermore, since \begin{displaymath}
	\GL(S) \, = \, \GL(eRe) \times \GL((1-e)R(1-e)) ,
\end{displaymath} we see that $\psi \colon \GL(eRe) \to \GL(S), \, a \mapsto (a,1-e)$ is a well-defined group homomorphism, too. Thus, $\phi \circ \psi \colon \GL(eRe) \to \GL(R)$ is a homomorphism. \end{proof}

Given any nest $E$ in a unital ring $R$, the closure operator to be devised in Definition~\ref{definition:jordan} will involve projections along interval partitions of $E$. We record the details in the following definition.

\begin{definition} Let $R$ be a unital ring. \begin{itemize}
	\item[$(1)$] If $n \in \N$ and $e = (e_{0},\ldots,e_{n}) \in \Ed(R)^{n+1}$ with $0 = e_{0} \leq \ldots \leq e_{n} = 1$, then \begin{displaymath}
			\qquad \bar{e}_{i} \, \defeq \, e_{i}-e_{i-1} \, \stackrel{\ref{remark:idempotent.difference}}{\in} \, \Ed(R) \qquad (i \in \{ 1,\ldots,n\}) .
		\end{displaymath}
	\item[$(2)$] Let $E \in \Nest (R)$. Then we define \begin{displaymath}
			\qquad \Diff(E) \, \defeq \, \{ e-f \mid e,f \in E\cup\{ 0,1\}, \, f \leq e \} .
		\end{displaymath} Furthermore, we let \begin{displaymath}
			\qquad \I_{n}(E) \, \defeq \, \left. \! \left\{ (e_{0},\ldots,e_{n}) \in (E \cup\{ 0,1 \})^{n+1} \, \right\vert 0 = e_{0} \leq \ldots \leq e_{n} = 1 \right\} 
		\end{displaymath} for every $n \in \N$, as well as $\I(E) \defeq \bigcup\nolimits_{n \in \N } \I_{n}(E)$.
\end{itemize} \end{definition}

\begin{remark}\label{remark:nest.rings} Let $R$ be a unital ring. \begin{itemize}
	\item[$(1)$] Let $n \in \N$ and $e = (e_{0},\ldots,e_{n}) \in \Ed(R)^{n+1}$ with $0 = e_{0} \leq \ldots \leq e_{n} = 1$. Then $\bar{e}_{1},\ldots,\bar{e}_{n}$ are pairwise orthogonal and $\sum\nolimits_{i=1}^{n} \bar{e}_{i} = 1$. Indeed, \begin{align*}
			\qquad \bar{e}_{i}\bar{e}_{j} \, &= \ (e_{i} -e_{i-1})(e_{j} -e_{j-1}) \, = \, e_{i}-e_{i} - e_{i-1}+e_{i-1} \, = \, 0 \\
			& = \, e_{i}-e_{i-1} - e_{i}+e_{i-1} \, = \, (e_{j} -e_{j-1})(e_{i} -e_{i-1}) \, = \, \bar{e}_{j}\bar{e}_{i}
	\end{align*} for any two elements $i,j \in \{ 1,\ldots,n \}$ with $i<j$, and \begin{displaymath}
		\qquad \sum\nolimits_{i=1}^{n} \bar{e}_{i} \, = \, \sum\nolimits_{i=1}^{n} e_{i} -e_{i-1} \, = \, e_{n}-e_{0} \, = \, 1-0 \, = \, 1 .
	\end{displaymath}
	\item[$(2)$] Let $E \in \Nest (R)$. Then $\Diff(E) \subseteq \Ed(R_{E})$ due to Proposition~\ref{proposition:nest.subring}(2)+(3) and Remark~\ref{remark:idempotent.difference}. Hence, if $n \in \N$ and $e \in \I_{n}(E)$, then $\bar{e}_{1},\ldots,\bar{e}_{n} \in \Ed(R_{E})$.
\end{itemize} \end{remark}

We now get to the Jordan--Chevalley-type decomposition along interval partitions.

\begin{lem}\label{lemma:idempotent.difference.1} Let $R$ be a unital ring and let $E$ be a nest in $R$. \begin{itemize}
	\item[$(1)$] Let $e \in \Diff(E)$. Then $R_{E} \to eR_{E}e, \, a \mapsto eae$ is a unital ring homomorphism. Furthermore, $\GL(eR_{E}e) = e\GL(R_{E})e$.
	\item[$(2)$] Let $n \in \N$ and $e \in \I_{n}(E)$. Then \begin{displaymath}
					\qquad \pi_{E,e} \colon \, R_{E} \, \longrightarrow \, \prod\nolimits_{i=1}^{n} \bar{e}_{i}R_{E}\bar{e}_{i}, \quad a \, \longmapsto \, (\bar{e}_{1}a\bar{e}_{1},\ldots,\bar{e}_{n}a\bar{e}_{n})
				\end{displaymath} and \begin{displaymath}
					\qquad \iota_{E,e} \colon \, \prod\nolimits_{i=1}^{n} \bar{e}_{i}R_{E}\bar{e}_{i} \, \longrightarrow \, R_{E}, \quad (a_{1},\ldots,a_{n}) \, \longmapsto \, a_{1}+\ldots +a_{n}
			\end{displaymath} are unital ring homomorphisms with $\pi_{E,e} \circ \iota_{E,e} = \id_{\prod\nolimits_{i=1}^{n} \bar{e}_{i}R_{E}\bar{e}_{i}}$.
\end{itemize} \end{lem}

\begin{proof} (1) According to Remark~\ref{remark:nest.rings}(2), we have $e \in \Ed(R_{E})$. Clearly, $e1e = e$ as $e \in \Ed(R)$, and $e(a+b)e = eae + ebe$ for all $a,b \in R_{E}$. Since $e \in \Diff(E)$, there exist $e_{0},e_{1} \in E \cup\{ 0,1\}$ such that $e_{0} \leq e_{1}$ and $e = e_{1}-e_{0}$, whence \begin{align*}
	eabe \, &= \, (e_{1}-e_{0})ab(e_{1}-e_{0}) \, = \, e_{1}(1-e_{0})abe_{1}(1-e_{0}) \\
		& \stackrel{\ref{lemma:ringrose}}{=} \, e_{1}(1-e_{0})a(1-e_{0})e_{1}be_{1}(1-e_{0}) \, = \, eaebe \, \stackrel{e \in \Ed(R)}{=} \, eaeebe 
\end{align*} for all $a,b \in R_{E \cup\{ 0,1\}} = R_{E}$. Therefore, $R_{E} \to eR_{E}e, \, a \mapsto eae$ is a unital ring homomorphism, which by Remark~\ref{remark:ring.homs}(2) entails that $e\GL(R_{E})e \subseteq \GL(eR_{E}e)$. Conversely, if $a \in \GL(eR_{E}e)$, then Lemma~\ref{lemma:induced.units} asserts that \begin{displaymath}
	a \, = \, e(a+1-e)e \, \in \, e\GL(R_{E})e .
\end{displaymath} Thus, $\GL(eR_{E}e) = e\GL(R_{E})e$ as desired.

(2) We note that $\bar{e}_{1},\ldots,\bar{e}_{n} \in \Ed(R_{E})$ by Remark~\ref{remark:nest.rings}(2). Let $S \defeq \prod\nolimits_{i=1}^{n} \bar{e}_{i}R_{E}\bar{e}_{i}$. It follows from (1) that $\pi_{E,e} \colon R_{E} \to S$ is a unital ring homomorphism. Combining Remark~\ref{remark:nest.rings}(1) with Remark~\ref{remark:ring.homs}(1), we conclude that $\iota_{E,e} \colon S \to R$ is a unital ring homomorphism, too. Finally, thanks to $\bar{e}_{1},\ldots,\bar{e}_{n}$ being pairwise orthogonal by Remark~\ref{remark:nest.rings}(1), if $(a_{1},\ldots,a_{n}) \in S$, then \begin{displaymath}
	\bar{e}_{i}(a_{1}+\ldots +a_{n})\bar{e}_{i} \, = \, \bar{e}_{i}a_{1}\bar{e}_{i}+\ldots +\bar{e}_{i}a_{n}\bar{e}_{i} \, = \, a_{i}
\end{displaymath} for each $i \in \{1,\ldots,n\}$, thus \begin{align*}
	\pi_{E,e}(\iota_{E,e}(a_{1},\ldots,a_{n})) \, &= \, \pi_{E,e}(a_{1}+\ldots +a_{n}) \\
	& = \, (\bar{e}_{1}(a_{1}+\ldots +a_{n})\bar{e}_{1},\ldots,\bar{e}_{1}(a_{1}+\ldots +a_{n})\bar{e}_{n}) \, = \, (a_{1},\ldots,a_{n}) .
\end{align*} This shows that $\pi_{E,e} \circ \iota_{E,e} = \id_{S}$. \end{proof}

\begin{lem}\label{lemma:nilpotent.nest.ideal} Let $R$ be a unital ring, let $E$ be a nest in $R$, let $n \in \N$ and $e \in \I_{n}(E)$. Then \begin{equation}\tag{$\ast$}\label{NilpotentNestIdeal}
	\Ker \pi_{E,e} \, = \, \{ a \in R \mid \forall i \in \{ 1,\ldots,n \} \colon \, e_{i-1}ae_{i} = ae_{i} \} .
\end{equation} Furthermore, $(\Ker \pi_{E,e})^{n} = \{ 0 \}$. \end{lem}

\begin{proof} We start off by verifying~\eqref{NilpotentNestIdeal}.
	
($\supseteq $) Let $a \in R$ such that $e_{i-1}ae_{i} = ae_{i}$ for each $i \in \{ 1,\ldots, n\}$. If $f \in E$, then $\{ e_{0},\ldots,e_{n}, f\}$ is a nest in $R$, thus there exists $i \in \{ 1,\ldots,n \}$ with $e_{i-1} \leq f \leq e_{i}$, whence \begin{displaymath}
	faf \, = \, fae_{i}f \, = \, fe_{i-1}ae_{i}f \, = \, e_{i-1}ae_{i}f \, = \, ae_{i}f \, = \, af .
\end{displaymath} Therefore, $a \in R_{E}$. Moreover, for every $i \in \{ 1,\ldots,n \}$, \begin{displaymath}
	\bar{e}_{i}a\bar{e}_{i} \, = \, e_{i}(1-e_{i-1})ae_{i}(1-e_{i-1}) \, = \, e_{i}(1-e_{i-1})e_{i-1}ae_{i}(1-e_{i-1}) \, = \, 0 .
\end{displaymath}

($\subseteq$) If $a \in \Ker \pi_{E,e}$, then \begin{align*}
	ae_{i} \, &= \, a(e_{i}-e_{i-1}) + ae_{i-1} \, = \, ae_{i}(1-e_{i-1}) + ae_{i-1} \, = \, e_{i}ae_{i}(1-e_{i-1}) + ae_{i-1} \\
	&= \, e_{i}a(e_{i}-e_{i-1}) + ae_{i-1} \, = \, (e_{i}-e_{i-1})a(e_{i}-e_{i-1}) + e_{i-1}a(e_{i}-e_{i-1}) + ae_{i-1} \\
	&= \, e_{i-1}a(e_{i}-e_{i-1}) + ae_{i-1} \, = \, e_{i-1}a(e_{i}-e_{i-1}) + e_{i-1}ae_{i-1} \, = \, e_{i-1}ae_{i}
\end{align*} for every $i \in \{ 1,\ldots, n\}$, as desired.
	
Finally, if $a_{1},\ldots,a_{n} \in \Ker \pi_{E,e}$, then~\eqref{NilpotentNestIdeal} entails that \begin{align*}
	a_{1}\cdots a_{n} \, &\stackrel{e_{n}=1}{=} \, a_{1} \cdots a_{n-1}a_{n}e_{n} \, = \, a_{1}\cdots a_{n-2}a_{n-1}e_{n-1}a_{n}e_{n} \\
	& = \, \ldots \, = \, e_{0}a_{1}e_{1} \cdots a_{n-2}e_{n-2}a_{n-1}e_{n-1}a_{n}e_{n} \, \stackrel{e_{0}=0}{=} \, 0 .
\end{align*} This shows that $(\Ker \pi_{E,e})^{n} = \{ 0 \}$. \end{proof}

We arrive at the definition of the announced closure operators.

\begin{definition}\label{definition:jordan} Let $R$ be a unital ring and let $E \in \Nest (R)$. If $G$ is a subgroup of $\GL(R_{E})$, then we define \begin{displaymath}
	[G]_{E,e} \, \defeq \, \pi_{E,e}^{-1}\left(\prod\nolimits_{i=1}^{n} \bar{e}_{i}G\bar{e}_{i}\right) \qquad (e \in \I_{n}(E), \, n \in \N),
\end{displaymath} and we refer to $[G]_{E} \defeq \bigcup\nolimits_{e \in \I(E)} [G]_{E,e}$ as the \emph{$E$-envelope} of $G$. \end{definition}

Preparing the proof of Lemma~\ref{lemma:envelope.basic}, we recall the following basic facts about semi-direct products of topological groups.

\begin{remark}\label{remark:envelope.basic} Let $G$ and $H$ be topological groups. \begin{itemize}
	\item[$(1)$] If $G \times H \to H, \, (g,h) \mapsto {}_{g}h$ is a continuous action of $G$ by automorphisms on $H$, then the respective semi-direct product $H \rtimes G$ constitutes a topological group, as equipped with the product topology and the usual multiplication given by \begin{displaymath}
		\qquad (h,g)\cdot (h',g') \, \defeq \, (h \cdot {}_{g}h',g\cdot g') \qquad (h,h' \in H, \, g,g' \in G) .
	\end{displaymath}
	\item[$(2)$] Let $\pi \colon G \to H$ and $\iota \colon H \to G$ be continuous homomorphisms such that $\pi \circ \iota = \id_{H}$. Then $H \times \Ker \pi \to \Ker \pi, \, (h,g) \mapsto \iota(h)g\iota(h)^{-1}$ is a continuous action of $H$ by automorphisms on $\Ker \pi$ and, with respect to the induced semi-direct product, \begin{align*}
	&\qquad (\Ker \pi) \rtimes H \, \longrightarrow \, G , \quad (g,h) \, \longmapsto \, g\iota(h), \\
	&\qquad G \, \longrightarrow \, (\Ker \pi) \rtimes H , \quad g \, \longmapsto \, \left(g\iota(\pi(g))^{-1},\pi(g)\right)
\end{align*} are mutually inverse continuous homomorphisms, hence isomorphisms of topological groups.
\end{itemize} \end{remark} 

Starting from the proof of the following lemma, we will make use of some basic facts concerning nilpotency in rings. The relevant background material is compiled in the Appendix~\ref{section:levitzki}.

\begin{lem}\label{lemma:envelope.basic} Let $R$ be a unital ring, let $E \in \Nest (R)$ and $G \leq \GL(R_{E})$. Moreover, let $n \in \N$ and $e \in \I_{n}(E)$. Then $[G]_{E,e}$ is a subgroup of $\GL(R_{E})$ and  \begin{align*}
	\sigma \colon \, (1+\Ker \pi_{E,e}) \rtimes \prod\nolimits_{i=1}^{n} \bar{e}_{i}G\bar{e}_{i} \, &\longrightarrow \, [G]_{E,e}, \\
	(h,g_{1},\ldots,g_{n}) \, &\longmapsto \, h\cdot (g_{1}+\ldots +g_{n})
\end{align*} is an isomorphism. If $R$ is regular and  $\rho \colon R \to [0,1]$ is a pseudo-rank function, then $\sigma$ is a homeomorphism with respect to the relative $\rho$-topologies. \end{lem}

\begin{proof} First of all, we note that \begin{displaymath}
	\GL\left( \prod\nolimits_{i=1}^{n} \bar{e}_{i}R_{E}\bar{e}_{i} \right) \, = \, \prod\nolimits_{i=1}^{n} \GL(\bar{e}_{i}R_{E}\bar{e}_{i}) \, \stackrel{\ref{lemma:idempotent.difference.1}(1)}{=} \, \prod\nolimits_{i=1}^{n} \bar{e}_{i}\GL(R_{E})\bar{e}_{i} \, \eqdef \, H .
\end{displaymath} Due to Lemma~\ref{lemma:idempotent.difference.1}(2) and Remark~\ref{remark:ring.homs}(2), the maps \begin{align*}
	\pi \colon \, &\GL(R_{E}) \, \longrightarrow \, H, \quad g \, \longmapsto \, \pi_{E,e}(g), \\
	\iota \colon \, &H \, \longrightarrow \, \GL(R_{E}), \quad g \, \longmapsto \, \iota_{E,e}(g)
\end{align*} are well-defined group homomorphisms satisfying $\pi \circ \iota = \id_{H}$. Since \begin{displaymath}
	\Ker \pi \, \stackrel{\ref{lemma:idempotent.difference.1}(2)+\ref{remark:ring.homs}(2)}{=} \, \GL(R_{E}) \cap (1 + \Ker \pi_{E,e}) \, \stackrel{\ref{lemma:nilpotent.nest.ideal}+\ref{remark:nilpotent.element}(2)}{=} \, 1 + \Ker \pi_{E,e} ,
\end{displaymath} Remark~\ref{remark:envelope.basic}(2) asserts that \begin{align*}
	\tau \colon \, (1+\Ker \pi_{E,e}) \rtimes H \, &\longrightarrow \, \GL(R_{E}), \\
	(h,g_{1},\ldots,g_{n}) \, &\longmapsto \, h\cdot (g_{1}+\ldots +g_{n})
\end{align*} is a well-defined group isomorphism. Now, $H_{0} \defeq \prod\nolimits_{i=1}^{n} \bar{e}_{i}G\bar{e}_{i}$ is a subgroup of $H$ by Lemma~\ref{lemma:idempotent.difference.1}(1) and Remark~\ref{remark:ring.homs}(2), and thus $\iota_{E,e}(H_{0}) = \iota(H_{0})$ is a subgroup of $\GL(R_{E})$. From these observations and the fact that $\Ker \pi_{E,e}$ is a (right) ideal of~$R_{E}$, we infer that \begin{align*}
	[G]_{E,e} \, &= \, \pi_{E,e}^{-1}(H_{0}) \, \stackrel{\ref{lemma:idempotent.difference.1}(2)}{=} \, \iota_{E,e}(H_{0}) + \Ker \pi_{E,e} \, = \, (1 + \Ker \pi_{E,e}) \cdot \iota_{E,e}(H_{0}) \\
		& = \, \tau( (1 + \Ker \pi_{E,e}) \rtimes H_{0}) \, \leq \, \GL(R_{E}) .
\end{align*} In turn, the restriction $\sigma \colon (1+\Ker \pi_{E,e}) \rtimes H_{0} \to [G]_{E,e}, \, z \mapsto \tau(z)$ constitutes a well-defined group isomorphism. Finally, if $R$ is regular and $\rho \colon R \to [0,1]$ is a pseudo-rank function, then $\pi$ and $\iota$ are continuous with respect to the relative $\rho$-topologies by Lemma~\ref{lemma:rank.metric}(2) (see also Remark~\ref{remark:topological.ring}), hence $\tau$ is a homeomorphism by Remark~\ref{remark:envelope.basic}(2), and so is $\sigma$.  \end{proof}

\begin{remark}\label{remark:directed.envelope} Let $R$ be a unital ring and let $E \in \Nest (R)$. For any $e \in \I_{m}(E)$ and $f \in \I_{n}(E)$ with $m,n \in \N$, let \begin{displaymath}
	e \leq f \quad :\Longleftrightarrow \quad \{ e_{i} \mid i \in \{ 0,\ldots,m \} \} \subseteq \{ f_{j} \mid j \in \{ 0,\ldots,n \} \} .
\end{displaymath} Then $(\I(E),{\leq})$ is a directed set \end{remark}

\begin{lem}\label{lemma:envelope} Let $R$ be a unital ring, let $E \in \Nest (R)$, and let $G \leq \GL(R_{E})$. \begin{itemize}
	\item[$(1)$] If $e,f \in \I(E)$ and $e \leq f$, then $[G]_{E,e} \subseteq [G]_{E,f}$.
	\item[$(2)$] $\{ [G]_{E,e} \mid e \in \I(E) \}$ is directed with respect to inclusion.
	\item[$(3)$] $[G]_{E}$ is a subgroup of $\GL(R_{E})$.
\end{itemize} \end{lem}

\begin{proof} (1) Consider any $m,n \in \N$, $e \in \I_{m}(E)$ and $f \in \I_{n}(E)$ with $e\leq f$. Let \begin{displaymath}
	\tau \colon \, \{ 1,\ldots,n\} \, \longrightarrow \, \{1,\ldots,m\}, \quad j \, \longmapsto \, \min \{ i \in \{ 1,\ldots,m\} \mid f_{j} \leq e_{i} \} .
\end{displaymath} Since $e \leq f$, for each $j \in \{ 1,\ldots,n\}$ it follows that $e_{\tau(j)-1} \leq f_{j-1} \leq f_{j} \leq e_{\tau(j)}$ and therefore \begin{equation}\tag{$\ast$}\label{index}
	\begin{split}
		\bar{e}_{\tau(j)}\bar{f}_{j} \, &= \, (e_{\tau(j)}-e_{\tau(j)-1})(f_{j}-f_{j-1}) \, = \, f_{j}-f_{j-1} - e_{\tau(j)-1}+e_{\tau(j)-1} \, = \, \bar{f}_{j} \\
		& = \, f_{j}-f_{j-1} - e_{\tau(j)-1}+e_{\tau(j)-1} \, = \, (f_{j}-f_{j-1})(e_{\tau(j)}-e_{\tau(j)-1}) \, = \, \bar{f}_{j}\bar{e}_{\tau(j)} .
	\end{split}
\end{equation} Hence, if $a \in [G]_{E,e}$, then \begin{displaymath}
	\bar{f}_{j}a\bar{f}_{j} \, \stackrel{\eqref{index}}{=} \, \bar{f}_{j}\bar{e}_{\tau(j)}a\bar{e}_{\tau(j)}\bar{f}_{j} \, \in \, \bar{f}_{j}\bar{e}_{\tau(j)}G\bar{e}_{\tau(j)}\bar{f}_{j} \, \stackrel{\eqref{index}}{=} \, \bar{f}_{j}G\bar{f}_{j} 
\end{displaymath} for all $j \in \{ 1,\ldots,n \}$, i.e., $a \in \pi_{E,f}^{-1}\left(\prod\nolimits_{j=1}^{n} \bar{f}_{j}G\bar{f}_{j}\right) = [G]_{E,f}$. Thus, $[G]_{E,e} \subseteq [G]_{E,f}$.

(2) This follows from Remark~\ref{remark:directed.envelope} and (1).
	
(3) By Lemma~\ref{lemma:envelope.basic}, $\mathscr{G} \defeq \{ [G]_{E,e} \mid e \in \I(E) \}$ is a set of subgroups of $\GL(R_{E})$. As $(\mathscr{G},{\subseteq})$ is directed by~(2), this implies that $[G]_{E} = \bigcup \mathscr{G} \leq \GL(R_{E})$, too. \end{proof}

We proceed to some topological aspects of nest envelopes.

\begin{cor}\label{corollary:amenable.envelope} Let $R$ be a regular ring, let $E \in \Nest (R)$, and let $\rho \colon R \to [0,1]$ be a pseudo-rank function. If a subgroup $G \leq \GL(R_{E})$ is amenable with respect to the relative $\rho$-topology, then so are $[G]_{E,e}$ $(e \in \I(E))$ and $[G]_{E}$. \end{cor}

\begin{proof} Let $G \leq \GL(R_{E})$ be amenable with respect to the relative $\rho$-topology.
	
First, let $n \in \N$ and $e \in \I_{n}(E)$. For each $i \in \{ 1,\ldots,n \}$, the surjection \begin{displaymath}
	G \, \longrightarrow \, \bar{e}_{i}G\bar{e}_{i}, \quad g \, \longmapsto \, \bar{e}_{i}g\bar{e}_{i}
\end{displaymath} is a homomorphism by Lemma~\ref{lemma:idempotent.difference.1}(1) and Remark~\ref{remark:ring.homs}(2), which is continuous with respect to the relative $\rho$-topology according to Remark~\ref{remark:topological.ring}. Hence, by~\cite[Theorem~4.6]{rickert}, amenability of the topological group $G$ implies amenability of the topological groups $\bar{e}_{1}G\bar{e}_{1},\ldots ,\bar{e}_{n}G\bar{e}_{n}$. Moreover, by Lemma~\ref{lemma:nilpotent.nest.ideal} and Lemma~\ref{lemma:nilpotent.3}, the group $1+\Ker \pi_{E,e}$ is nilpotent, thus amenable with respect to the discrete topology (see~\cite[Theorem~12.4(b,\,e)]{Wagon}), therefore amenable with regard to the relative $\rho$-topology. As $[G]_{E,e}$ is isomorphic to the topological group $(1+\Ker \pi_{E,e}) \rtimes \prod\nolimits_{i=1}^{n} \bar{e}_{i}G\bar{e}_{i}$ by Lemma~\ref{lemma:envelope.basic} and the latter is amenable by~\cite[Theorem~4.8]{rickert}, it follows that $[G]_{E,e}$ is amenable with respect to the relative $\rho$-topology, too.
	
By the preceding paragraph, Lemma~\ref{lemma:envelope}(2), and~\cite[Theorem~4.7]{rickert}, the topological group $[G]_{E} = \bigcup_{e \in \I(E)} [G]_{E,e}$ is amenable. \end{proof}

Similarly, nest envelopes preserve certain algebraic properties. Anticipating our next remark, we recall the following well-known fact: if a group $G$ admits a solvable normal subgroup $N \unlhd G$ such that $G/N$ is (locally\footnote{A group G is called \emph{locally solvable} if every finitely generated subgroup of G is solvable.}) solvable, then $G$ is (locally) solvable, too.

\begin{remark}\label{remark:solvable} Let $R$ be a unital ring, let $E \in \Nest (R)$, and let $G \leq \GL(R_{E})$. \begin{itemize}
	\item[$(1)$] Let $n \in \N$ and $e \in \I_{n}(E)$. Due to Lemma~\ref{lemma:nilpotent.nest.ideal} and Lemma~\ref{lemma:nilpotent.3}, the group $1+\Ker \pi_{E,e}$ is nilpotent, thus solvable. Hence, if $G$ is (locally) solvable, then \begin{displaymath}
					 \qquad [G]_{E,e} \, \stackrel{\ref{lemma:envelope.basic}}{\cong} \, (1+\Ker \pi_{E,e}) \rtimes \prod\nolimits_{i=1}^{n} \bar{e}_{i}G\bar{e}_{i}
				\end{displaymath} is (locally) solvable, too.
	\item[$(2)$] If $G$ is locally solvable, then $[G]_{E} = \bigcup_{e \in \I(E)} [G]_{E,e}$ is locally solvable by~(1) and Lemma~\ref{lemma:envelope}(2).
\end{itemize} \end{remark}

We conclude this section by noting that, if $E$ is a nest in a unital ring $R$, then the map $\Sub (\GL(R_{E})) \to \Sub (\GL(R_{E})) , \, G \mapsto [G]_{E}$ is equivariant with respect to the action of $\GL(R_{E})$ by conjugation on the set $\Sub (\GL(R_{E}))$ of all subgroups of $\GL(R_{E})$.

\begin{prop}\label{proposition:conjugation} Suppose that $R$ is a unital ring. Furthermore, let $E \in \Nest (R)$, $G \leq \GL(R_{E})$, and $a \in \GL(R_{E})$. Then $a[G]_{E,e}a^{-1} = [aGa^{-1}]_{E,e}$ for every $e \in \I(E)$. Consequently, $a[G]_{E}a^{-1} = [aGa^{-1}]_{E}$. \end{prop}

\begin{proof} We first show that \begin{equation}\tag{$\ast$}\label{conjugation}
	\forall e \in \I(E) \, \forall H \leq \GL(R_{E}) \, \forall b \in \GL(R_{E}) \colon \qquad b[H]_{E,e}b^{-1}\! \, \subseteq \, [bHb^{-1}]_{E,e} .
\end{equation} So, let $n \in \N$, $e \in \I_{n}(E)$, $H \leq \GL(R_{E})$ and $b \in \GL(R_{E})$. If $h \in [H]_{E,e}$, then \begin{displaymath}
	\bar{e}_{i}bhb^{-1}\bar{e}_{i} \, \stackrel{\ref{lemma:idempotent.difference.1}(1)+\ref{remark:idempotent.difference}}{=} \, \bar{e}_{i}b\bar{e}_{i}h\bar{e}_{i}b^{-1}\bar{e}_{i} \, \in \, \bar{e}_{i}b\bar{e}_{i}H\bar{e}_{i}b^{-1}\bar{e}_{i} \, \stackrel{\ref{lemma:idempotent.difference.1}(1)+\ref{remark:idempotent.difference}}{=} \, \bar{e}_{i}bHb^{-1}\bar{e}_{i}
\end{displaymath} for every $i \in \{1,\ldots,n\}$, that is, $bhb^{-1} \in \pi_{E,e}^{-1}\left(\prod\nolimits_{i=1}^{n} \bar{e}_{i}bHb^{-1}\bar{e}_{i}\right) = [bHb^{-1}]_{E,e}$. Therefore, $b[H]_{E,e}b^{-1} \subseteq [bHb^{-1}]_{E,e}$. This proves~\eqref{conjugation}.

For every $e \in \I(E)$, we now conclude that \begin{align*}
	a[G]_{E,e}a^{-1} \! \, &\stackrel{\eqref{conjugation}}{\subseteq} \, [aGa^{-1}]_{E,e} \, = \, aa^{-1}[aGa^{-1}]_{E,e}aa^{-1} \\
	& \stackrel{\eqref{conjugation}}{\subseteq} \, a[a^{-1}aGa^{-1}a]_{E,e}a^{-1}\! \, = \, a[G]_{E,e}a^{-1}
\end{align*} and hence $a[G]_{E,e}a^{-1} = [aGa^{-1}]_{E,e}$. In turn, \begin{displaymath}
	a[G]_{E}a^{-1} \! \, = \, \bigcup\nolimits_{e \in \I(E)} a[G]_{E,e}a^{-1}\! \, = \, \bigcup\nolimits_{e \in \I(E)} [aGa^{-1}]_{E,e} \, = \, [aGa^{-1}]_{E} . \qedhere
\end{displaymath} \end{proof}

\section{Dynamics of stable subgroups}\label{section:stability}

Combining our results about maximal nests in continuous rings with the concentration techniques provided in Section~\ref{section:concentration.of.invariant.means}, we now exhibit new examples of extremely amenable topological groups (Theorem~\ref{theorem:stable.groups} and Corollaries~\ref{corollary:prefinal}, \ref{corollary:final}, \ref{corollary:levitzki.extremely.amenable}). Such examples will be found as certain topological subgroups of the unit group $\GL(R)$ of an arbitrary non-discrete irreducible, continuous ring $R$, endowed with the topology generated by the metric $d_{R}$ (see Definition~\ref{definition:intrinsic.metric} and Remark~\ref{remark:topological.ring}).

For convenience, we introduce a concept of \emph{stability} (Definition~\ref{definition:stability}), which is based on the following family of homomorphisms.

\begin{lem}\label{lemma:idempotent.difference.2} Let $R$ be a unital ring, $E \in \Nest (R)$, $e \in \Diff(E)$. The following hold. \begin{itemize}
	\item[$(1)$] The map \begin{displaymath}
			\qquad \psi_{e} \colon \, \GL(R_{E}) \, \longrightarrow \, \GL(R_{E}) , \quad a \, \longmapsto \, eae + 1-e
		\end{displaymath} is a group endomorphism.
	\item[$(2)$] If $f \in \Diff(E)$ and $f \leq e$, then $\psi_{e} \circ \psi_{f} = \psi_{f} = \psi_{f} \circ \psi_{e}$. In particular, $\psi_{e}$ is idempotent.
	\item[$(3)$] If $R$ is regular and $\rho \colon R \to [0,1]$ is a pseudo-rank function, then \begin{displaymath}
			\qquad \psi_{e} \colon \, (\GL(R_{E}),d_{\rho}) \, \longrightarrow \, (\GL(R_{E}),d_{\rho})
		\end{displaymath} is $1$-Lipschitz, hence continuous.
\end{itemize} \end{lem}

\begin{proof} (1) This follows from Lemma~\ref{lemma:idempotent.difference.1}(1), Remark~\ref{remark:ring.homs}(2), Lemma~\ref{lemma:induced.units}.  
	
(2) If $f \in \Diff(E)$ and $f \leq e$, then indeed, for every $a \in \GL(R_{E})$, \begin{align*}
	(\psi_{e} \circ \psi_{f})(a) \, &= \, e(faf+1-f)e+1-e \, = \, efafe + e-efe+1-e \\
		& = \, faf+1-f \, = \, \psi_{f}(a) \, = \, faf+1-f \\
		& = \, feaef+f-fef+1-f \, = \, f(eae+1-e)f+1-f \\
		& = \, (\psi_{f} \circ \psi_{e})(a) .
\end{align*}

(3) Suppose that $R$ is regular and $\rho \colon R \to [0,1]$ is a pseudo-rank function. For all $a,b \in \GL(R_{E})$, \begin{displaymath}
	\rho(\psi_{e}(a)-\psi_{e}(b)) \, = \, \rho((eae + 1-e)-(ebe + 1-e)) \, = \, \rho(e(a-b)e) \, \leq \, \rho(a-b) .
\end{displaymath} Thus, $\psi_{e} \colon (\GL(R_{E}),d_{\rho}) \to (\GL(R_{E}),d_{\rho})$ is $1$-Lipschitz, hence continuous. \end{proof}

\begin{definition}\label{definition:stability} Let $R$ be a unital ring, $E \in \Nest (R)$. A subgroup $G \leq \GL(R_{E})$ will be called \emph{$E$-stable} if $\psi_{e}(G) \subseteq G$ for every $e \in E$. \end{definition}

Nest envelopes, as introduced in Definition~\ref{definition:jordan}, provide natural examples of stable subgroups.

\begin{lem}\label{lemma:stable.envelopes} Let $R$ be a unital ring and let $E \in \Nest (R)$. The $E$-envelope of every subgroup of $\GL(R_{E})$ is $E$-stable. \end{lem}

\begin{proof} Let $G \leq \GL(R_{E})$ and $f \in E$. By definition of $[G]_{E}$, it suffices to verify that $\psi_{f}([G]_{E,e}) \subseteq [G]_{E}$ for every $e \in \I(E)$. To this end, let $n \in \N$ and $e \in \I_{n}(E)$. Since $E \cup \{ 0,1 \} \in \Nest (R)$, there exists $m \in \{ 0,\ldots,n-1\}$ such that $e_{m} \leq f \leq e_{m+1}$. It follows that $(e_{0},\ldots,e_{m},f,1) \in \I(E)$, \begin{equation}\tag{$\ast$}\label{increment2}
	\bar{e}_{i}f \, = \, e_{i}f-e_{i-1}f \, = \, e_{i}-e_{i-1} \, = \, \bar{e}_{i} \, = \, e_{i}-e_{i-1} \, = \, fe_{i}-fe_{i-1} \, = \, f\bar{e}_{i} 
\end{equation} for each $i \in \{ 1,\ldots,m\}$, and \begin{equation}\tag{$\ast\ast$}\label{increment}
	\bar{e}_{m+1}f \, = \, e_{m+1}f-e_{m}f \, = \, f-e_{m} \, = \, fe_{m+1}-fe_{m} \, = \, f\bar{e}_{m+1} .
\end{equation} Now, if $a \in \GL(R_{E})$, then \begin{displaymath}
	\bar{e}_{i}\psi_{f}(a)\bar{e}_{i} \, \stackrel{\eqref{increment2}}{=} \, \bar{e}_{i}f(faf+1-f)f\bar{e}_{i} \, = \, \bar{e}_{i}faf\bar{e}_{i} \, \stackrel{\eqref{increment2}}{=} \, \bar{e}_{i}a\bar{e}_{i} 
\end{displaymath} for each $i \in \{ 1,\ldots,m\}$, \begin{align*}
	(f-e_{m})\psi_{f}(a)(f-e_{m}) \, & \stackrel{\eqref{increment}}{=} \, \bar{e}_{m+1}f(faf+1-f)f\bar{e}_{m+1} \\
		& = \, \bar{e}_{m+1}faf\bar{e}_{m+1} \, \stackrel{\eqref{increment}}{=} \, f\bar{e}_{m+1}a\bar{e}_{m+1}f ,
\end{align*} and finally \begin{displaymath}
	(1-f)\psi_{f}(a)(1-f) \, = \, (1-f)(faf+1-f)(1-f) \, \stackrel{\ref{remark:idempotent.difference}}{=} \, 1-f .
\end{displaymath} Hence, for every $a \in [G]_{E,e}$, \begin{align*}
	\pi_{E,(e_{0},\ldots,e_{m},f,1)}(\psi_{f}(a)) \, &= \, (\bar{e}_{1}a\bar{e}_{1}, \ldots, \bar{e}_{m}a\bar{e}_{m},f\bar{e}_{m+1}a\bar{e}_{m+1}f,1-f) \\
		&\in \, \left( \prod\nolimits_{i=1}^{m} \bar{e}_{i}G\bar{e}_{i} \right) \times f\bar{e}_{m+1}G\bar{e}_{m+1}f \times (1-f)G(1-f) \\
		&\stackrel{\eqref{increment}}{=} \, \left( \prod\nolimits_{i=1}^{m} \bar{e}_{i}G\bar{e}_{i} \right) \times (f-e_{m})G(f-e_{m}) \times (1-f)G(1-f)
\end{align*} and therefore $\psi_{f}(a) \in [G]_{E,(e_{0},\ldots,e_{m},f,1)} \subseteq [G]_{E}$. Thus, $\psi_{f}([G]_{E,e}) \subseteq [G]_{E}$. \end{proof}

The main reason for our interest in stability is Proposition~\ref{proposition:stable.groups}, whose proof proceeds by combining the abstract Proposition~\ref{proposition:amenable.folding} with the following concrete observations.

\begin{lem}\label{lemma:folding} Suppose that $R$ is a non-discrete irreducible, continuous ring. Let $E \in \Nestmax (R)$ and define \begin{displaymath}
	e_{t} \, \defeq \, (\rho_{R}\vert_{E})^{-1}(t) \qquad (t \in [0,1]) .
\end{displaymath} Then the following hold.  \begin{itemize}
	\item[$(1)$] $\psi_{e_{0}} \equiv 1_{R}$ and $\psi_{e_{1}} = \id_{\GL(R_{E})}$.
	\item[$(2)$] For any two $s,t \in [0,1]$, \begin{displaymath}
			\qquad \psi_{e_{s}} \circ \psi_{e_{t}} \, = \, \psi_{e_{s \wedge t}} .
		\end{displaymath}
	\item[$(3)$] For all $s,t \in [0,1]$ and $a \in \GL(R_{E})$, \begin{displaymath}
			\qquad d_{R}\left(\psi_{e_{s}}(a), \psi_{e_{t}}(a)\right) \, \leq \, \vert s-t \vert .
		\end{displaymath}
\end{itemize} \end{lem}

\begin{proof} First of all, by Corollary~\ref{corollary:maximal.nests.2}(2), the family $(e_{t})_{t \in [0,1]}$ is a well-defined element of $E^{[0,1]}$.
	
(1) Evidently, $\psi_{e_{0}} = \psi_{0_{R}} \equiv 1_{R}$ and $\psi_{e_{1}} = \psi_{1_{R}} = \id_{\GL(R_{E})}$.
	
(2) This follows from Corollary~\ref{corollary:maximal.nests.2}(2) and Lemma~\ref{lemma:idempotent.difference.2}(2).
	
(3) Let $s,t \in [0,1]$. Without loss of generality, we may and will assume that $s \leq t$. Thus, $e_{s} \leq e_{t}$ by Corollary~\ref{corollary:maximal.nests.2}(2). Hence, for every $a \in \GL(R_{E})$, \begin{align*}
	d_{R}\left(\psi_{e_{s}}(a), \psi_{e_{t}}(a)\right) \, & = \, \rho_{R}((e_{s}ae_{s}+1-e_{s})-(e_{t}ae_{t}+1-e_{t})) \\
		& = \, \rho_{R}(e_{s}ae_{s}-e_{t}ae_{t}+e_{t}-e_{s}) \\
		& \stackrel{a\in R_{E}}{=} \, \rho_{R}(ae_{s}-ae_{t}+e_{t}-e_{s}) \, = \, \rho_{R}((1-a)(e_{t}-e_{s})) \\
		& \leq \, \rho_{R}(e_{t}-e_{s}) \, \stackrel{\ref{lemma:rank.estimates}(2)}{=} \, \rho_{R}(e_{t})-\rho(e_{s}) \, = \, t-s \, = \, \vert s-t \vert .\qedhere
\end{align*} \end{proof}

\begin{prop}\label{proposition:stable.groups} Let $R$ be a non-discrete irreducible, continuous ring and let $E \in \Nestmax (R)$. If a topological subgroup $G \leq \GL(R_{E})$ is amenable and $E$-stable, then $\ell (G,d_{R}) = 0$, thus $G$ is extremely amenable. \end{prop}

\begin{proof} Define $e_{t} \defeq (\rho_{R}\vert_{E})^{-1}(t)$ for each $t \in [0,1]$. If $s,t \in [0,1]$ and $s\leq t$, then $E$-stability of $G$ entails that \begin{displaymath}
	\psi_{e_{s}}(G) \, \stackrel{\ref{lemma:folding}(2)}{=} \, \psi_{e_{t}}\left(\psi_{e_{s}}(G)\right) \, \subseteq \, \psi_{e_{t}}(G) .
\end{displaymath} Note that $d_{R}$ is bi-invariant by Lemma~\ref{lemma:rank.metric}(2) (see also Remark~\ref{remark:topological.ring}). Invoking Lemma~\ref{lemma:idempotent.difference.2}(1)+(3) and Lemma~\ref{lemma:folding}(1)+(3), we conclude that \begin{displaymath}
	G \, \longrightarrow \, G, \quad g \, \longmapsto \, \psi_{e_{t}}(g) \qquad (t \in [0,1])
\end{displaymath} is a family of continuous endomorphisms satisfying the hypotheses of Proposition~\ref{proposition:amenable.folding} with respect to $d_{R}$. Thus, $\ell (G,d_{R}) = 0$ by Proposition~\ref{proposition:amenable.folding} and amenability of $G$. Hence, $G$ is extremely amenable by Corollary~\ref{corollary:extreme.amenability.2}. \end{proof}

Returning to nest envelopes, we deduce the following.

\begin{thm}\label{theorem:stable.groups} Let $R$ be a non-discrete irreducible, continuous ring and let $E \in \Nestmax (R)$. If $G$ is an amenable topological subgroup of $\GL(R_{E})$, then \begin{displaymath}
	\ell ([G]_{E},d_{R}) \, = \, 0 ,
\end{displaymath} hence $[G]_{E}$ is extremely amenable. \end{thm}

\begin{proof} This is a direct consequence of Corollary~\ref{corollary:amenable.envelope}, Lemma~\ref{lemma:stable.envelopes}, and Proposition~\ref{proposition:stable.groups}. \end{proof}

We proceed to some consequences of Theorem~\ref{theorem:stable.groups}.

\begin{cor}\label{corollary:prefinal} Suppose that $R$ is a non-discrete irreducible, continuous ring, and let $E \in \Nestmax(R)$. Every element of $\GL(R_{E})$ is contained in a locally solvable, extremely amenable topological subgroup of $\GL(R_{E})$. \end{cor}

\begin{proof} Let $a \in \GL(R_{E})$. Since the subgroup $G \leq \GL(R_{E})$ generated by $a$ is abelian and hence solvable, the group $[G]_{E}$ is locally solvable according to Remark~\ref{remark:solvable}(2). Furthermore, Theorem~\ref{theorem:stable.groups} asserts that $[G]_{E}$ is extremely amenable with respect to the relative topology inherited from $\GL(R_{E})$. \end{proof}

\begin{cor}\label{corollary:final} Let $R$ be a non-discrete irreducible, continuous ring. Every element of $\GL(R)$ algebraic over $\cent(R)$ is contained in a locally solvable, extremely amenable topological subgroup of $\GL(R)$. \end{cor}

\begin{proof} Let $a \in \GL(R)$ be algebraic over $\cent(R)$. Due to Remark~\ref{remark:algebraicity} and Corollary~\ref{corollary:invariant.nests}(2), there exists $E \in \Nestmax(R)$ such that $a \in \GL(R_{E})$. Thus, the desired conclusion follows by Corollary~\ref{corollary:prefinal}. \end{proof}

Another family of extremely amenable topological groups resulting from Theorem~\ref{theorem:stable.groups} is induced by Levitzki radicals (see Appendix~\ref{section:levitzki}).

\begin{lem}\label{lemma:levitzki.envelope} Let $R$ be a unital ring and let $E \in \Nest (R)$. Then \begin{displaymath}
	[1+\Lev R_{E}]_{E} \, = \, 1+\Lev R_{E} .
\end{displaymath} In particular, $1+\Lev R_{E}$ is $E$-stable. \end{lem}

\begin{proof} Of course, $1+ \Lev R_{E} = [1+ \Lev R_{E}]_{E,(0,1)} \subseteq [1+\Lev R_{E}]_{E}$. In order to prove the converse inclusion, let $g \in [1+\Lev R_{E}]_{E}$. Then there exist $n \in \N$ and $e \in \I_{n}(E)$ such that $g \in [1+ \Lev R_{E}]_{E,e}$. Hence, by Lemma~\ref{lemma:envelope.basic}, there exist $h \in 1+\Ker \pi_{E,e}$ and $g_{1},\ldots,g_{n} \in 1+ \Lev R_{E}$ such that \begin{displaymath}
	g \, = \, h (\bar{e}_{1}g_{1}\bar{e}_{1}+\ldots +\bar{e}_{n}g_{n}\bar{e}_{n}) .
\end{displaymath} Note that $\Ker \pi_{E,e} \subseteq \Lev R_{E}$, since $\Ker \pi_{E,e}$ is a nilpotent, two-sided ideal of $R_{E}$ due to Lemma~\ref{lemma:idempotent.difference.1}(2) and Lemma~\ref{lemma:nilpotent.nest.ideal}. In particular, $h \in 1+ \Lev R_{E}$. Furthermore, for each $i \in \{ 1,\ldots,n \}$, we find $a_{i} \in \Lev R_{E}$ such that $g_{i} = 1 + a_{i}$. As $\bar{e}_{1},\ldots,\bar{e}_{n} \in R_{E}$ by Remark~\ref{remark:nest.rings}(2) and $\Lev R_{E}$ is a two-sided ideal of $R_{E}$, we conclude that $\bar{e}_{1}a_{1}\bar{e}_{1},\ldots ,\bar{e}_{n}a_{n}\bar{e}_{n} \in \Lev R_{E}$. Thus, \begin{align*}
	g \, &= \, h (\bar{e}_{1}g_{1}\bar{e}_{1}+\ldots +\bar{e}_{n}g_{n}\bar{e}_{n}) \, \stackrel{\ref{remark:idempotent.difference}}{=} \, h (\bar{e}_{1}+ \ldots +\bar{e}_{n} + \bar{e}_{1}a_{1}\bar{e}_{1} + \ldots + \bar{e}_{n}a_{n}\bar{e}_{n}) \\
	&\stackrel{\ref{remark:nest.rings}(1)}{=} \, h (1 + \bar{e}_{1}a_{1}\bar{e}_{1} + \ldots + \bar{e}_{n}a_{n}\bar{e}_{n}) \, \in \, (1+ \Lev R_{E})(1+ \Lev R_{E}) \, \subseteq \, 1+ \Lev R_{E} .
\end{align*} This shows that $[1+\Lev R_{E}]_{E} = 1+\Lev R_{E}$, which readily entails $E$-stability of $1+\Lev R_{E}$ by Lemma~\ref{lemma:stable.envelopes}. \end{proof}

\begin{remark}\label{remark:levitzki.stable} Let $R$ be a unital ring and let $E \in \Nest (R)$. The $E$-stability of $1+\Lev R_{E}$ entailed by Lemma~\ref{lemma:levitzki.envelope} may be deduced more directly: if $e \in E$, then $e \in R_{E}$ due to Proposition~\ref{proposition:nest.subring}(3), thus $\Lev R_{E}$ being a two-sided ideal of $R_{E}$ implies that \begin{displaymath}
	\psi_{e}(1+a) \, = \, e(1+a)e +1-e \, = \, e+eae+1-e \, = \, 1+eae \, \in \, 1+\Lev R_{E} 
\end{displaymath} for every $a \in \Lev R_{E}$, that is, $\psi_{e}(1+\Lev R_{E}) \subseteq 1+\Lev R_{E}$. \end{remark}

\begin{cor}\label{corollary:levitzki.extremely.amenable} Suppose that $R$ is a non-discrete irreducible, continuous ring, and let $E \in \Nestmax(R)$. Then \begin{displaymath}
	\ell (1+\Lev R_{E},d_{R}) \, = \, 0 .
\end{displaymath} In particular, $1+\Lev R_{E}$ is extremely amenable. \end{cor}

\begin{proof} This follows from Proposition~\ref{proposition:nilpotent}, Lemma~\ref{lemma:levitzki.envelope}, and Theorem~\ref{theorem:stable.groups}. (Alternatively, perhaps more directly, the statement may be deduced by combining Proposition~\ref{proposition:nilpotent} with Remark~\ref{remark:levitzki.stable} and Proposition~\ref{proposition:stable.groups}.) \end{proof}

The following remark shows that the extremely amenable groups resulting from Theorem~\ref{theorem:stable.groups} (or Corollary~\ref{corollary:levitzki.extremely.amenable}, resp.) are never abelian.

\begin{remark}\label{remark:nontrivial.levitzki} Let $R$ be an irreducible, continuous ring with $\vert \rho_{R}(R) \vert \geq 4$. Let $E \in \Nestmax (R)$. Combining Corollary~\ref{corollary:maximal.nests} with Remark~\ref{remark:range} and Theorem~\ref{theorem:continuous.rank.rings}(2), we find $e \in \I_{4}(E)$ such that $\rho_{R}(\bar{e}_{1}) = \rho_{R}(\bar{e}_{2}) = \rho_{R}(\bar{e}_{3}) \in (0,1)$. Thanks to~\cite[II.XVII, Theorem~17.1(d), p.~224]{VonNeumannBook}, there exist $a \in \bar{e}_{1}R\bar{e}_{2}$, $a' \in \bar{e}_{2}R\bar{e}_{1}$, $b \in \bar{e}_{2}R\bar{e}_{3}$ and $b' \in \bar{e}_{3}R\bar{e}_{2}$ with $\bar{e}_{2} = a'\bar{e}_{1}a$ and $\bar{e}_{3} = b'\bar{e}_{2}b$. We observe that $b'a'\bar{e}_{1}ab = b'\bar{e}_{2}b = \bar{e}_{3} \ne 0$ and hence $ab\ne 0$, while \begin{displaymath}
	ba \, = \, b\bar{e}_{3}\bar{e}_{1}a \, \stackrel{\ref{remark:nest.rings}(1)}{=} \, 0 .
\end{displaymath} This entails that \begin{displaymath}
	(1+a)(1+b) \, = \, 1+a+b+ab \, \ne \, 1+a+b \, = \, 1+a+b+ba \, = \, (1+b)(1+a) .
\end{displaymath} Using Lemma~\ref{lemma:nilpotent.nest.ideal}, one readily checks that $\{ a,b \} \subseteq \Ker \pi_{E,e}$, which implies that $1+\Ker \pi_{E,e}$ is non-abelian. Now, if $G \leq \GL(R_{E})$, then \begin{displaymath}
	1+ \Ker \pi_{E,e} \, \stackrel{\ref{lemma:envelope.basic}}{\subseteq} \, [G]_{E,e} \, \subseteq \, [G]_{E} ,
\end{displaymath} wherefore $[G]_{E}$ is non-abelian, too. \end{remark}

For the remainder of this section, we turn to the phenomenon of dynamical inertness, as defined in the paragraph preceding Corollary~\ref{corollary:inert}.

\begin{lem}\label{lemma:inert} Let $G$ be a topological group. \begin{itemize}
	\item[$(1)$] If $G$ acts continuously on a compact Hausdorff space $X$, then the subset $\{ g \in G \mid \exists x \in X \colon \, gx = x \}$ is closed in $G$.
	\item[$(2)$] If $\bigcup \{ H \leq G \mid H \text{ extremely amenable} \}$ is dense in $G$, then $G$ is inert.
\end{itemize} \end{lem}

\begin{proof} (1) Let $G$ act continuously on a compact Hausdorff space $X$. As $X$ is a Hausdorff space, $D \defeq \{ (x,x) \mid x \in X \}$ is closed in $X \times X$. Since the map $\phi \colon G \times X \to X \times X, \, (g,x) \mapsto (gx,x)$ is continuous, $\phi^{-1}(D)$ is closed in $G \times X$. Furthermore, compactness of $X$ implies that $\pi \colon G \times X \to G, \, (g,x) \mapsto g$ is a closed map~\cite[I.10.2, Corollary~5, p.~103]{bourbaki}. Hence, \begin{displaymath}
	\{ g \in G \mid \exists x \in X \colon \, gx = x \} \, = \, \pi\left(\phi^{-1}(D)\right)
\end{displaymath} is closed in $G$.
	
(2) Suppose that $S \defeq \bigcup \{ H \leq G \mid H \text{ extremely amenable} \}$ is dense in $G$. If $G$ acts continuously on a non-empty compact Hausdorff space $X$, then \begin{displaymath}
	S \, \subseteq \, \{ g \in G \mid \exists x \in X \colon \, gx = x \} \, \eqdef \, T,
\end{displaymath} thus $T$ must be dense in $G$, too, whence $T = G$ by~(1). \end{proof}

\begin{cor}\label{corollary:inert.final} Let $R$ be a non-discrete irreducible, continuous ring. Then the union of its extremely amenable topological subgroups is dense in $\GL(R)$. In particular, the topological group $\GL(R)$ is inert. \end{cor}

\begin{proof} Thanks to Corollary~\ref{corollary:final} and Proposition~\ref{proposition:algebraic.units}, the union of its extremely amenable topological subgroups is dense in $\GL(R)$. Hence, $\GL(R)$ is inert by Lemma~\ref{lemma:inert}(2). \end{proof}

\appendix

\section{Nilpotency and Levitzki radical}\label{section:levitzki}

This appendix is devoted to some well-known facts about nilpotency and Levitzki radicals used in Sections~\ref{section:nest.envelopes}--\ref{section:stability}. We refer to~\cite{JacobsonBook} for background.

Let $R$ be a ring. If $k \in \N_{>0}$ and $N \subseteq R$, then we consider \begin{displaymath}
	N^{k} \, = \, \{ a_{1}\cdots a_{k} \mid a_{1},\ldots,a_{k} \in N \}
\end{displaymath} and we let $N^{(k)}$ denote the additive submonoid of $R$ generated by $N^{k}$. It is easy to see that, if $N$ is a left (right, two-sided) ideal of $R$, then so is $N^{(k)}$ for any $k \in \N_{>0}$. A subset $N \subseteq R$ is said to be \emph{nilpotent} if there exists $k \in \N_{>0}$ such that $N^{k} \subseteq \{ 0 \}$ (equivalently, $N^{(k)} = \{ 0 \}$). An element $a \in R$ is called \emph{nilpotent} if $\{ a \}$ is nilpotent, i.e., there exists $k \in \N_{>0}$ with $a^{k} = 0$. A subset $N \subseteq R$ is said to be \emph{nil} if every element of $N$ is nilpotent. A subring $S \leq R$ is called \emph{locally nilpotent} (or \emph{semi-nilpotent}~\cite{levitzki}) if every finitely generated subring of $S$ is nilpotent.

\begin{remark}\label{remark:nilpotent.element} Let $R$ be a unital ring. \begin{itemize}
	\item[$(1)$] If $a \in R$ and $n \in \N_{>0}$ such that $a^{n} = 0$, then $1-a \in \GL(R)$ and \begin{displaymath}
			\qquad (1-a)^{-1} \, = \, \sum\nolimits_{i=0}^{n-1}a^{i} \, = \, 1+\sum\nolimits_{i=1}^{n-1}a^{i} .
		\end{displaymath}
	\item[$(2)$] If $N$ is a nil subring of $R$, then $1+N$ is a subgroup of $\GL(R)$.
\end{itemize} \end{remark}

\begin{lem}\label{lemma:normal.subgroup} Let $N$ be a nil subring of a unital ring $R$. \begin{itemize}
	\item[$(1)$] If $N$ is a two-sided ideal of $R$, then $1+N \unlhd \GL(R)$.
	\item[$(2)$] If $M$ is a two-sided ideal of $N$, then $1+M \unlhd 1+N$.
	\item[$(3)$] If $k \in \N_{>0}$, then $1+N^{(k)} \unlhd 1+N$.
\end{itemize} \end{lem}

\begin{proof} (1) Suppose that $N$ is a two-sided ideal of $R$. If $a \in N$ and $g \in \GL(R)$, then $g(1+a)g^{-1} = 1+ gag^{-1} \in 1 + N$. Thus, $1+N \unlhd \GL(R)$ as desired. 
	
(2) Denote by $S$ the unital subring of $R$ generated by $N$, i.e., \begin{displaymath}
	S \, \defeq \, \{ k\cdot 1 + a \mid k \in \Z, \, a \in N \} .
\end{displaymath} By Remark~\ref{remark:nilpotent.element}(2), the set $1+N$ constitutes a subgroup of $\GL(S)$. Now, if $M$ is a two-sided ideal of $N$, then $M$ is a two-sided ideal of $S$, too, whence $1+M \unlhd \GL(S)$ by (1), and thus $1+M \unlhd 1+N$.
	
(3) If $k \in \N_{>0}$, then $N^{(k)}$ is a two-sided ideal of $N$, whence the claim follows by~(2). \end{proof}

\begin{lem}\label{lemma:nilpotent} Let $N$ be a nil subring of a unital ring $R$. Then, \begin{displaymath}
	\forall k,\ell \in \N_{>0} \, \forall a \in N^{(k)} \, \forall b \in N^{(\ell)} \colon \quad (1+a)(1+b) \, \in \,  (1+a+b)\left( 1+N^{(k+\ell)} \right) .
\end{displaymath} \end{lem}

\begin{proof} Let $k,\ell \in \N_{>0}$, $a \in N^{(k)}$ and $b \in N^{(\ell)}$. Since $N$ is nil, Remark~\ref{remark:nilpotent.element} asserts that $g \defeq 1+a+b \, \in \, \GL(R)$ with $g^{-1}\! \in 1+N$. Furthermore, we observe that $c \defeq ab \in N^{(k+\ell)}$, whence $g^{-1}c \in (1+N)N^{(k+\ell)} \subseteq N^{(k+\ell)}$. Consequently, \begin{displaymath}
	(1+a)(1+b) \, = \, 1+a+b+ab \, = \, g+c \, = \, g\left(1+g^{-1}c\right) \, \in \, (1+a+b)\left( 1+N^{(k+\ell)} \right) 
\end{displaymath} as claimed. \end{proof}

As usual, by the \emph{center} of a group $G$ we mean its normal subgroup \begin{displaymath}
	\cent(G) \, \defeq \, \{ g \in G \mid \forall h \in G \colon \, gh=hg \} .
\end{displaymath}

\begin{lem}\label{lemma:nilpotent.2} Let $N$ be a nil subring of a unital ring $R$. For every $k \in \N_{>0}$, \begin{displaymath}
	1+N^{(k)} \big{/} \, 1+N^{(k+1)} \, \subseteq \, \cent\left( 1+N \, \big{/} \, 1+N^{(k+1)} \right) .
\end{displaymath} \end{lem}

\begin{proof} Consider any $k \in \N_{>0}$. If $a \in N^{(k)}$, then \begin{align*}
	(1+a)(1+b)\left( 1+N^{(k+1)} \right) \, &\stackrel{\ref{lemma:nilpotent}}{=} \, (1+a+b)\left( 1+N^{(k+1)} \right) \\
		& = \, (1+b+a)\left( 1+N^{(k+1)} \right) \\
		& \stackrel{\ref{lemma:nilpotent}}{=} \, (1+b)(1+a)\left( 1+N^{(k+1)} \right)
\end{align*} for every $b \in N$, whence $1+a \in \cent\left( 1+N \, \big{/} \, 1+N^{(k+1)} \right)$. \end{proof}

Let us recall that a group $G$ is said to be \emph{nilpotent} if there exist $n \in \N$ and normal subgroups $G_{0},\ldots,G_{n} \unlhd G$ such that \begin{displaymath}
	\{ e \} \, = \, G_{0} \, \subseteq \, G_{1} \, \subseteq \, \ldots \, \subseteq \, G_{n-1} \, \subseteq \, G_{n} \, = \, G
\end{displaymath} and $G_{i+1}/G_{i} \subseteq \cent(G/G_{i})$ for each $i \in \{ 0,\ldots,n-1\}$. A group $G$ is called \emph{locally nilpotent} if every finitely generated subgroup of $G$ is nilpotent.

\begin{lem}\label{lemma:nilpotent.3} Let $N$ be a nilpotent subring of a unital ring $R$. Then $1+N$ is a nilpotent group. \end{lem}

\begin{proof} Since $N$ is nilpotent, there exists some $n \in \N_{>0}$ such that $N^{n} = \{ 0 \}$. Let $G \defeq 1+N$. For each $i \in \{ 0,\ldots,n-1 \}$, consider \begin{displaymath}
	G_{i} \, \defeq \, 1+N^{(n-i)} \, \stackrel{\ref{lemma:normal.subgroup}(3)}{\unlhd} \, G .
\end{displaymath} Clearly, $G_{n-1} = G$. As $N^{n} = \{ 0 \}$ and therefore $N^{(n)} = \{ 0 \}$, also $G_{0} = \{ 1 \}$. Thus, \begin{displaymath}
	\{ 1 \} \, = \, G_{0} \, \subseteq \, G_{1} \, \subseteq \, \ldots \, \subseteq \, G_{n-2} \, \subseteq \, G_{n-1} \, = \, G .
\end{displaymath} Moreover, for every $i \in \{ 0,\ldots,n-2\}$, \begin{displaymath}
	G_{i+1}/G_{i} \, = \, 1+N^{(n-i-1)} \big{/} \, 1+N^{(n-i)} \, \stackrel{\ref{lemma:nilpotent.2}}{\subseteq} \, \cent\left( 1+N \, \big{/} \, 1+N^{(n-i)} \right) \, = \, \cent(G/G_{i}) .
\end{displaymath} This shows that $G$ is nilpotent. \end{proof}

Let $R$ be a unital ring. The \emph{Levitzki radical}~\cite{levitzki} of $R$ is defined as \begin{align*}
	\Lev R \, \defeq \, \sum \{ N \mid N \text{ locally nilpotent, two-sided ideal of } R \} .
\end{align*} As established by Levitzki~\cite[Theorem~2]{levitzki} (see also~\cite[Section~VIII.3]{JacobsonBook}), $\Lev R$ constitutes a locally nilpotent, two-sided ideal of $R$ and, moreover, contains every locally nilpotent, left (resp., right) ideal of $R$.

\begin{prop}\label{proposition:nilpotent} If $R$ is a unital ring, then the group $1+\Lev R$ is locally nilpotent, thus amenable. \end{prop}

\begin{proof} Let $R$ be an arbitrary unital ring. Consider any finitely generated subgroup $G \leq 1+\Lev R$. Then we find a finite subset $F \subseteq \Lev R$ such that $1+F$ generates the group $G$. Since $\Lev R$ is a locally nilpotent ring, the subring $S$ of $\Lev R$ generated by $F$ is nilpotent, i.e., there exists $n \in \N_{>0}$ with $S^{n} = \{ 0 \}$. Hence, the group $1+S$ is nilpotent by Lemma~\ref{lemma:nilpotent.3}. As $1+F \subseteq 1+S$ and thus $G \subseteq 1+S$, it follows (see, e.g.,~\cite[Theorem~5.1.3(2)]{bechtell}) that $G$ is nilpotent, too. This shows that $1+\Lev R$ is locally nilpotent. In particular, $1+\Lev R$ is amenable (see, e.g.,~\cite[Theorem~12.4(b,\,e,\,f)]{Wagon}). \end{proof}

\section*{Acknowledgments}
	
The author is deeply indebted to Vladimir Pestov, Jan Pachl, Maxime Gheysens, Christian Rosendal, and Luis Carlos Suarez for their insightful and inspiring comments on earlier versions of this manuscript.


\end{document}